\documentclass[a4paper,fleqn]{elsarticle}

\usepackage[numbers]{natbib}

\usepackage{amsthm}
\usepackage{amsmath}
\usepackage{amssymb}
\usepackage{verbatim}
\usepackage{cleveref}
\usepackage{subfig}
\usepackage{float}
\usepackage{color}
\usepackage[printonlyused]{acronym}

\crefalias{subequation}{equation} 
\crefalias{eqnarray}{equation} 
\crefformat{pluraleq}{(#2#1#3)}

\usepackage{bm}
\usepackage[ruled]{algorithm2e}

\newcommand{\bk}{\boldsymbol{k}}

\newcommand{\bu}{\boldsymbol{u}}

\newcommand{\bx}{\boldsymbol{x}}

\newtheorem{theorem}{Theorem}
\newtheorem{assumption}{Assumption}
\newtheorem{remark}{Remark}
\newtheorem{definition}{Definition}
\newtheorem{corollary}{Corollary}
\newtheorem{lemma}{Lemma}

\def\tsc#1{\csdef{#1}{\textsc{\lowercase{#1}}\xspace}}
\tsc{WGM}
\tsc{QE}
\tsc{EP}
\tsc{PMS}
\tsc{BEC}
\tsc{DE}
\begin{document}
\let\WriteBookmarks\relax
\def\floatpagepagefraction{1}
\def\textpagefraction{.001}
\shorttitle{Moving Sampling Physics-informed Neural Networks}
\shortauthors{Yu Yang et~al.}
%\begin{frontmatter}

\title [mode = title]{Moving Sampling Physics-informed Neural Networks induced by Moving Mesh PDE}                      
\tnotemark[1]

\tnotetext[1]{This research is partially sponsored by the National key R \& D Program of China (No.2022YFE03040002) and the National Natural Science Foundation of China (No.11971020, No.12371434). }

% \tnotetext[2]{The second title footnote which is a longer text matter
%    to fill through the whole text width and overflow into
%    another line in the footnotes area of the first page.}

% %%%%%%%%%%%%%%%%%%%%%%%%%%%%%
        \author[1]{Yu Yang}
        [type=editor,
        auid=000,bioid=1,
        orcid=0009-0005-1428-6745
        ]
        %\fnmark[1,3]   
        %\fnmark[2]
        %\fnmark[1] %    ---->    \fntext[fn1]
        \ead{yangyu1@stu.scu.edu.cn}
        %\ead[URL]{www.sayahna.org}
        
        %\credit{Software, Writing - Original draft preparation}
        
        \address[1]{School of Mathematics, Sichuan University, 610065, Chengdu, China.}

% % %%%%%%%%%%%%%%%%%%%%%%%%%%%%%        
%         \author[2,3]{Helin Gong}
%         % [auid=000,bioid=1,
%                                 %prefix=Doctor,
%                                 %role=Researcher,
%                                 % orcid=0000-0002-4094-6795,
%                                 %]
                                
%         %\cormark[1] %    ---->    \cortext[cor1]
%         \cormark[1] %   ---->    \fntext[fn1]
        
%         \ead{gonghelin@sjtu.edu.cn}
%         %\ead[url]{www.cvr.cc, cvr@sayahna.org}
        
%         % \credit{Conceptualization of this study, Data curation, Methodology, Writing - Original draft preparation}
%         %\credit{Conceptualization of this study, Methodology, Software}
        
%         \address[2]{Paris Elite Institute of Technology, Shanghai Jiao Tong University, 200240, Shanghai, China.}
%         \address[3]{Nuclear Power Institute of China, 610041, Chengdu, China.}

% % %%%%%%%%%%%%%%%%%%%%%%%%%%%%%

% % %%%%%%%%%%%%%%%%%%%%%%%%%%%%%     
        \author[1]{Qihong Yang}[style=chinese]
        %[auid=000,bioid=1,
                                %prefix=Doctor,
                                %role=Researcher,
                                % orcid=0000-0002-4094-6795,
                                %]
                                
        %\cormark[1] %    ---->    \cortext[cor1]
        %\fnmark[1] %   ---->    \fntext[fn1]
        
        \ead{yangqh0808@163.com}

% % %%%%%%%%%%%%%%%%%%%%%%%%%%%%%     
        \author[1]{Yangtao Deng}[style=chinese]
        %[auid=000,bioid=1,
                                %prefix=Doctor,
                                %role=Researcher,
                                % orcid=0000-0002-4094-6795,
                                %]
                                
        %\cormark[1] %    ---->    \cortext[cor1]
        %\fnmark[1] %   ---->    \fntext[fn1]
        
        \ead{ytdeng1998@foxmail.com}

% % %%%%%%%%%%%%%%%%%%%%%%%%%%%%%     
        \author[1]{Qiaolin He}[style=chinese]
        %[auid=000,bioid=1,
                                %prefix=Doctor,
                                %role=Researcher,
                                % orcid=0000-0002-4094-6795,
                                %]
                                
        %\cormark[1] %    ---->    \cortext[cor1]
        \cormark[1] %   ---->    \fntext[fn1]
        
        \ead{qlhejenny@scu.edu.cn}
        %\ead[url]{www.cvr.cc, cvr@sayahna.org}
        
        % \credit{Conceptualization of this study, Data curation, Methodology, Writing - Original draft preparation}
        %\credit{Conceptualization of this study, Methodology, Software}
        
        % \address[1]{School of Mathematics, Sichuan University, 610065, Chengdu, China.}

	%800 Dongchuan Road, Minhang District, 200240, Shanghai, China
	\cortext[cor1]{Corresponding author}

\begin{abstract}
In this work, we propose an end-to-end adaptive sampling framework based on deep neural networks and the moving mesh method  (MMPDE-Net), which can adaptively generate new sampling points by solving the moving mesh PDE. This model focuses on improving the quality of sampling points generation. Moreover, we develop an iterative algorithm based on MMPDE-Net, which makes sampling points distribute more precisely and controllably. Since MMPDE-Net is independent of the deep learning solver, we combine it with physics-informed neural networks (PINN) to propose moving sampling PINN (MS-PINN) and show the error estimate of our method under some assumptions. Finally, we demonstrate the performance improvement of MS-PINN compared to PINN through numerical experiments of four typical examples, which numerically verify the effectiveness of our method.
% This template helps you to create a properly formatted \LaTeX\ manuscript.

% \noindent\texttt{\textbackslash begin{abstract}} \dots 
% \texttt{\textbackslash end{abstract}} and
% \verb+\begin{keyword}+ \verb+...+ \verb+\end{keyword}+ 
% which
% contain the abstract and keywords respectively. 
% Each keyword shall be separated by a \verb+\sep+ command.
\end{abstract}

% \begin{graphicalabstract}
% \includegraphics{figs/cas-grabs.pdf}
% \end{graphicalabstract}

% \begin{highlights}
% \item Research highlights item 1
% \item Research highlights item 2
% \item Research highlights item 3
% \end{highlights}


\maketitle

	% @@@@@@@@@@@@@@@@@@@@@@@@@@
	% @@@@@@@@@@@@@@@@@@@@@@@@@@

 	\section{Introduction}
	\label{sec:intro}

    With the rapid increase in computational resources, solving partial differential equations (PDEs) \cite{renardy2006introduction} using deep learning  has been an emerging point \cite{weinan2021dawning}. 
    % So far, scholars have proposed many deep learning methods in computational science, such as based on strong forms (\cite{PINN}, \cite{DGM}); based on variational forms \cite{DeepRitz}; based on stochastic differential equations (SDEs) \cite{DeepBSDE} and operators learning (\cite{DeepONet}, \cite{FNO}). 
    Currently, many researchers have proposed widely used deep learning solvers based on deep neural networks, such as the Deep Ritz method \cite{DeepRitz}, which solve the variational problems arising from PDEs; the Deep BSDE model \cite{DeepBSDE}, which is developed from stochastic differential equations and performs well at solving high-dimensional problems, and the DeepONet framework \cite{DeepONet}, which is used to learn operators accurately and efficiently from a relatively small dataset. In this article, we use physics-informed neural networks (PINN) \cite{PINN}. In PINN, the governing equations of PDEs, boundary conditions, and related physical constraints are incorporated into the design of the loss function, and an optimization algorithm is used to find the network parameters to minimize the loss function, so that the approximated solution output by the neural networks satisfies the governing equations and constraints.

    In fact, deep learning solvers need to construct and optimize the loss function. When constructing the loss function, a common choice is the loss function under the $L_{2}$ norm. However, in code-level implementations, it is difficult to compute the loss function directly in integral form, so it is usual to discretize the loss function by sampling uniformly distributed points as empirical loss, commonly known as the mean square error loss function. %Ultimately, the optimization algorithm is to minimize the discrete loss function, so the accuracy 
    The minimization of the discrete loss function at the end of the optimization directly affects the accuracy of the approximate solution. %the accuracy of the approximation solution is closely related to the accuracy of the discrete loss function. 
    Then, how to choose the sampling points to write the discrete loss function becomes an important problem. There are several uniform sampling methods \cite{wu2023comprehensive}, which include 1) equispaced uniform point sampling, which samples uniformly from the equispaced uniformly distributed points; 2) uniformly random sampling, which randomly samples the points according to a continuous uniform distribution in the computational domain; 3) Latin hypercube sampling (\cite{stein1987large},\cite{mckay2000comparison}); and 4) Sobol sequence (\cite{sobol1967distribution},\cite{pang2019fpinns}), that is one type of quasi-random low-discrepancy sequence. 
    %目前，许多研究者都基于神经网络提出了广泛使用的深度学习求解器，比如有基于变分形式来专注于求解特征值问题DEEP RITZ，基于随机微分方程而且擅长求解高维问题的DEEP BSDE，基于神经算子来学习解函数的算子的DeepoNet。这篇文章中，我们更加关注物理信息神经网络。在PINN里，将偏微分控制方程，边界条件以及相关物理约束纳入神经网络损失函数的设计中，通过优化算法寻找满足损失函数达到最小值的网络参数，使得神经网络输出的逼近解满足控制方程和约束条件。
    %实际上，不止是PINN，所有的深度学习求解器都需要构建并且优化损失函数。在构建损失函数时，一个常用的选择L2范数下的损失函数。然而在代码层面的实现中，直接求解带有积分的损失函数是很难完成的，因此通常会采样服从均匀分布的采样点作为经验损失来离散积分形式的损失函数，常见的有均方误差损失函数。最终优化算法是对离散损失函数进行最小化，所以逼近解的精度是与离散损失函数的精度密切相关的。那么，如何选择采样点来构建离散损失函数就成为了一个值得研究的问题。

    There are also many works about non-uniform adaptive sampling methods. Lu et al. in \cite{lu2021deepxde} proposed a new residual-based adaptive refinement (RAR) method to improve the training efficiency of PINN. That is, training points are adaptively added where the residual loss is large. In \cite{gu2021selectnet}, Gu et al. introduced a neural network that assigns higher weights for sampling with larger residuals, which can lead to more accurate predicted solution if the neural network is chosen properly.
    %a selection network is introduced to serve as a weight function to assign higher weights for samples with large point-wise residuals, which yields a more accurate approximate solution if the selection network is properly chosen. 
   In \cite{wu2023comprehensive}, the residual-based adaptive distribution (RAD) method is proposed, and its main idea is to construct a probability density function based on residual information and to adopt adaptive sampling according to this density function.
    In \cite{nabian2021efficient}, collocation points are sampled according to a distribution proportional to the loss function in each training iteration. \textcolor{black}{In \cite{daw2022rethinking}, an evolutionary sampling algorithm is used to optimize the sampling points using the mean of the residuals as the criterion.} In \cite{das-pinns}, Tang et al. proposed an adaptive sampling framework for solving PDEs from the perspective of the variance of the residual loss term via KRNet (\cite{tang2020deep},\cite{wan2020vae},\cite{wan2021augmented}). All of these adaptive sampling methods above are from the perspective of the loss function, and we would like to focus more on the performance of the solution function. This reminds us of the mesh movement method in traditional numerical methods.
    
    One obvious characteristic of the moving mesh method is that only the mesh nodes are relocated without changing the mesh topology \cite{huang2010adaptive}, i.e., the number of points of the mesh is fixed before and after the mesh is redistributed. In this paper, we focus on the moving mesh PDE (MMPDE) method \cite{huang1998moving}. The purpose of the MMPDE method is to provide a system of partial differential equations to control the movement of the mesh nodes. By adjusting the monitor function in MMPDE, the mesh can be made centralized or dispersed which depending on the performance of the solution. Compared with the uniform mesh, the number of nodes in the moved mesh is greatly reduced with the same numerical solution accuracy, thus the efficiency is improved greatly \cite{he2009numerical}. In this work, based on the  MMPDE method, we propose the adaptive sampling neural network, MMPDE-Net, which is independent of the deep learning solver. It is essentially a coordinate transformation mapping, which enables the coordinates of the output points to be concentrated in the region where the function has large variations. In \cite{ren-wang}, Ren and Wang proposed an iterative method to generate the mesh. Inspired by this approach, we also develop an iterative algorithm on MMPDE-Net to make the distribution of the sampling points more precise and controllable.

    Since MMPDE-Net is essentially a coordinate transformer independent of solving the physical problem, it can be combined with many deep learning solvers. \textcolor{black}{In this work, we choose the most widely used Physics-informed Neural Networks (PINN).} We propose a deep learning framework, whose main idea is to use the preliminary solution provided by PINN to help MMDPE-Net develop the monitor function, and then accelerate the convergence of the loss function of PINN through the adaptive sampling points obtained by MMPDE-Net. \textcolor{black}{To emphasize the role of MMPDE-Net based on the moving grid method for the redistribution of sampled points, we name this deep learning framework as Moving Sampling PINN (MS-PINN).} Meanwhile,  with some reasonable assumptions we prove that the error bounds of the approximation solution and the true solution are lower than that of the traditional PINN under a certain probability. Finally, we verify the effectiveness of MS-PINN using the numerical experiments.
    %(1) two-dimensional Poisson equation with one peak, (2)two-dimensional Poisson equation with two peaks, (3)one-dimensional Burgers equation, (4)two-dimensional Burgers equation.
    
    In summary, our main contributions in this article are:
		\begin{itemize}
			\item A deep adaptive sampling framework (MMPDE-Net) that adaptively generates new sample points by solving the Moving Mesh PDE is proposed.
			\item A deep learning framework (MS-PINN) is presented. The core idea is to improve the approximate solution of PINN by adaptive sampling through MMPDE-Net.
			\item An error estimate of the proposed MS-PINN is given. The effectiveness is demonstrated by reducing the error bounds between the approximate solution and true solution.
		\end{itemize}

    The rest of the paper is organized as follows. In Section 2, the main idea of PINN is briefly introduced and a typical example is used to underline the importance of adaptive sampling. In Section 3, MMPDE-Net  and the iterative algorithm based on MMPDE-Net are illustrated. The MS-PINN framework is presented and  the error estimates are given in Section 4. 
    Numerical experiments are shown in Section 5 to verify the effectiveness of our method. Finally, conclusions are given in Section 6.

    %Section 5 illustrates the numerical results with four typical problems. Finally, Section 6 concludes and summarizes the paper.

	% @@@@@@@@@@@@@@@@@@@@@@@@@@
	\section{Preliminary Work}
	\label{sec:Preliminary Work}
 \textcolor{black}{
\subsection{Problem formulation}
\label{sec:Problem formulation}
The general form of the problem that we consider is as follows
\begin{equation}\label{eq:gen_PDE}
    \left\{
  \begin{aligned}
 \mathcal{A}[\bu(\bx)] &= f(\bx), \quad \bx \in \Omega, \\
 \mathcal{B}[\bu(\bx)] &= g(\bx), \quad \bx \in \partial\Omega,
  \end{aligned}
    \right.
\end{equation}
where $\Omega \subset \mathbb{R}^d$ is a  bounded, open and connected domain which has a polygonal boundary of $\partial\Omega$, $\mathcal{A}$ and $\mathcal{B}$ are the partial differential operators defined in $\Omega$ and on $\partial \Omega$ respectively, $f(\bx)$ and $g(\bx)$ are the source function and the boundary condition respectively. The solution $\bu(\bx)$ has large variations in subregion of $\Omega$. For time dependent problem, the solution may has singular behaviors with time evolution. The sampling method is very important when we use neural networks to solve this kind of problems.}

%In order to highlight the advantages of MMPDE-Net, the solution $\bu(\bx)$ considered in this paper behaves large variations within a certain subregion of $\Omega$.}

 \subsection{ A brief introduction to PINN}
\label{sec:Brief introduction to PINN}

% We consider a general  form of partial differential equations
% \begin{equation}\label{eq:gen_PDE}
%     \left\{
%   \begin{aligned}
%  \mathcal{A}[\bu(\bx)] &= f(\bx), \quad \bx \in \Omega, \\
%  \mathcal{B}[\bu(\bx)] &= g(\bx), \quad \bx \in \partial\Omega,
%   \end{aligned}
%     \right.
% \end{equation}
% where $\Omega \subset  \mathbb{R}^d$ is a bounded, open and connected domain which has a polygonal boundary $\partial \Omega$, $\mathcal{A}$ is a partial differential operator, $f(\bx)$ is the source function, $\mathcal{B}$ and $g(\bx)$ are the boundary operator and the boundary condition.

% 通常来说，PINN的网络架构是由全连接神经网络组成的。假设全连接神经网络的输入是d维，输出是d_out维，每层的神经元有d_h个，共L层。那么，可以知道PINN的参数量(包含每层的权重和偏差)，共有(d*d_h+d_h)+(d_h*d_h+d_h)*(L-1)+(d_h*d_out+d_out)个

\textcolor{black}{The core idea of PINN is to incorporate the equation itself and its constraints into the loss function. Typically, the network architecture of PINN is a fully connected neural network. Suppose the input of the fully connected neural network is in $d_{in}$ dimension, the output is in $d_{out}$ dimension, and there are $L$ layers and each layer has $d_h$ neurons. Therefore the number of parameters of PINN (including the weights and biases of each layer) is $d_{NN} =(d_{in}*d_h+d_h)+(L-1)*(d_h*d_h+d_h)+(d_h*d_{out}+d_{out})$.}

Suppose the approximate solution output by PINN is $\bu(\bx;\theta)$, $\mathcal{A}[\bu(\bx;\theta)]-f(\bx)\in L^2(\Omega)$ and $\mathcal{B}[\bu(\bx;\theta)]-g(\bx)\in L^2(\partial\Omega)$, the loss function can be written as 

\begin{equation}\label{eq:L2_loss}
    \begin{aligned}
     \mathcal{L}(\bx;\theta) &= \alpha_1\Vert  \mathcal{A}[\bu(\bx;\theta)]-f(\bx) \Vert_{2,\Omega}^2 +  \alpha_2\Vert  \mathcal{B}[\bu(\bx;\theta)]-g(\bx) \Vert_{2,\partial\Omega}^2 \\
     &= \alpha_1\int_{\Omega} |\mathcal{A}[\bu(\bx;\theta)]-f(\bx)|^2 d\bx  + \alpha_2\int_{\partial\Omega} |\mathcal{B}[\bu(\bx;\theta)]-g(\bx)|^2 d\bx \\
     & \triangleq \alpha_1\int_{\Omega} |r(\bx;\theta)|^2 d\bx + \alpha_2\int_{\partial\Omega} |b(\bx;\theta)|^2 d\bx, \\
    \end{aligned}
\end{equation}
\textcolor{black}{where $\theta \in \mathbb{R}^{d_{NN}} $ represents the parameters of the neural networks, usually consisting of weights and biases, and $\mathbb{R}^{d_{NN}}$ is the space determined by the structure of the neural network,} $\alpha_1$ is the weight of the residual loss term and $\alpha_2$ is the weight of the boundary loss term. Typically,  when discretizing the loss function Eq \eqref{eq:L2_loss}, we sample uniformly distributed residual training points $\left\{\bx_i\right\}_{i=1}^{M_r} \subset \Omega$ and boundary training points $\left\{\bx_i\right\}_{i=1}^{M_b}\subset \partial\Omega$, and the empirical loss is written as
\begin{equation}\label{eq:L2_empiricalloss}
     \mathcal{L}_M(\bx;\theta) =   \frac{\alpha_1}{M_{r}} \sum_{i = 1}^{M_{r}} |r(\bx_i;\theta)|^2 +
     \frac{\alpha_2}{M_{b}} \sum_{i = 1}^{M_{b}} |b(\bx_i;\theta)|^2.
\end{equation}

After that, optimization algorithms such as Adam \cite{Adam} and LBFGS \cite{LBFGS} are used to make the value of $\mathcal{L}_M(\bx;\theta)$ in Eq \eqref{eq:L2_empiricalloss} decrease gradually during the training process. When the value of $\mathcal{L}_M(\bx;\theta)$ is small enough, the solution $\bu(\bx;\theta)$ output by the neural networks can be regarded as the approximate solution with a sufficiently small error (\cite{shin2020convergence},\cite{shin2023error}).

 \subsection{A discussion of a certain solution }
\label{sec:An interesting phenomenon}
Consider the following one-dimensional example
\begin{equation}
	\label{eq:1d_Poisson_test}
	\hspace{-0.3cm}
	\begin{array}{r@{}l}
		\left\{
		\begin{aligned}
			& - u_{xx}(\bx) = f(\bx), \quad  \bx \in \Omega, \\
			& u(\bx) = g(\bx), \quad   \bx \in \partial \Omega,\\
		\end{aligned}
		\right.
	\end{array}
\end{equation}
where the computational domain $\Omega = (0,3\pi)$. For a given $f(x)$,  the analytical solution of Eq \eqref{eq:1d_Poisson_test} is 
\begin{equation}
	\label{eq:1d_Poisson_test_soluton}
    u(\bx)=e^{-2(x-4)(x-5)}\sin(kx),
\end{equation}
where $k \in \mathbb{N}$ is a parameter. %From Fig.\ref{fig:Test_Loss_u_k}, it is observed that 
It is well known that the analytical solution $u(\bx)$ vibrates up and down more and more frequently on $[\pi,2\pi]$ as the frequency $k$ becomes larger.

\textcolor{black}{We use PINN to solve this problem and uniformly sample $M_r = 300$ points in $\Omega$ as the residual training set, while uniformly sampling $600$ points as the test set.} For the boundary training points, there are only left and right endpoints $M_b =2$. %Here are some parameter settings for PINN, which
We take 4 hidden layers with 60 neurons per layer and use tanh as the activation function. The Adam optimization method is used and the initial learning rate is 0.0001.
The loss function  is defined  as follows
%as in the form of Eq.\ref{eq:L2_empiricalloss}.
\textcolor{black}{\begin{equation}\label{eq:Test_Loss_empiricalloss}
     \mathcal{L}oss(\bx;\theta) =   \frac{1}{M_{r}} \sum_{i = 1}^{M_{r}} |-\Delta u(\bx_i;\theta) - f(\bx_i)|^2 +
     \frac{1}{M_{b}} \sum_{i = 1}^{M_{b}} |u(\bx_i;\theta) - g(\bx_i)|^2,
\end{equation}}
where the weights $\alpha_1 =1 $ and $\alpha_2 = 1$. The minimum value of the loss function and the error $e_2(u)$ are given in Table \ref{tab:Test_Loss_vark}, \textcolor{black}{where $e_2(u)$ is defined in Eq \eqref{eq:measure} in Section \ref{sec:Symbols and parameter settings}.} It is clear that the values of the loss function and the error become larger as $k$ increases.

%When the number of training epochs (N=20000) and the number of sampling points ($M_r$=300) are fixed, only the frequency k (k=2,4,8,16) varies, the minimum value of the loss function and the error $e(u)$ of the approximation solution are given in Tab.\ref{tab:Test_Loss_vark}, where $e(u)$ is defined in Section \ref{sec:Symbols and parameter settings}. It is clear that
%the values of the loss function and the error become larger as $k$ increases.

\begin{table}[h]
\scriptsize
\centering
% \caption{\textcolor{black}{Variation of the minimum of the loss function and the relative error between the approximate solution and analytical solution for different frequency $k$ with $20000$ training epochs, where $e_2(u)$ is the relative $L_{2}$ error and defined in Eq \eqref{eq:measure}.}}
\caption{Variation of the minimum of the loss function and the relative error between the approximate solution and analytical solution for different frequency $k$ with $20000$ training epochs.}
\setlength{\tabcolsep}{3.mm}{
\begin{tabular}{|c|c|c|c|c|}%{p{1cm}p{2.2cm}p{1.2cm}p{1.2cm}p{1.2cm}p{1.2cm}p{1.2cm}p{1.6cm}}
\hline\noalign{\smallskip}
   k& 2 & 4 & 6 & 8 \\
\hline
Loss  & $1.033 \times 10^{-5}$  &   $8.530 \times 10^{-5}$ & $1.215 \times 10^{-4}$  &   1.244 \\
\hline
$e_2(u)$  & $3.301 \times 10^{-4}$  &   $2.843 \times 10^{-3}$ & $2.284 \times 10^{-3}$  &   $1.201  \times 10^{-1}$\\
\hline
\end{tabular}
}
\label{tab:Test_Loss_vark} 
\end{table}

%Since the complexity of the analytic solution $u(\bx)$ varies in different regions, and it is obvious that the most complex part is in $[\pi,2\pi]$.
Since the behavior of the solution varies in different regions, now we let the numbers of sampling points on $(0,\pi]$ and $[2\pi,3\pi)$ are fixed, i.e., $M_{r1}$ and $M_{r3}$ are fixed,  and only the number of sampling points on $[\pi,2\pi]$ is increased, i.e., $M_{r2}$ changes. 
%We observe the variation of the minimum of the loss function.
%In Fig.\ref{fig:Test_Loss_VarpartM}, we fix the
\textcolor{black}{For frequency $k = 16$, the number of sampling points ($M_r = M_{r1} + M_{r2}+M_{r3}$, $M_{r1} = M_{r3} =100$, $M_{r2} =100,150,250,300$) is varied with 10000 and 40000 epochs.
%we observe the minimum of the loss function when the epochs are 10,000 and 20,000. 
It is easy to observe from Fig \ref{fig:Test_Loss_VarpartM}
%from the Fig.\ref{fig:Test_Loss_VarpartM} 
that in the region with large vibrations, more sampling points in $[\pi, 2\pi]$ tend to make the loss function decrease fast and help the neural networks to learn the solution with small errors.}

\begin{figure}[htbp]
\centering
\subfloat[epochs N=10000]{\includegraphics[width = 0.33\textwidth]{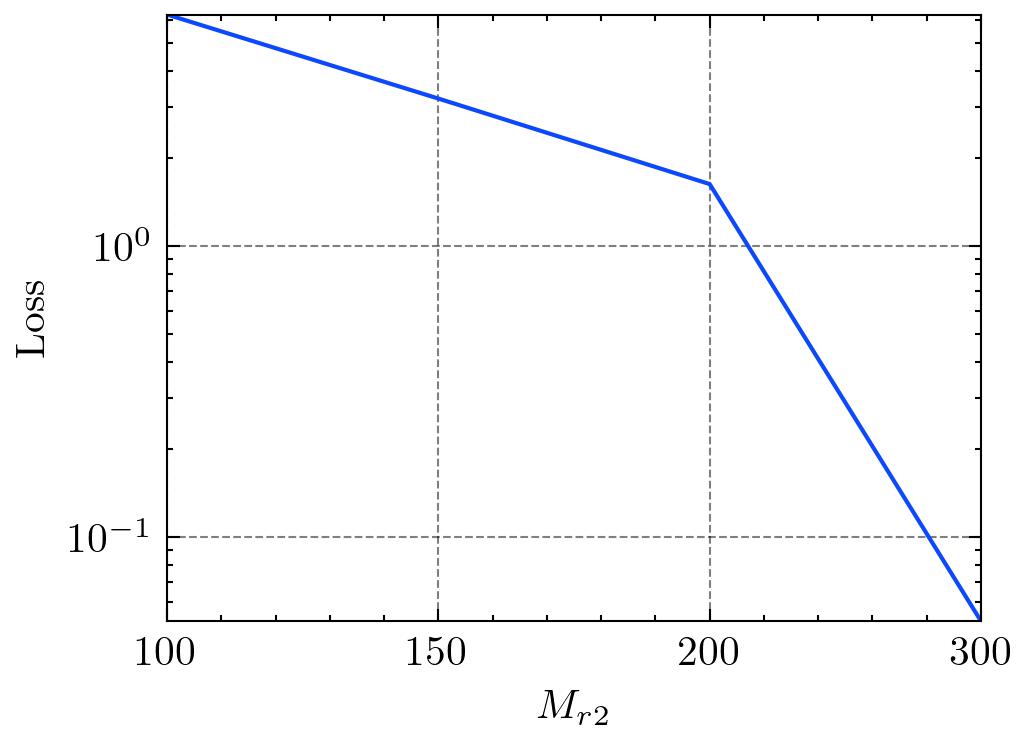}}
\subfloat[epochs N=40000]{\includegraphics[width = 0.33\textwidth]{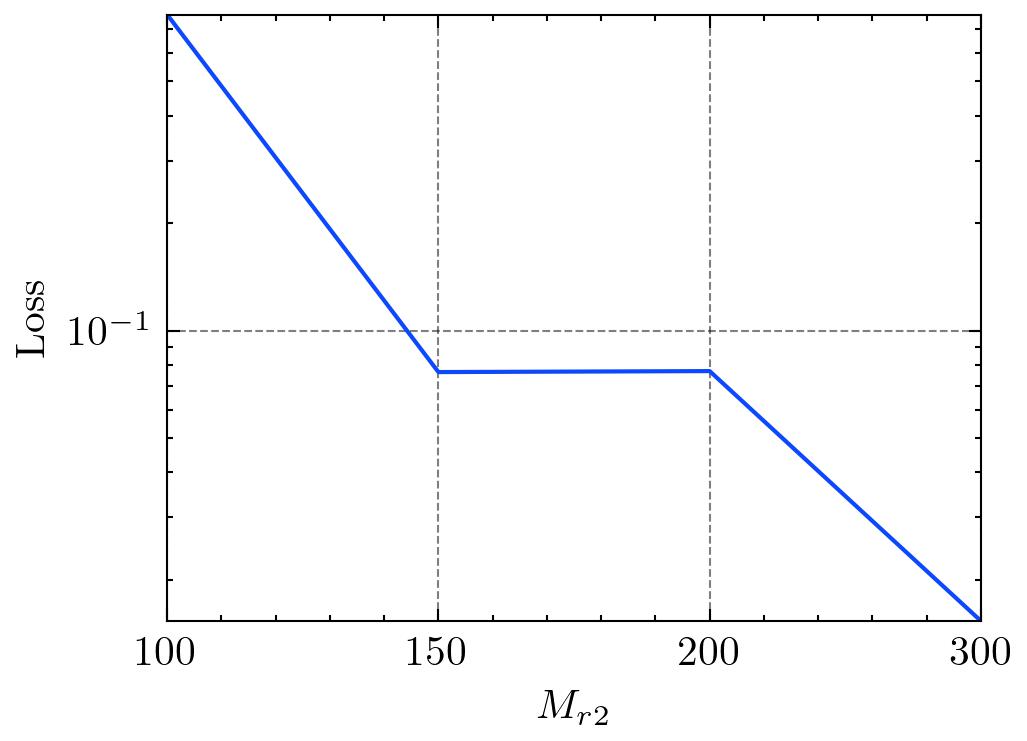}}\\
\subfloat[epochs N=10000]{\includegraphics[width = 0.33\textwidth]{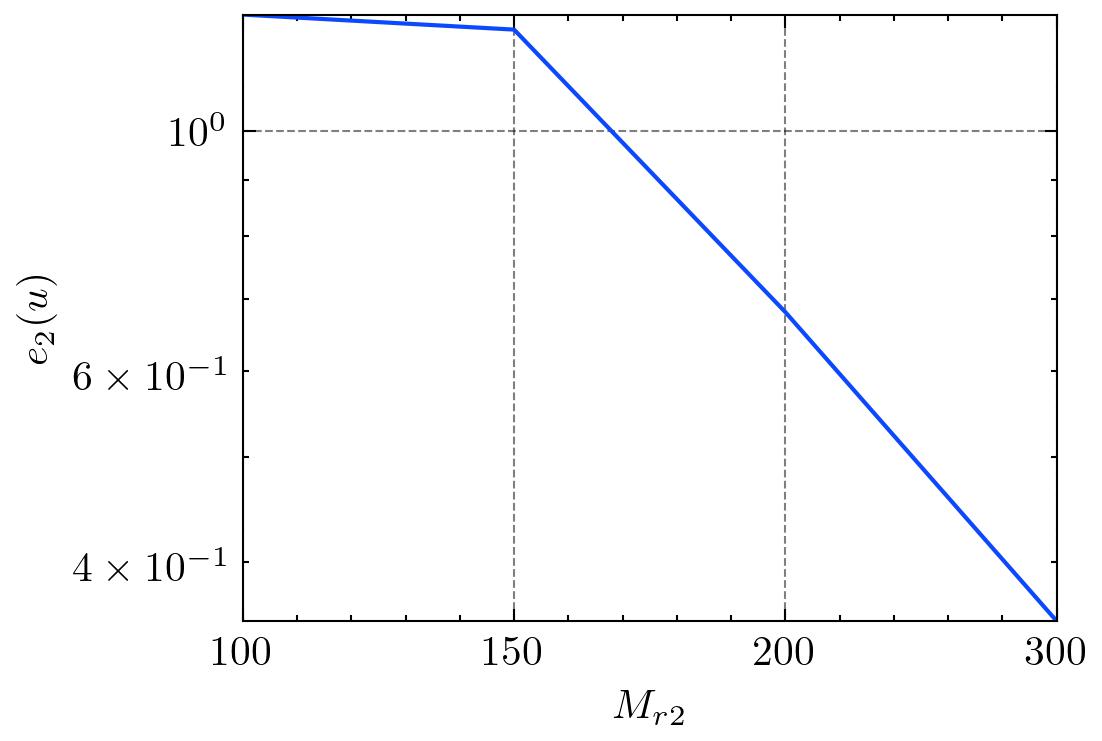}}
\subfloat[epochs N=40000]{\includegraphics[width = 0.33\textwidth]{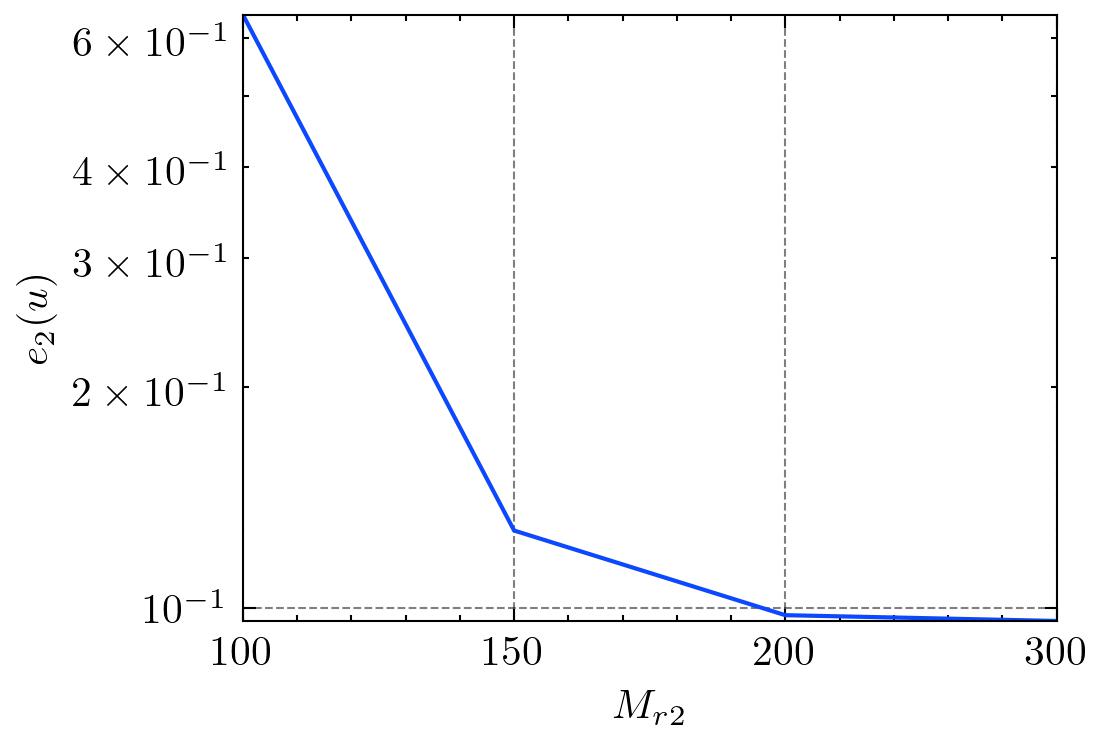}} 
\caption{\textcolor{black}{Variations of the minimum of the loss function and the relative $L_2$ error with different number of sampling points ($M_{r2}$ =100,150,250,300) for $k=16$.}}
\label{fig:Test_Loss_VarpartM}
\end{figure}
%In fact, there are many similarities between the patterns summarized in \cite{frequency-sampling} and the phenomena we found.
In fact, there are similar results in \cite{frequency-sampling}, where 1D sine function of frequency $k$ is learned in time $O\left(\frac{k^2}{\rho}\right)$, and $\rho$ denotes the minimum sampling density in the input space. %However, even if the inputs in the neural network are the same each time, the outputs of neural networks that have been trained the same epochs are different because of some randomness, such as the randomness in the initialization method and the randomness in the optimization algorithm. 
 If we regard the parameter $\theta$ of the neural networks as a random variable, for the empirical loss Eq \eqref{eq:L2_empiricalloss} we define the loss function expectation $\mathbb{E}_\theta \left[\mathcal{L}_M(\bx;\theta)\right]$. %Then, combining the conclusions drawn from Tab.\ref{tab:Test_Loss_vark}, Fig.\ref{fig:Test_Loss_Var_M} and Fig.\ref{fig:Test_Loss_VarpartM}, 
From the results shown in Fig \ref{fig:Test_Loss_VarpartM}, 
we consider the loss function expectation $\mathbb{E}_\theta \left[\mathcal{L}_M(\bx;\theta)\right]$ positively related to the function frequency $k$ and negatively related to the number of sampling points. That is, when the frequency of the learning function is fixed, the more sampling points, the smaller $\mathbb{E}_\theta \left[\mathcal{L}_M(\bx;\theta)\right]$ is, regardless of the location of the sampling points.

In the following, we use $k$ to denote the complexity (e.g. frequency) of the solution. In order to further study the effect of the complexity of the solution and the number of sampling points on the loss function, we rewrite the residual loss expectation in Definition \ref{de:F}.
\begin{definition}\label{de:F}
    For  $\forall M_r, M \in \mathbb{N^+} ,M \leq M_r$ and $\forall k>0$, the residual loss expectation in a certain subregion(region) is defined as
    \begin{equation}
    \mathcal{F}(k,M,M_r) =\mathbb{E}_\theta \left[\frac{1}{M_{r}} \sum_{i = 1}^{M} |r(x_i;\theta|^2 \right],
    % , \quad \bx \sim \rho(\bx)
    \end{equation}
    where $\mathbb{N^+}$ denotes the set of positive integers.
\end{definition}

\begin{remark}
\label{rem:mk_F}
%    If the complexity of the solution $\bk=\{{k_i}\}_{i=1}^{N_k}$ is different in different regions $\Omega = \{{\Omega_i}\}_{i=1}^{N_k}$,  and the number of sampling points $\{{M_i}\}_{i=1}^{N_k}$ corresponding to different regions is also different, then
In different subregion $\Omega_i$, where $\Omega =\{\Omega_{i}\}_{i = 1}^{N_k}$, the complexity  $k_i$ of the solution is different and $\bk =\{{k_i}\}_{i=1}^{N_k}$, therefore the number of sampling points ${M_i}$ in different subregion is different. % and  $\{{M_i}\}_{i=1}^{N_k}$,
We have
    \begin{equation}\label{eq:mk_F}
    \mathcal{F}(k,M_r,M_r) =  \sum_{i = 1}^{N_{k}} \mathcal{F}(k_i,M_i,M_r), \quad \mbox{where} \quad \sum_{i = 1}^{N_{k}} M_i = M_r.
    \end{equation}
\end{remark}

%In conjunction with the conclusions of the above account, we naturally give the following one hypothesis.

% In order to establish the error estimation, we also have the following assumption.

In order to establish the error estimate, we give Assumption \ref{asu:F} following the reference \cite{frequency-sampling}. 
\begin{assumption}
    \label{asu:F}
     $\forall k,\alpha,\beta >0, \forall M \leq M_r \in \mathbb{N^+}$, $\exists$ $p,q >0$, s.t.,
    $$\mathcal{F}(\alpha k,M,M_r) = \alpha^p \mathcal{F}(k,M,M_r),  \quad \mathcal{F}(k,\beta M,M_r) = \frac{1}{\beta^q} \mathcal{F}(k,M,M_r), $$
    which is due to the positively relation of the function complexity $k$ and the negatively relation of the number of sampling points $M$.

\end{assumption}
	% @@@@@@@@@@@@@@@@@@@@@@@@@@
	% @@@@@@@@@@@@@@@@@@@@@@@@@@
	\section{Implement of MMPDE-Net}
	\label{sec:Implement of MMPDE-Net}
	
		% @@@@@@@@@@@@@@@@@@@@@@@@@@
\subsection{MMPDE-Net}
\label{sec:MMPDE-Net}
In this section, our goal is to develop the neural networks that are independent of the problem and used only for adaptive sampling. %According to the account in Section \ref{sec:An interesting phenomenon}, 
%In order to increase the efficiency, 
We would like to concentrate the sampling points as much as possible on the regions where the behaviors of the solution are intricate. This is similar to the idea of the moving mesh PDE (MMPDE) in the traditional numerical methods. Firstly, we briefly introduce the MMPDE derivation in two-dimensional case.
Let's %start by 
define a one-to-one coordinate transformation on $\Omega \subset \mathbb{R}^2$
\begin{equation}\label{eq:coordinate transformation}
     \bx = \bx(\boldsymbol{\xi}),  \quad
     \forall \boldsymbol{\xi} = (\xi,\eta) \in \Omega, 
\end{equation}
where $\boldsymbol{\xi} = (\xi,\eta)$ represents the coordinates before transformation and $\bx = (x,y)$ represents the coordinates after transformation. 

The mesh functional in variational approach can usually be expressed in the form
\begin{equation}\label{eq:mesh energy functional0}
    E(\boldsymbol{\xi})=\int_{\Omega}^{} \sum_{i,j,\alpha,\beta} w^{i,j} \frac{\partial \boldsymbol{\xi}^\alpha}{\partial \bx^i} \frac{\partial \boldsymbol{\xi}^\beta}{\partial \bx^j} d\bx,
\end{equation}
where $W = (w_{i,j}), W^{-1} = (w^{i,j})$ are symmetric positive definite matrices that are monitor functions in a matrix form. 
According to Winslow's variable diffusion method \cite{winslow1966numerical},  a special case of Eq \eqref{eq:mesh energy functional0} is defined as follows

\begin{equation}\label{eq:mesh energy functional}
E(\xi,\eta)=\int_{\Omega}^{}\ \frac{1}{w} \left(\vert\nabla \xi\vert^2+\vert\nabla \eta\vert^2\right) dx dy,
\end{equation}
where $\nabla = \left(\frac{\partial}{\partial x},\frac{\partial}{\partial y}\right)^T$, $w$ is a monitor function that depends on the solution. The Euler–Lagrange equation whose solution minimizes Eq \eqref{eq:mesh energy functional} is as follows
\begin{equation} \label{eq:The Euler–Lagrange equation}
    \left\{
        \begin{aligned}
            \nabla \left(\frac{1}{w} \nabla \xi\right) = 0, \\
            \nabla \left(\frac{1}{w} \nabla \eta \right) = 0, 
        \end{aligned}
    \right.
\end{equation}
which can be regarded as a steady state of the heat flow equation Eq \eqref{eq:the heat flow equations}
\begin{equation}\label{eq:the heat flow equations}
\left\{
\begin{aligned}
    \frac{\partial \xi}{\partial t}-\nabla \left(\frac{1}{w} \nabla \xi\right) = 0, \\
    \frac{\partial \eta}{\partial t}-\nabla \left(\frac{1}{w} \nabla \eta \right) = 0. 
\end{aligned}
\right.
\end{equation}
Interchanging the dependent and independent variables in Eq \eqref{eq:the heat flow equations}, we have

\textcolor{black}{\begin{equation}\label{eq:MMPDE_2D}
   \begin{split}
        (x_t,y_t)^T=-\frac{(x_\xi,y_\xi)^T}{J}\left\{\frac{\partial}{\partial \xi}\left((x_\eta,y_\eta) \frac{1}{Jw}(x_\eta,y_\eta)^T\right)-\frac{\partial}{\partial \eta}\left((x_\xi,y_\xi) \frac{1}{Jw}(x_\eta,y_\eta)^T\right)\right\} \\
        -\frac{(x_\eta,y_\eta)^T}{J}\left\{-\frac{\partial}{\partial \xi}\left((x_\eta,y_\eta) \frac{1}{Jw}(x_\xi,y_\xi)^T\right)+\frac{\partial}{\partial \eta}\left((x_\xi,y_\xi) \frac{1}{Jw}(x_\xi,y_\xi)^T\right)\right\},
   \end{split}
\end{equation}
where $J=x_\xi y_\eta -x_\eta y_\xi$.}

%From the neural network perspective, 
The neural networks can be treated as a coordinate transformation,  as in Eq \eqref{eq:coordinate transformation}. An intuitive idea is to solve Eq \eqref{eq:MMPDE_2D} with input $(\xi,\eta)$ and output $(x,y)$.
%We give the following notation for the convenience of the later derivation.
We denote that
\begin{equation}\label{eq:MMPDE_2D_notation}
    \left\{
    \begin{aligned}
        \mathcal{S}_1(x,y,\xi,\eta,J,w)=&\frac{\partial}{\partial \xi}\left((x_\eta,y_\eta) \frac{1}{Jw}(x_\eta,y_\eta)^T\right)-\frac{\partial}{\partial \eta}\left((x_\xi,y_\xi) \frac{1}{Jw}(x_\eta,y_\eta)^T\right), \\
        \mathcal{S}_2(x,y,\xi,\eta,J,w)=&-\frac{\partial}{\partial \xi}\left((x_\eta,y_\eta) \frac{1}{Jw}(x_\xi,y_\xi)^T\right)+\frac{\partial}{\partial \eta}\left((x_\xi,y_\xi) \frac{1}{Jw}(x_\xi,y_\xi)^T\right).
    \end{aligned}
    \right.
\end{equation}
Using Eq \eqref{eq:MMPDE_2D_notation}, Eq \eqref{eq:MMPDE_2D} is rewritten as

\begin{equation}\label{eq:MMPDE_2D_version2}
    \left\{
    \begin{aligned}
        x_t& + \frac{x_\xi}{J}\mathcal{S}_1(x,y,\xi,\eta,J,w)
        + \frac{x_\eta}{J}\mathcal{S}_2(x,y,\xi,\eta,J,w) = 0 , \quad  (\xi,\eta) \in \Omega, \\
        y_t& + \frac{y_\xi}{J}\mathcal{S}_1(x,y,\xi,\eta,J,w)
        + \frac{y_\eta}{J}\mathcal{S}_2(x,y,\xi,\eta,J,w) = 0 , \quad  (\xi,\eta) \in \Omega,\\
        x &= \xi,  y = \eta, \quad  (\xi,\eta) \in \partial\Omega.\\
    \end{aligned}
    \right.
\end{equation}
Now we establish the neural networks to solve Eq \eqref{eq:MMPDE_2D_version2}. When dealing with time discretization, we use the forward Euler formula. %but there are still two issues to be aware of. 
There are still two issues to be addressed. One is to define the monitor function $w(\xi,\eta)$, the other is to enforce the boundary conditions.
%\begin{itemize}
%    \item How to construct the monitor function $w(\xi,\eta)$.
%    \item How to handle the boundary conditions.
%\end{itemize}

In traditional numerical methods, the role of the monitor function is to guide the movement of the mesh points so that they are concentrated in the region where the monitor function is large. 
\textcolor{black}{As mentioned in Section \ref{sec:Problem formulation}, the solution $\bu(\bx)$ considered in this paper varies greatly in a certain subregion of $\Omega$,
therefore the monitor function can be defined empirically as the function of $\bu$, the derivatives of $\bu$ and/or the second derivatives of $\bu$ (\cite{huang1994moving},\cite{ren-wang},\cite{he2009numerical}), e.g., Eq \eqref{eq:2d_monitorfunction}. }
%so empirically we will choose the monitor function as %$w(\xi,\eta) = (1+\bu(\xi,\eta) + (\bu(\xi,\eta))^2 + (\lvert \nabla \bu(\xi,\eta) \rvert)^2)^{\frac{1}{2}}$(\cite{huang1994moving},\cite{ren-wang},\cite{he2009numerical}), which including function itself and/or the derivatives.}
%Therefore, we expect that in regions where the solution is more complex, e.g., where the gradient of the solution is larger, the larger the value of the monitor function will be in this region. 
%Then, 
%The monitor function is chosen as $w(\xi,\eta) = (1+(\lvert \nabla u(\xi,\eta) \rvert)^2)^{\frac{1}{2}}$. 
%Additionally, in constructing the monitor function, 
% Since MMPDE-Net is independent of solving  equations, we need to know some prior knowledge about the solution.
%e.g., either by using the exact solution or by obtaining the reference solution in some way.
When dealing with boundary conditions, instead of adding loss terms of the boundary conditions as in PINN, we utilize the enforcement approach(\cite{YANG202314},\cite{lyu2020enforcing}). The fully connected neural networks is employed as our network architecture. We assume that the sampled training points are $\left\{(\xi_i,\eta_i)\right\}_{i=1}^{M}$, then the loss function of the MMPDE is given by
% \begin{equation}\label{eq:MMPDE_loss}
%     \begin{aligned}
%       & \mathcal{L}oss_{MMPDE}\left((\xi,\eta);\theta\right) \\
%       & =  \frac{1}{M} \sum_{i = 1}^{M} \left|x_t(\xi_i,\eta_i) + \frac{x_\xi(\xi_i,\eta_i)}{J}\mathcal{S}_1(x,y,\xi_i,\eta_i,J,w) 
%       + \frac{x_\eta(\xi_i,\eta_i)}{J}\mathcal{S}_2(x,y,\xi_i,\eta_i,J,w)\right|^2 \\
%      & + \frac{1}{M} \sum_{i = 1}^{M} \left|y_t(\xi_i,\eta_i) + \frac{y_\xi(\xi_i,\eta_i)}{J}\mathcal{S}_1(x,y,\xi_i,\eta_i,J,w)
%         + \frac{y_\eta(\xi_i,\eta_i)}{J}\mathcal{S}_2(x,y,\xi_i,\eta_i,J,w)\right|^2.
%     \end{aligned}
% \end{equation}
\begin{equation}\label{eq:MMPDE_loss}
    \begin{aligned}
      & \mathcal{L}oss_{MMPDE}\left((\xi,\eta);\theta\right) \\
      & =  \frac{1}{M} \sum_{i = 1}^{M} \left|\frac{x_i-\xi_i}{\tau} + \frac{x_\xi(\xi_i,\eta_i)}{J}\mathcal{S}_1(x,y,\xi_i,\eta_i,J,w) 
      + \frac{x_\eta(\xi_i,\eta_i)}{J}\mathcal{S}_2(x,y,\xi_i,\eta_i,J,w)\right|^2 \\
     & + \frac{1}{M} \sum_{i = 1}^{M} \left|\frac{y_i-\eta_i}{\tau} + \frac{y_\xi(\xi_i,\eta_i)}{J}\mathcal{S}_1(x,y,\xi_i,\eta_i,J,w)
        + \frac{y_\eta(\xi_i,\eta_i)}{J}\mathcal{S}_2(x,y,\xi_i,\eta_i,J,w)\right|^2.
    \end{aligned}
\end{equation}

\textcolor{black}{The method Adam is used to optimize the loss function (Eq \eqref{eq:MMPDE_loss}) in the training of MMPDE-Net.} Then, MMPDE-Net is described in Algorithm \ref{alg:MMPDE-Net} for the two-dimensional problem, and a flowchart of the computational process is given in Fig \ref{fig:MMPDE-Net}.

\begin{algorithm}[htbp]
\caption{Algorithm of MMPDE-Net in the 2D case}\label{alg:MMPDE-Net}
\textbf{Symbols:} Maximum epoch number of MMPDE-Net $N_1$; the number of total points $M_r$; initial training points $\mathcal{X} :=\left\{(\xi_k,\eta_k)\right\}_{k=1}^{M_r} \subset \Omega$; parameter $\tau$ in Eq \eqref{eq:MMPDE_loss}.

\textbf{Constructing the monitor function:}\\
% Input $\left\{\mathcal{X}_k\right\}_{k=1}^{M_r+M_b} :=\left\{\xi_k,\eta_k,t_k\right\}_{k=1}^{M_r+M_b}$ into PINN. 
Given the sampling points $\mathcal{X}$.

Utilizing prior knowledge to obtain $[u(\mathcal{X}),\nabla u(\mathcal{X}),...]$.

Construct the monitor function $w[u(\mathcal{X}),\nabla u(\mathcal{X}),...]$, abbreviated as $w(u(\mathcal{X}))$.

\textbf{Adaptive Sampling:}\\
Input $\mathcal{X}$ into neural networks.

Initialize the output $\tilde {\mathcal{X}} := \tilde {\mathcal{X}} (\mathcal{X};\theta^0)$.

\For{$i =0:N_1-1$}{

    $\mathcal{L}oss[\tilde {\mathcal{X}} (\mathcal{X};\theta^i);w(u(\mathcal{X}))] = 
    \mathcal{L}oss_{MMPDE}(\mathcal{X};\theta^i)$ (Eq \eqref{eq:MMPDE_loss});

    Update $\theta^{i+1}$ by descending the gradient of $\mathcal{L}oss[\tilde {\mathcal{X}} (\mathcal{X};\theta^i);w(u(\mathcal{X}))]$.
}

Output the new training points $\tilde{\mathcal{X}}=\tilde {\mathcal{X}} (\mathcal{X};\theta^{N_1}) = \left\{(x_k,y_k)\right\}_{k=1}^{M_r}$.
\end{algorithm}

\begin{figure}[htbp]
\centering
\includegraphics[width = 0.98\textwidth]{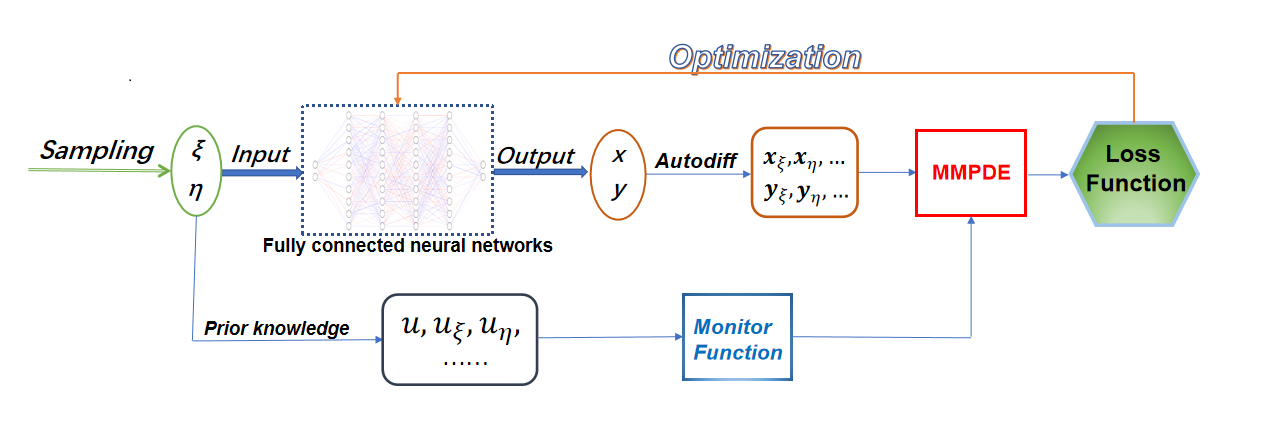}
\caption{Flow chart of MMPDE-Net.}
\label{fig:MMPDE-Net}
\end{figure}

Now we consider the solution function $u=ce^{-c^2{(x^2+y^2)}}$ and the monitor function $w = (1+u^2)^{\frac{1}{2}}$. The uniform sampling points are input into MMPDE-Net. The output of MMPDE-Net for $c = 5, 10, 50, 100$ are shown in Fig \ref{fig:wang_Fig2_a}. It can be seen the distribution of sampling points is clearly more concentrated on the place where the function value is large, which is due to the effect of the monitor function. In Fig \ref{fig:wgrad}, we consider function $u= e^{-(4x^2 + 9y^2 - 1)^2}$. It is obvious to see that the gradient of solution is important and  the monitor function is taken as $w = (1+(\lvert \nabla u \rvert)^2)^{\frac{1}{2}}$. The gradient of $u$ and the distribution of sampling points are shown in Fig \ref{fig:wgrad}. 
%still fed MMPDE-Net with sampling points from a uniform grid, but changed the form of the monitor function to make it more complex, but Fig.\ref{fig:wgrad} proves that MMPDE-Net still got the desired results.

\begin{figure}[htbp]
\centering
\subfloat[c =5]{\includegraphics[width = 0.24\textwidth]{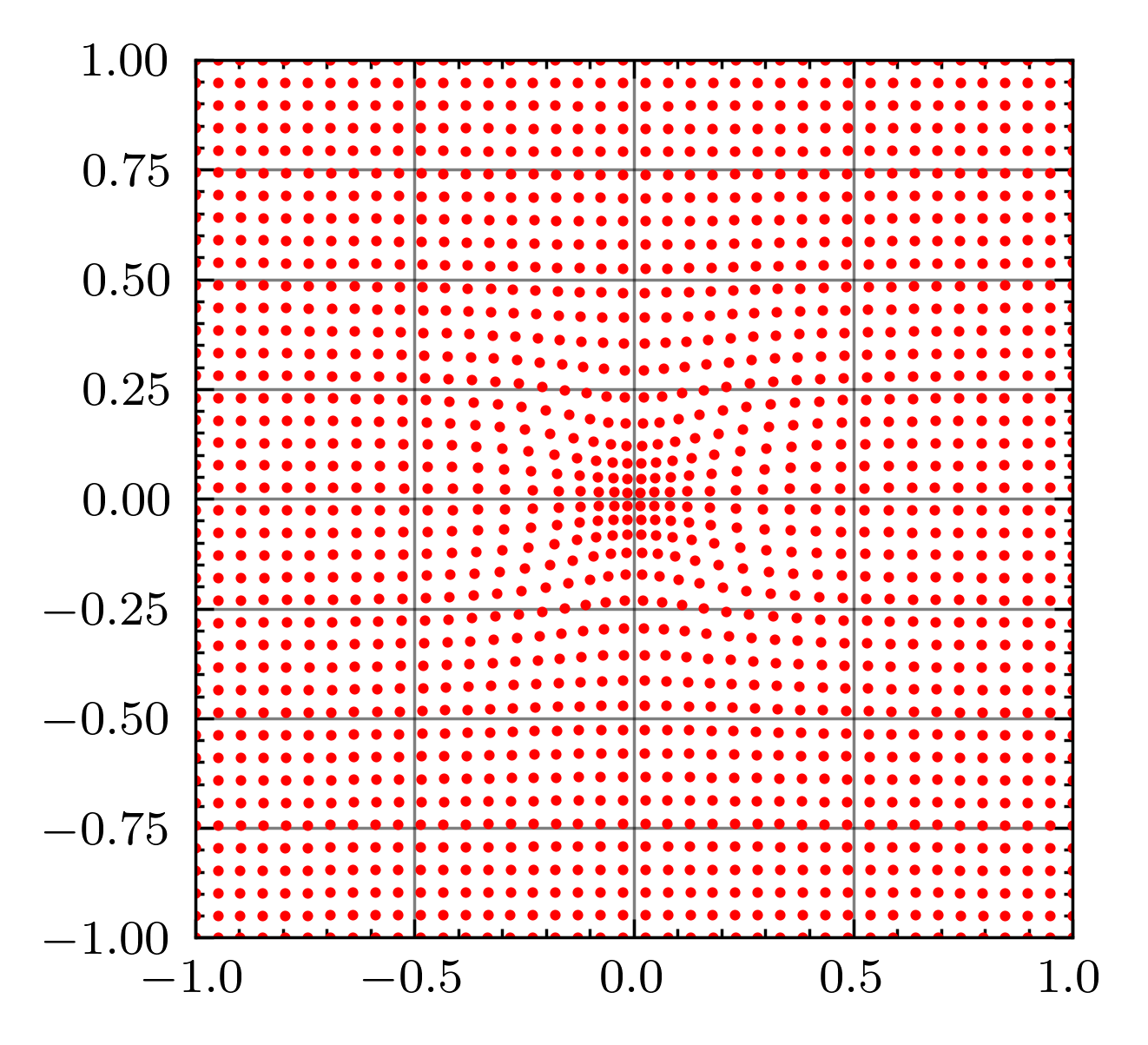}}
\subfloat[c =10]{\includegraphics[width = 0.24\textwidth]{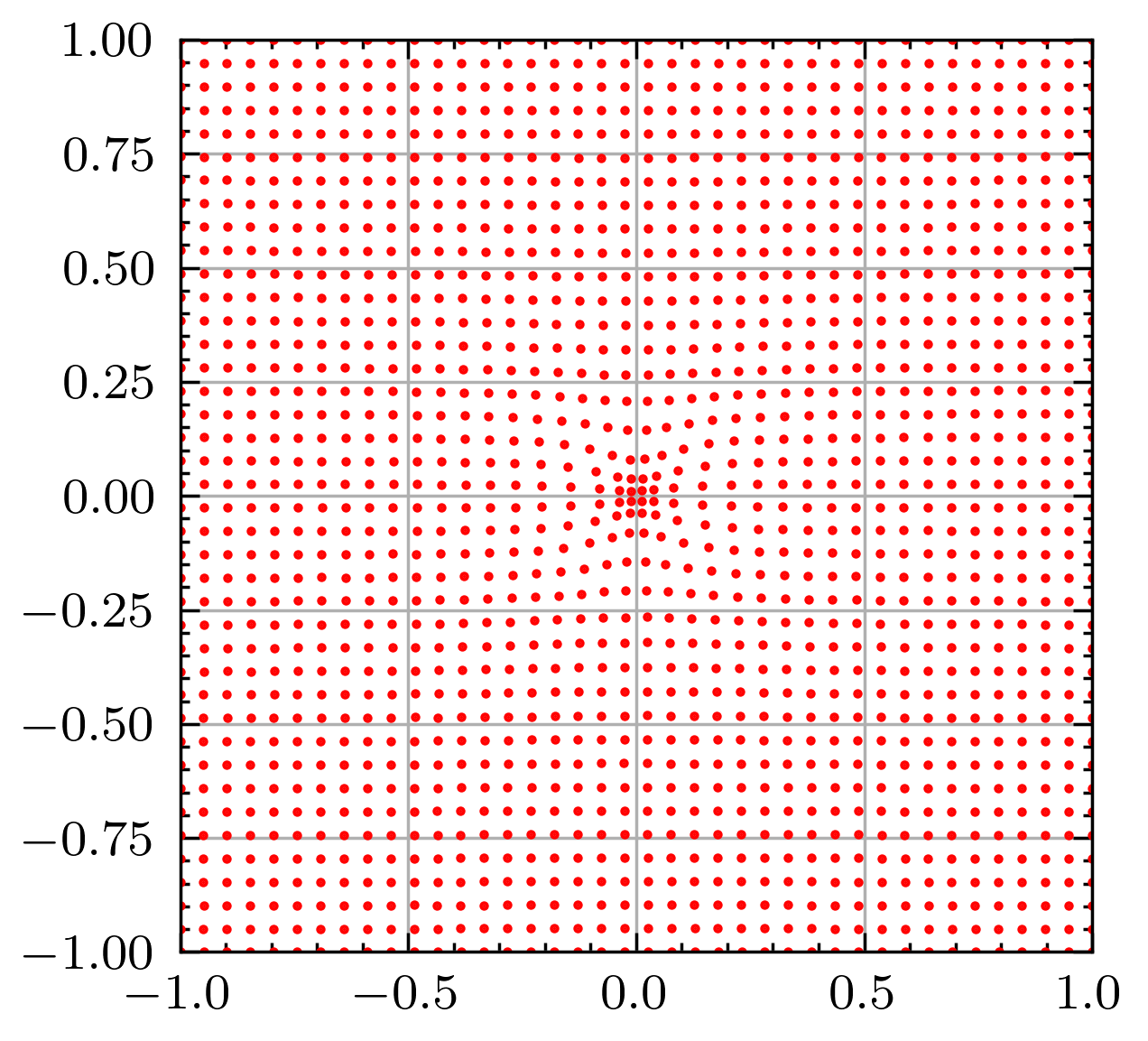}}
\subfloat[c =50]{\includegraphics[width = 0.24\textwidth]{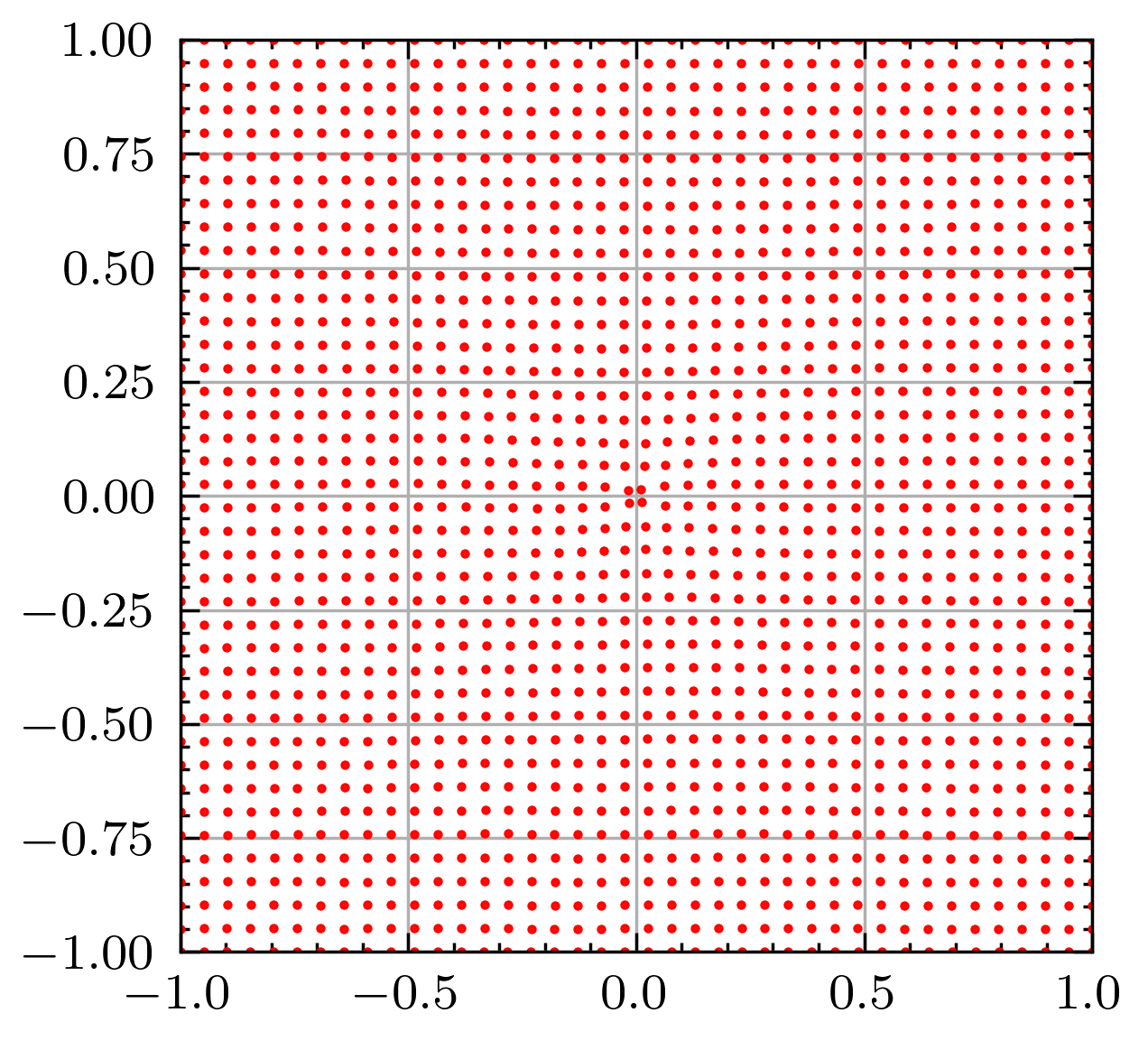}}
\subfloat[c =100]{\includegraphics[width = 0.24\textwidth]{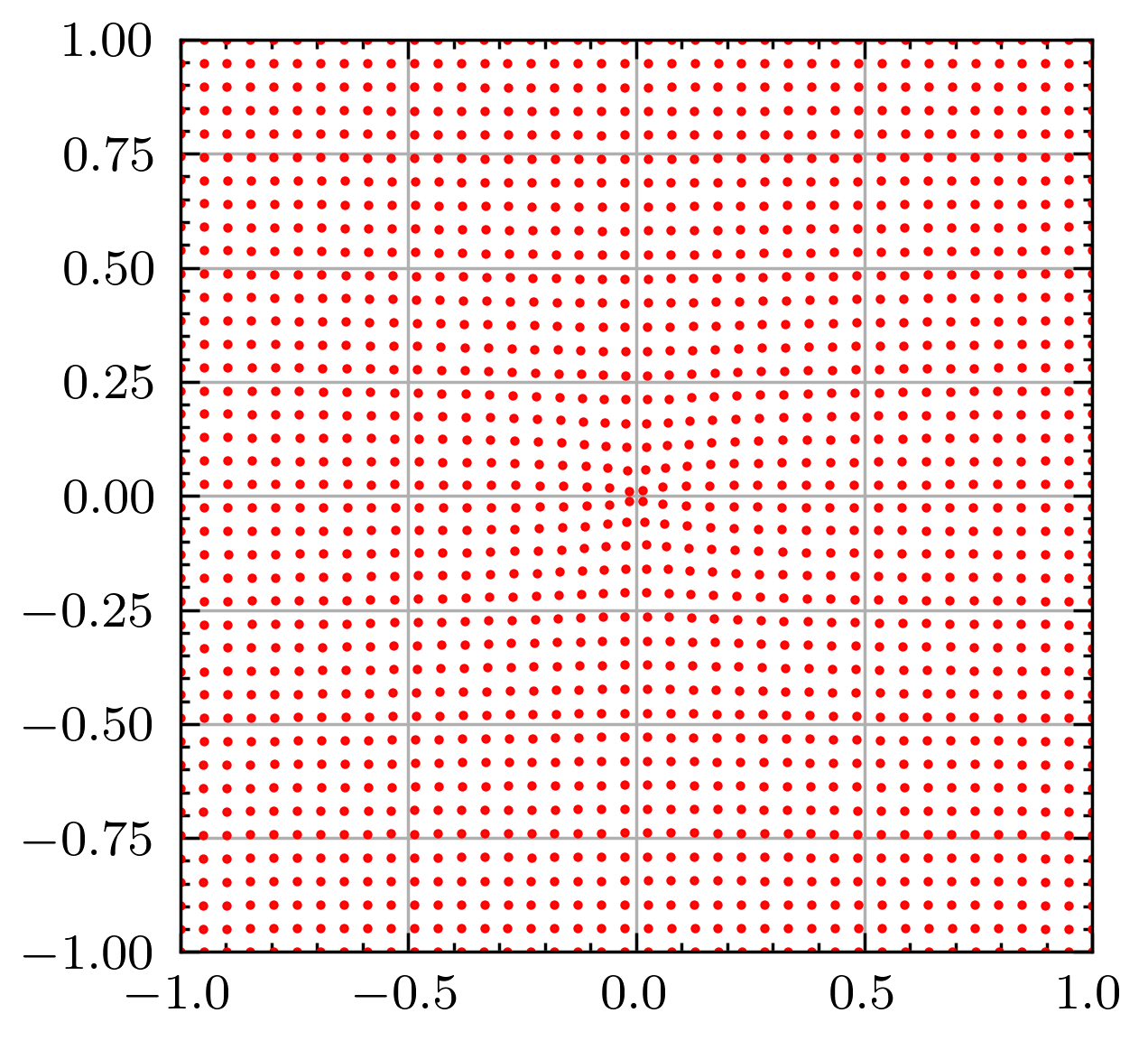}}
\caption{Points distribution for  $u=ce^{-c^2{(x^2+y^2)}}$  and $w = (1+u^2)^{\frac{1}{2}}$ in 2D case when $c$ is increased. (a) $c =5$;  (b) $c =10$;  (c) $c =50$;  (d) $c =100$.}
\label{fig:wang_Fig2_a}
\end{figure}

\begin{figure}[htbp]
\centering
\subfloat[$\nabla u$]{\includegraphics[width = 0.35\textwidth]{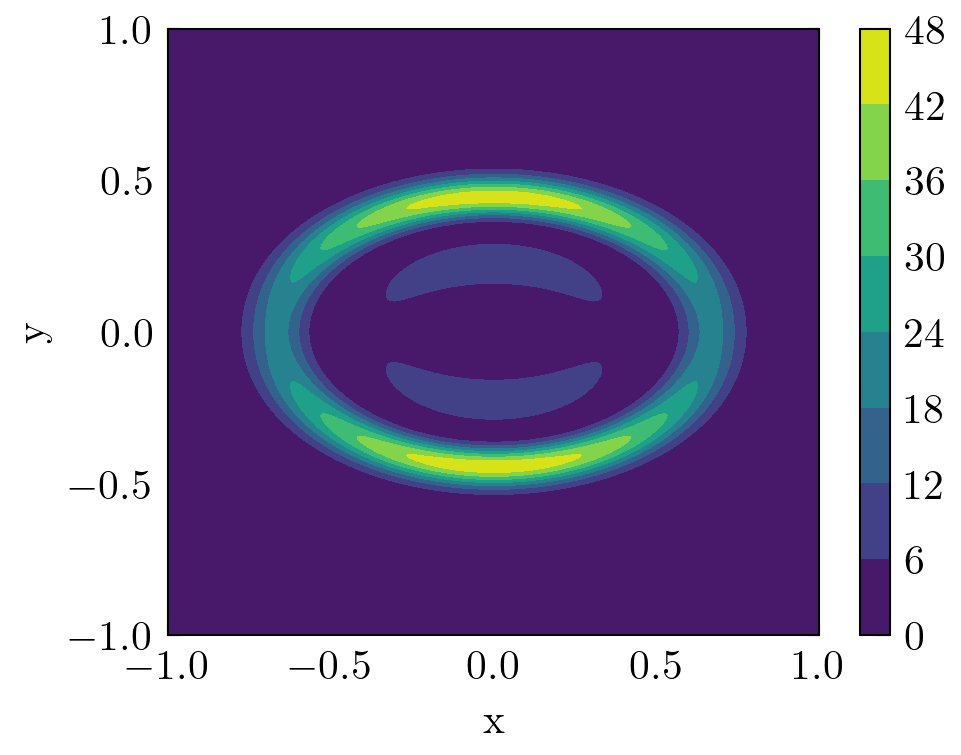}}\quad 
\subfloat[new distribution of sampling points]{\includegraphics[width = 0.30\textwidth]{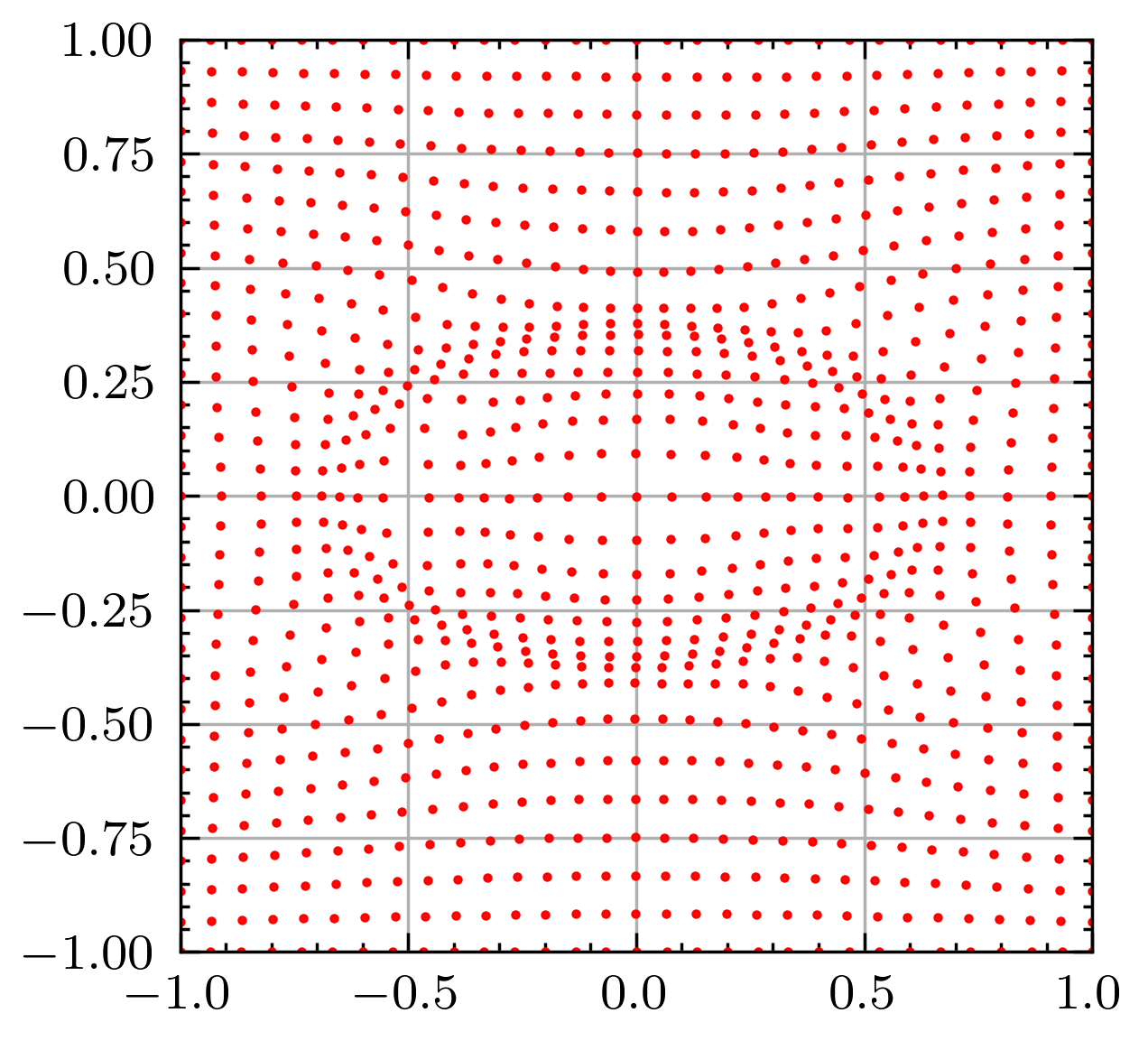}}
\caption{(a) the gradient of $u = e^{-(4x^2 + 9y^2 - 1)^2}$; (b) points distribution with $w = (1+(\lvert \nabla u \rvert)^2)^{\frac{1}{2}}$.} %(c)Heat-map of the gradient of $u = e^{-2(x^2+y^2)}$, (d)Points behavior for $w = (1+(\lvert \nabla u \rvert)^2)^{\frac{1}{2}}$. }
\label{fig:wgrad}
\end{figure}

% When actually sampling collocation points obeying a uniform distribution, we often give a well-divided suitably uniform grid and then sample uniformly from the points of the grid. However, when we use MMPDE-Net, the points 
% %in the original uniform grid are transformed and become not uniform, but 
% are concentrated on the place where the monitor function is large. If we sample the collocation  points uniformly in the transformed grid, in fact the sampled points no longer obey a uniform distribution, but a new distribution.

%\textcolor{red}{When uniformly sampling points, the equispaced uniform mesh sampling method was considered. However, after using MMPDE-Net, the points 
%are concentrated on the place where the monitor function is large. If we sample the collocation  points uniformly in the transformed grid, in fact the sampled points no longer obey a uniform distribution, but a new distribution.} 
%We follow the notation of Section \ref{sec:An interesting phenomenon} and 
As in Section \ref{sec:An interesting phenomenon},
the $\bk$ is denoted as the complexity of the solution $u$. After using MMPDE-Net, the sampling points obey the distribution $\bx \sim \rho(\bk;\bx)$ induced by the monitor function $w(\bk;\bx)$ and MMPDE Eq \eqref{eq:MMPDE_2D}, where $\rho(\bk;\bx)$ is a probability density function (PDF), i.e., $\int_{\Omega}\rho(\bk;\bx)d\bx = 1 $ and $\rho(\bk;\bx) \geq 0$. 
%Obviously, if the new sampling points $\bx \sim \rho(\bk;\bx)$, then the probability that the sampling points are distributed in a region where the monitor function $w(\bk;\bx)$ is larger is higher.
As a result, the sampling points are distributed in the region where the monitor function $w(\bk;\bx)$ is large, i.e., enough points are sampled from the PDF $\rho(\bk;\bx)$.
% the number of points sampled from the PDF $\rho(\bk;\bx)$ is large.
%which is related to some $\rho(\bk;\bx)$.

%As a result, when the number of points sampled from the PDF $\rho(\bk;\bx)$ is sufficiently large, they will concentrate in regions where the monitor function $w(\bk;\bx)$ is large.
		% @@@@@@@@@@@@@@@@@@@@@@@@@@
  	% @@@@@@@@@@@@@@@@@@@@@@@@@@
\subsection{MMPDE-Net with iterations}
\label{sec:MMPDE-Net-Iterations}

An iterative method for moving mesh is proposed in \cite{ren-wang}. The iteration of the grid generation allows us to more precisely control the grid distribution in the regions where the solution has large variations. %In short, it allows the mesh to be more centralized in regions with large solution variations. This is a very inspiring idea. 
In this section, we introduce the iterative algorithm for MMPDE-Net.

Since new sampling points are generated in each iteration based on the old ones, i.e., a one-to-one coordinate transformation is implemented in each iteration, which corresponds to performing the training of MMPDE-Net. It is worth noting that at each training of MMPDE-Net, we need to formulate the monitor function based on the most recent input training points and reinitialize the parameters of the neural networks. MMPDE-Net  with iterations is described in Algorithm \ref{alg:MMPDE-Net-iterations} for the 2D case.

\begin{algorithm}[htbp]
\caption{Algorithm of MMPDE-Net with iterations in the 2D case}
\label{alg:MMPDE-Net-iterations}
\textbf{Symbols:} Maximum epoch number of MMPDE-Net $N_1$;  the number of total points $M_r$; initial training points $\mathcal{X}_0 :=\left\{(\xi_k,\eta_k)\right\}_{k=1}^{M_r} \subset \Omega$;  maximum number of iterations $N_{it}$; parameter $\tau$ in Eq \eqref{eq:MMPDE_loss}.

\For{$j =0:N_{it}-1$}{
\textbf{Constructing the monitor function:}\\
% Input $\left\{\mathcal{X}_k\right\}_{k=1}^{M_r+M_b} :=\left\{\xi_k,\eta_k,t_k\right\}_{k=1}^{M_r+M_b}$ into PINN. 
Given the sampling points $\mathcal{X}_j$.

Utilizing prior knowledge to obtain $[u(\mathcal{X}_j),\nabla u(\mathcal{X}_j),...]$.

Construct the monitor function $w[u(\mathcal{X}_j),\nabla u(\mathcal{X}_j),...]$, abbreviated as $w(u(\mathcal{X}_j))$.

\textbf{Adaptive Sampling:}\\
Input $\mathcal{X}_j$ into neural networks.

Initialize the output $\tilde{\mathcal{X}}:=\tilde {\mathcal{X}} (\mathcal{X}_j;\theta^{0}_j)$.

\For{$i =0:N_1-1$}{

    $\mathcal{L}oss[\tilde {\mathcal{X}} (\mathcal{X}_j;\theta^i_j);w(u(\mathcal{X}_j))] = 
    \mathcal{L}oss_{MMPDE}(\mathcal{X}_j;\theta^i_j)$ (Eq \eqref{eq:MMPDE_loss});

    Update $\theta^{i+1}_j$ by descending the gradient of $\mathcal{L}oss[\tilde {\mathcal{X}}(\mathcal{X}_j;\theta^i_j);w(u(\mathcal{X}_j))]$.
}

Output the new training points $\tilde{\mathcal{X}} =\tilde {\mathcal{X}} (\mathcal{X}_j;\theta^{N_1}_j)$.

Update $\mathcal{X}_{j+1} =\tilde {\mathcal{X}}$.

}
Output the final training points $\mathcal{X}_{N_{it}}:=\tilde {\mathcal{X}} (\mathcal{X}_0;\theta^{N_1}_0,\theta^{N_1}_1,...,\theta^{N_1}_{N_{it}-1}) = \left\{(x_k,y_k)\right\}_{k=1}^{M_r}$.
\end{algorithm}

\begin{figure}[htbp]
\centering
\subfloat[1 iteration]{\includegraphics[width = 0.24\textwidth]{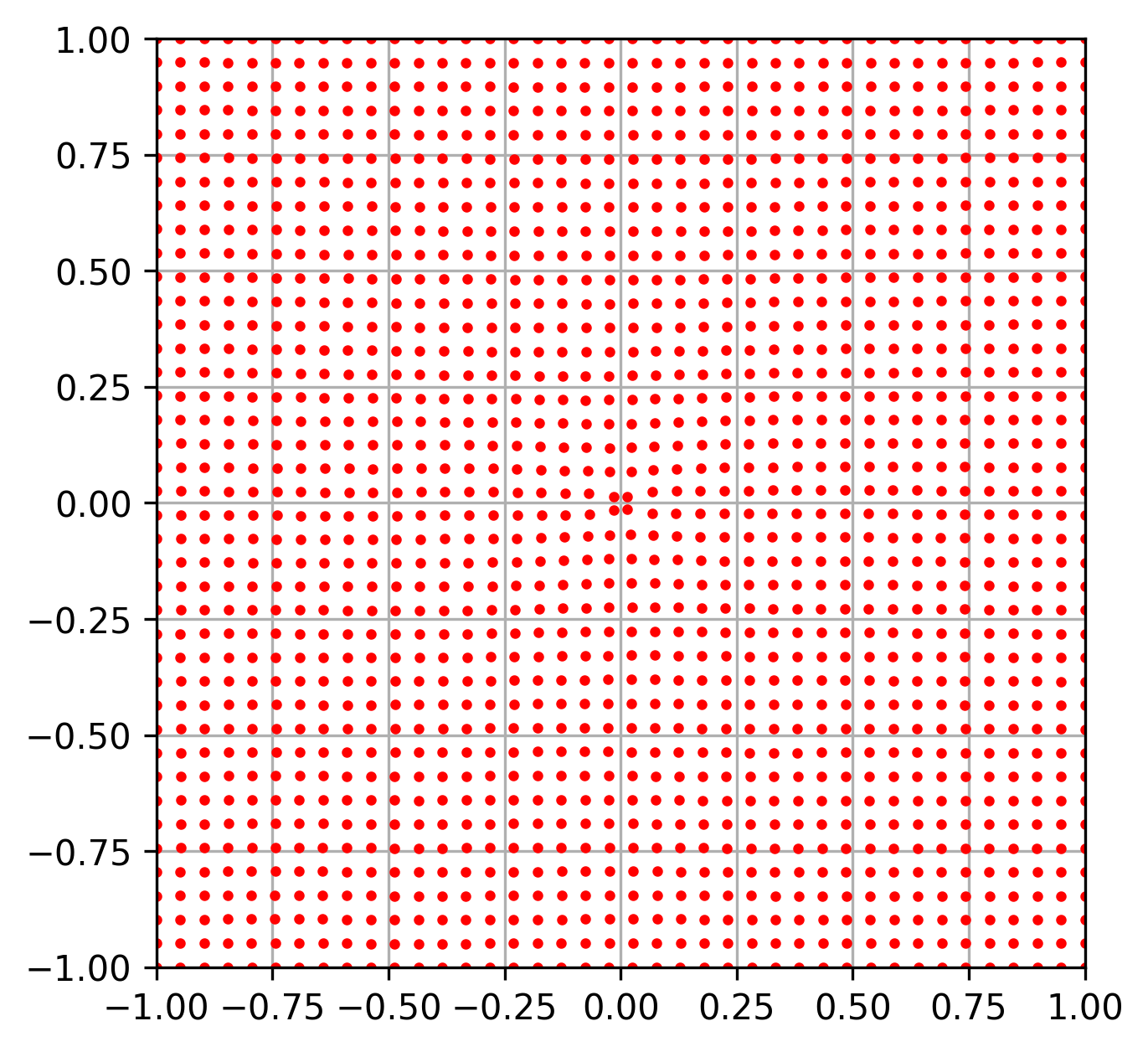}}
\subfloat[2 iterations]{\includegraphics[width = 0.24\textwidth]{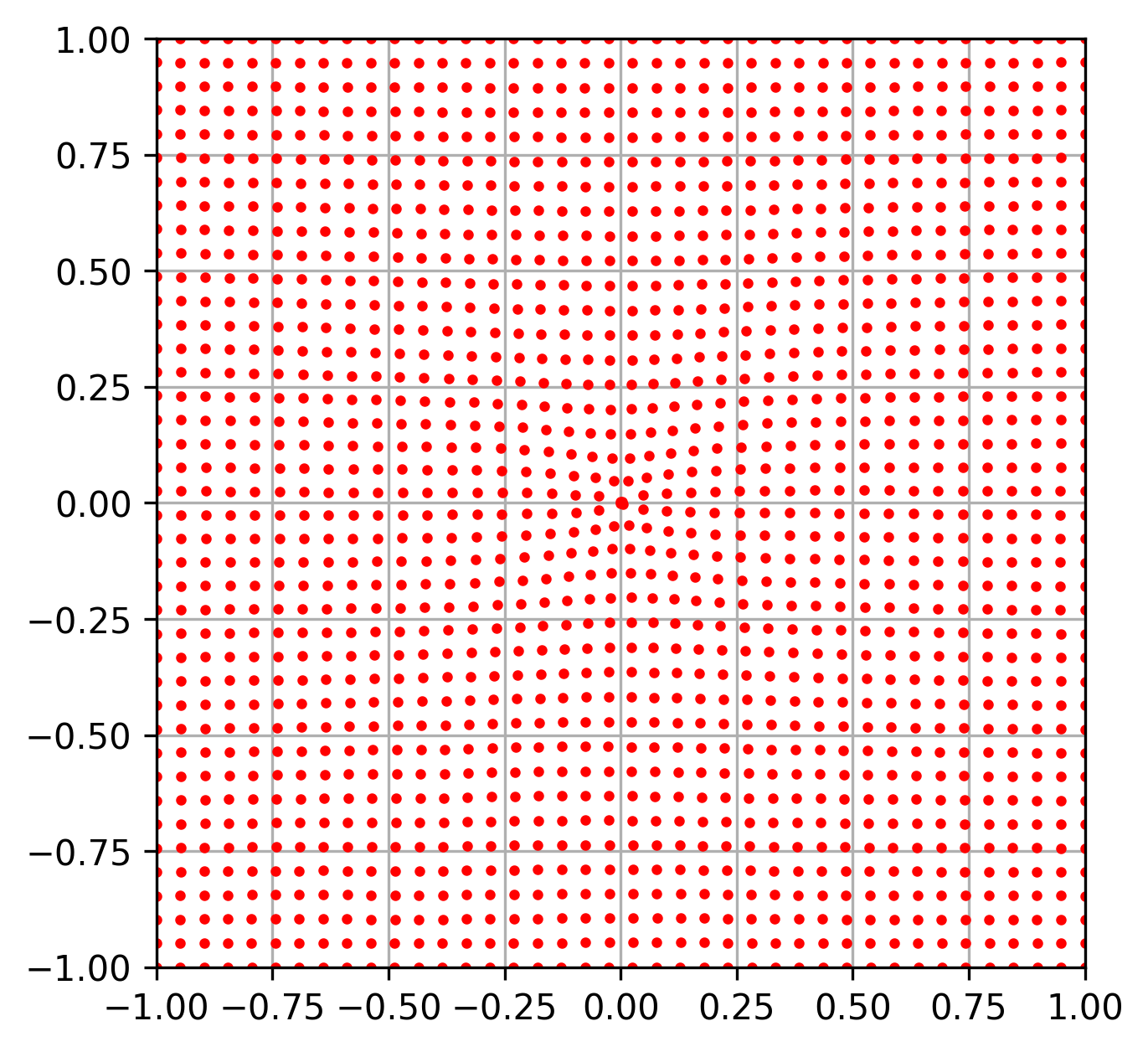}}
\subfloat[3 iterations]{\includegraphics[width = 0.24\textwidth]{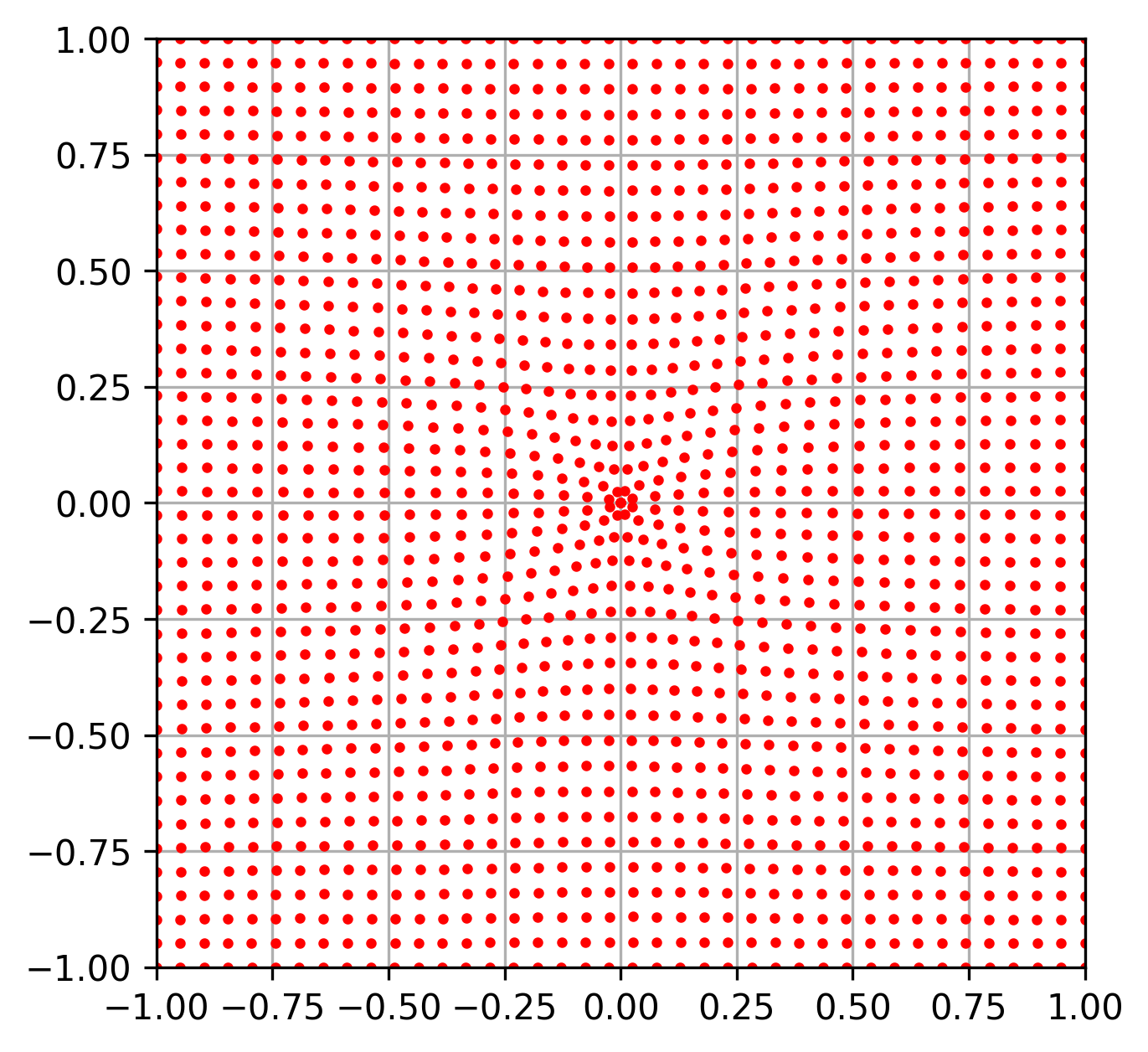}}
\subfloat[5 iterations]{\includegraphics[width = 0.24\textwidth]{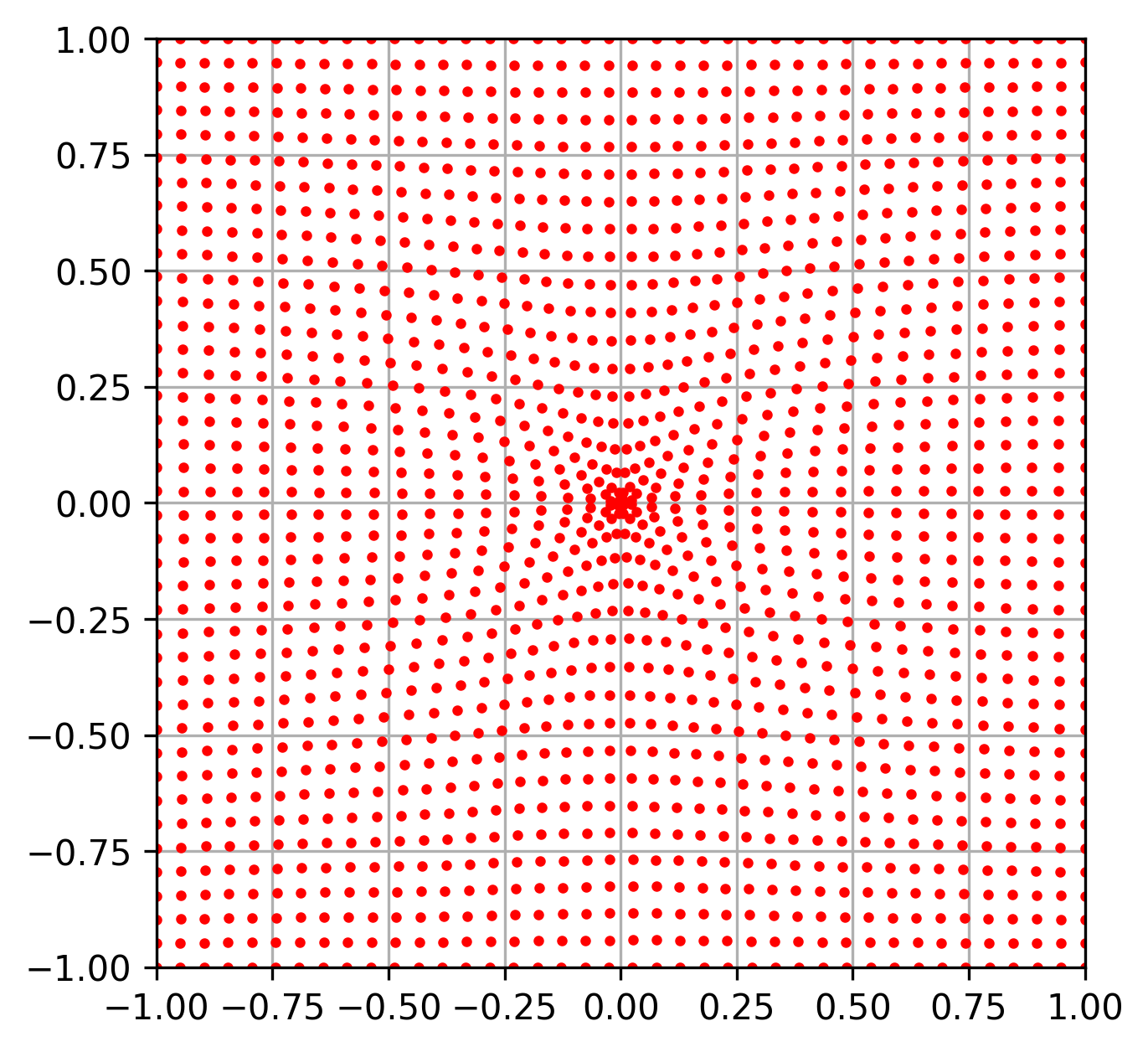}}
\caption{Points distribution  for $u=50e^{-c^2{(x^2+y^2)}}$ and $w = (1+u^2)^{\frac{1}{2}}$ after iterations. (a) after 1 iteration;  (b) after 2 iterations;  (c) after 3 iterations; (d) after 5 iterations.}
\label{fig:wang_Fig5}
\end{figure}

\begin{figure}[htbp]
\centering
\subfloat[1 iteration]{\includegraphics[width = 0.24\textwidth]{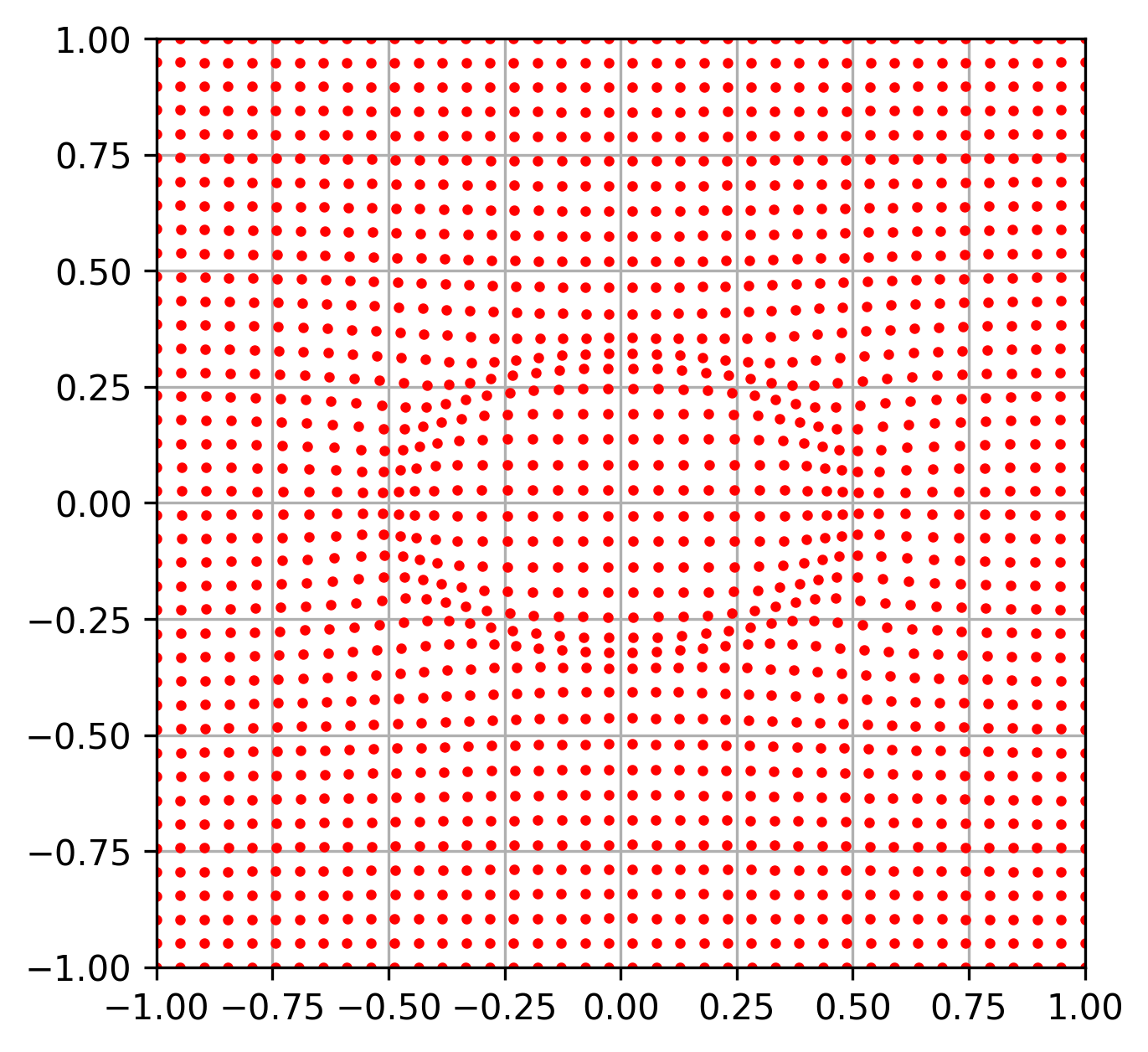}}
\subfloat[2 iterations]{\includegraphics[width = 0.24\textwidth]{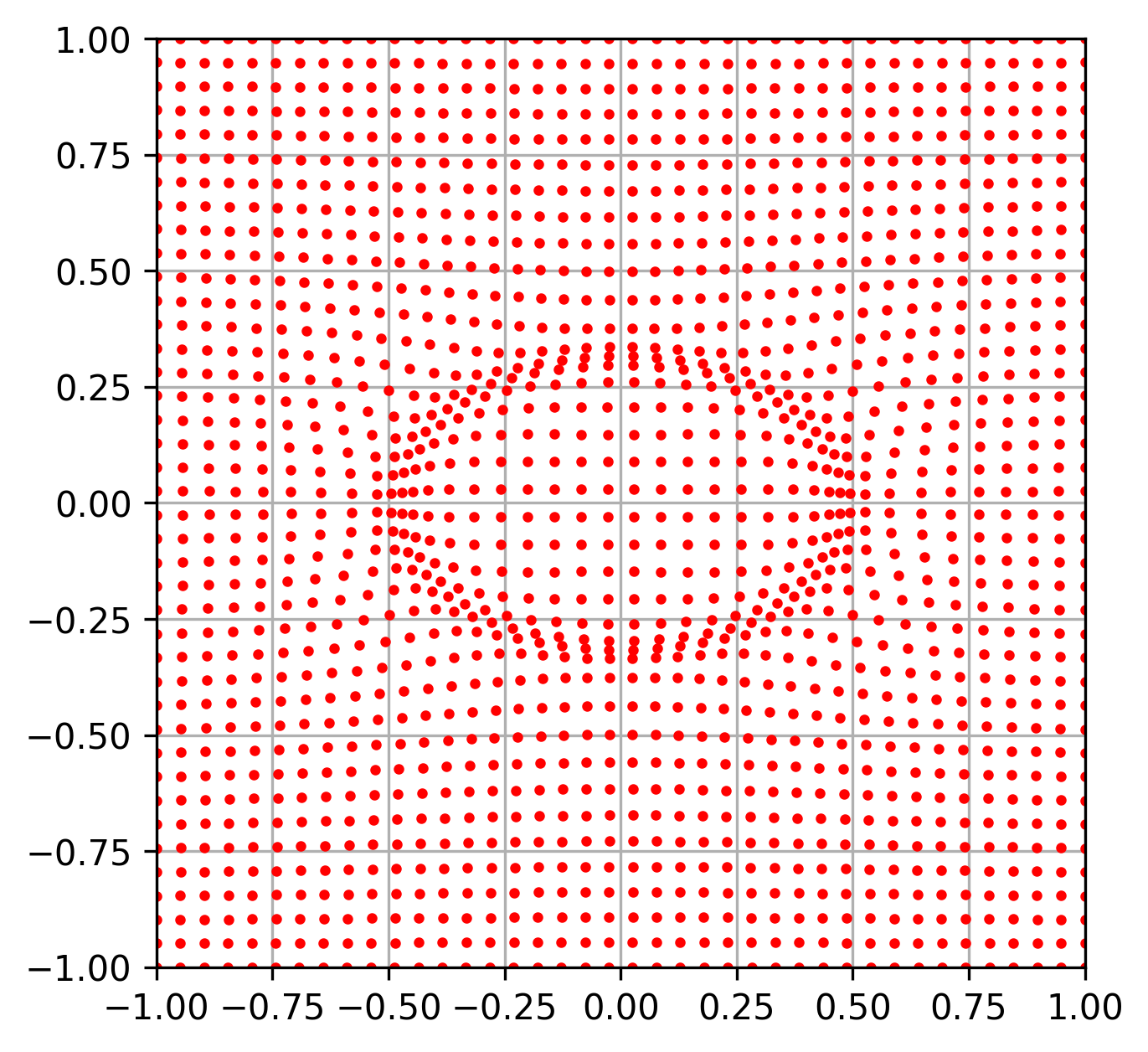}}
\subfloat[3 iterations]{\includegraphics[width = 0.24\textwidth]{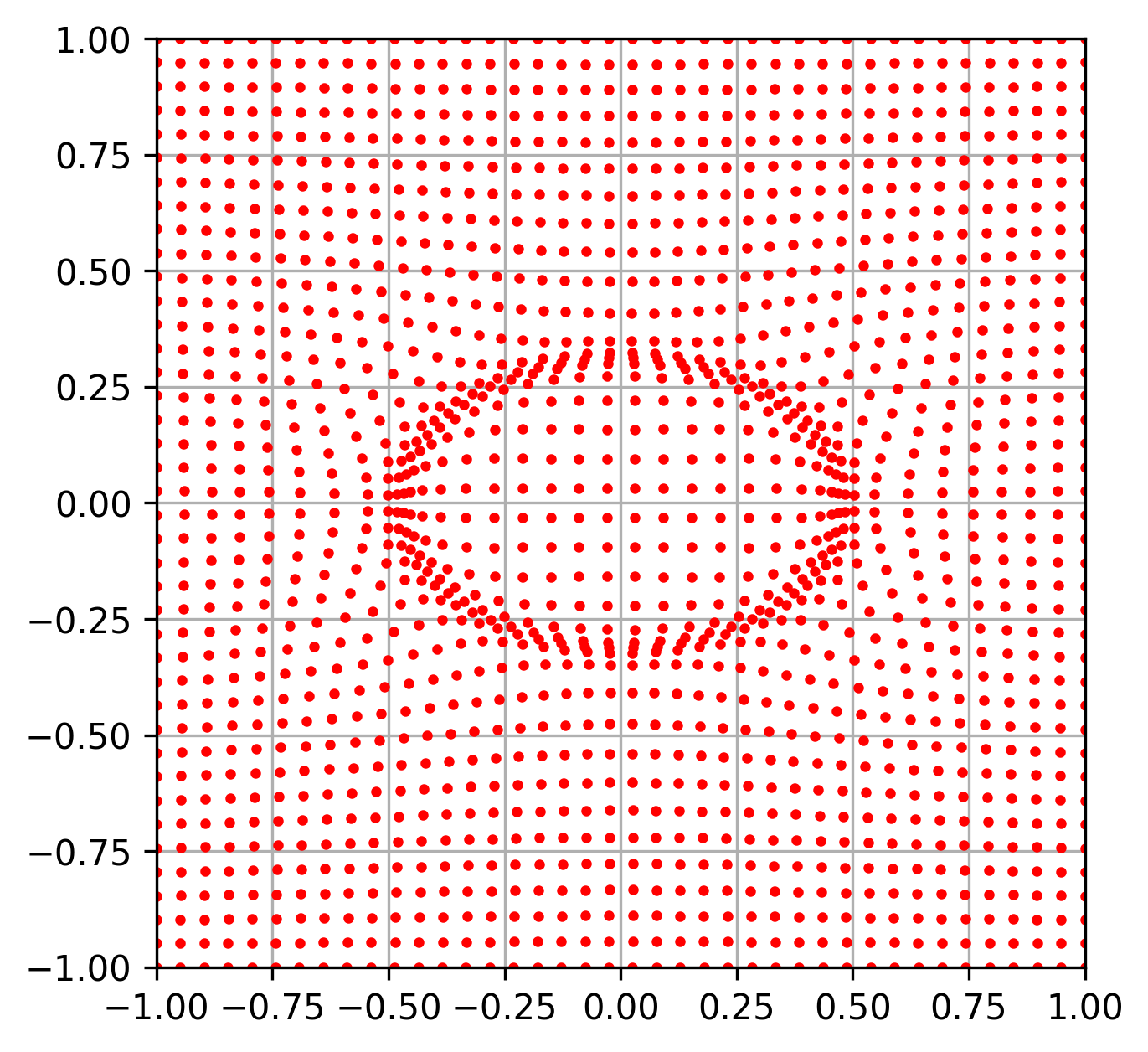}}
\subfloat[4 iterations]{\includegraphics[width = 0.24\textwidth]{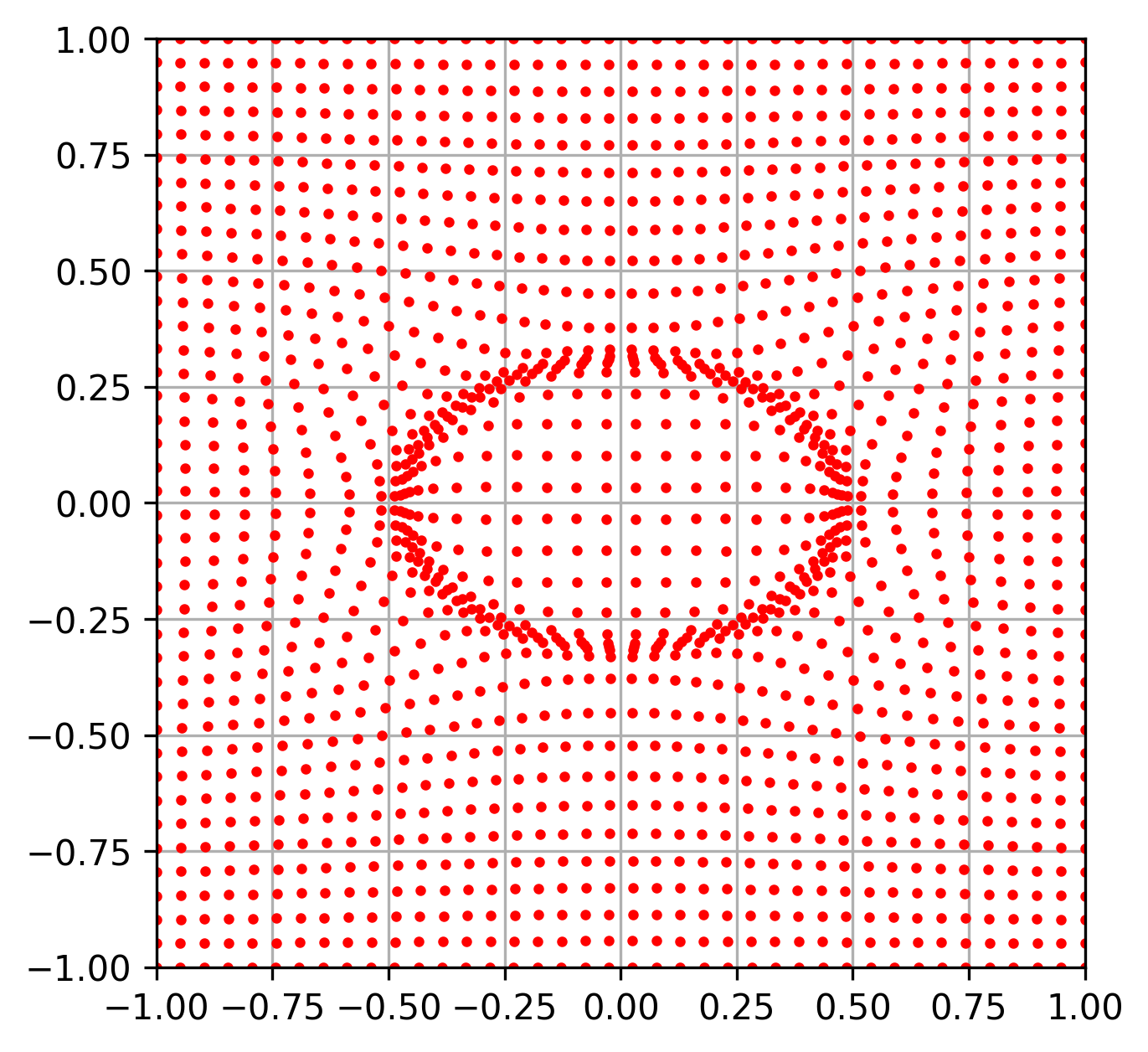}}
\caption{Points distribution  for $u=e^{-8{(4x^2+9y^2-1)^2}}$  and $w =1+u$ after several iterations. (a) after 1 iteration; (b) after 2 iterations;  (c) after 3 iterations; (d) after 4 iterations.}
\label{fig:wang_Fig6a}
\end{figure}

\begin{figure}[htbp]
\centering
\subfloat[1 iteration]{\includegraphics[width = 0.24\textwidth]{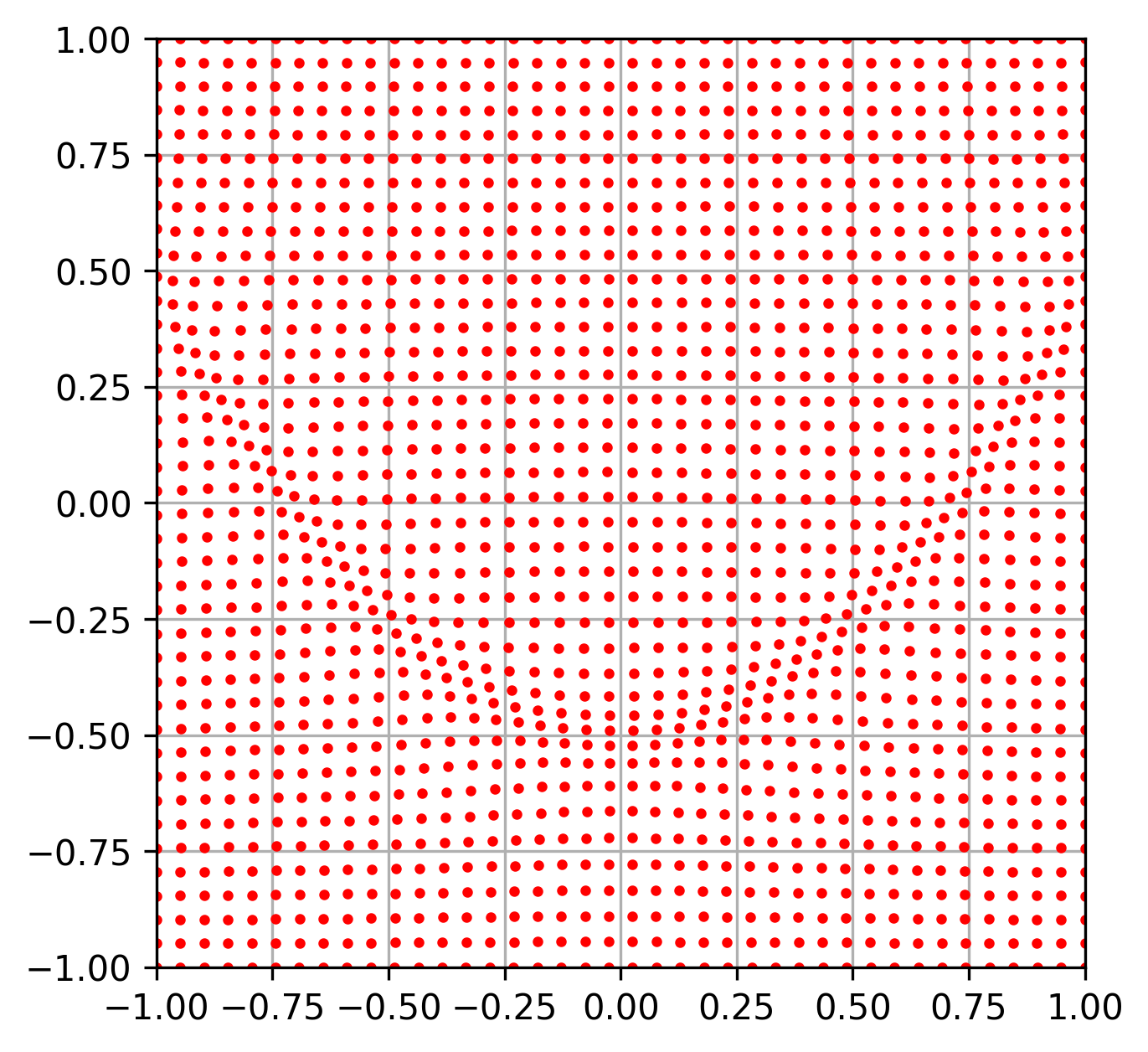}}
\subfloat[2 iterations]{\includegraphics[width = 0.24\textwidth]{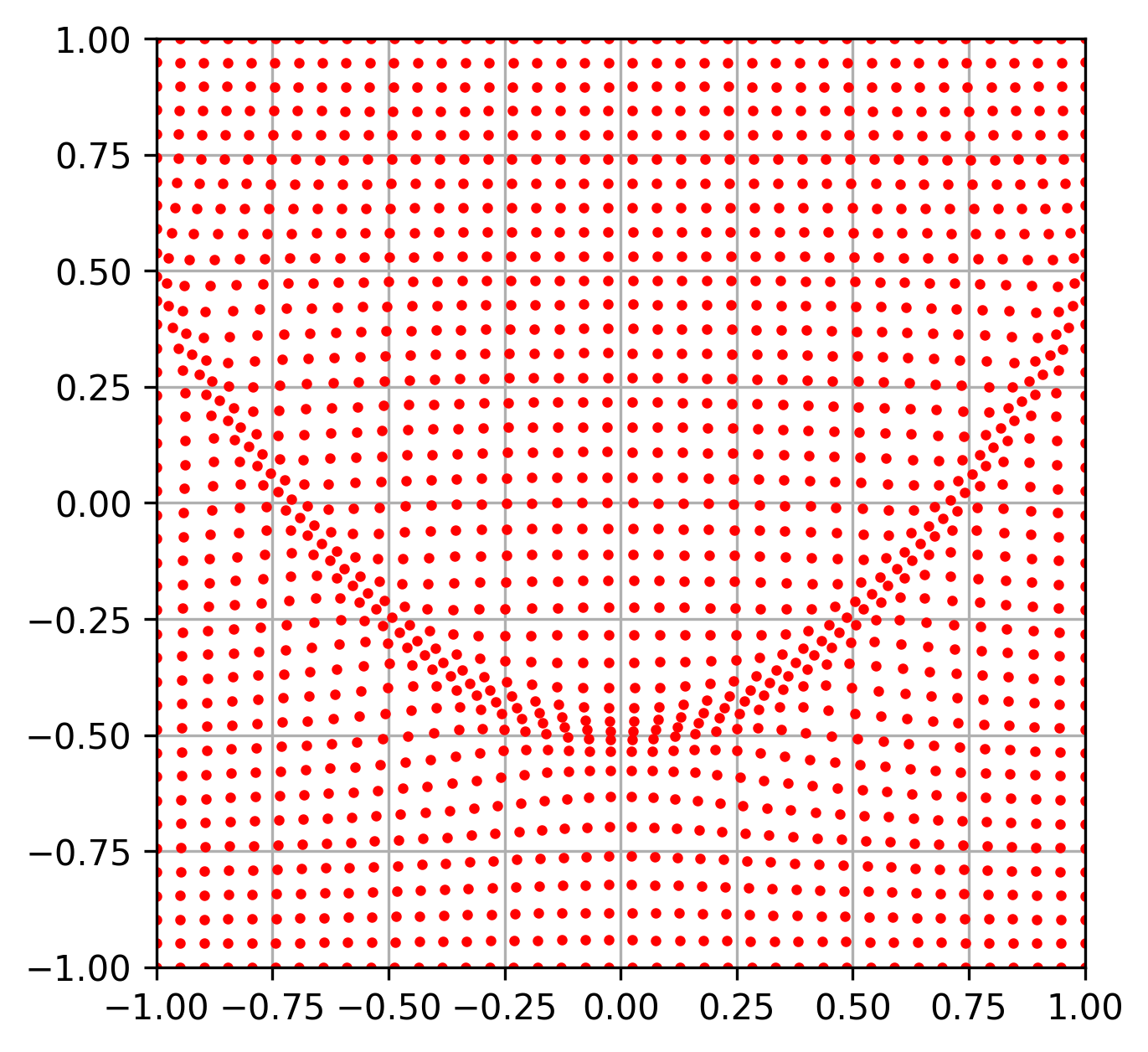}}
\subfloat[3 iterations]{\includegraphics[width = 0.24\textwidth]{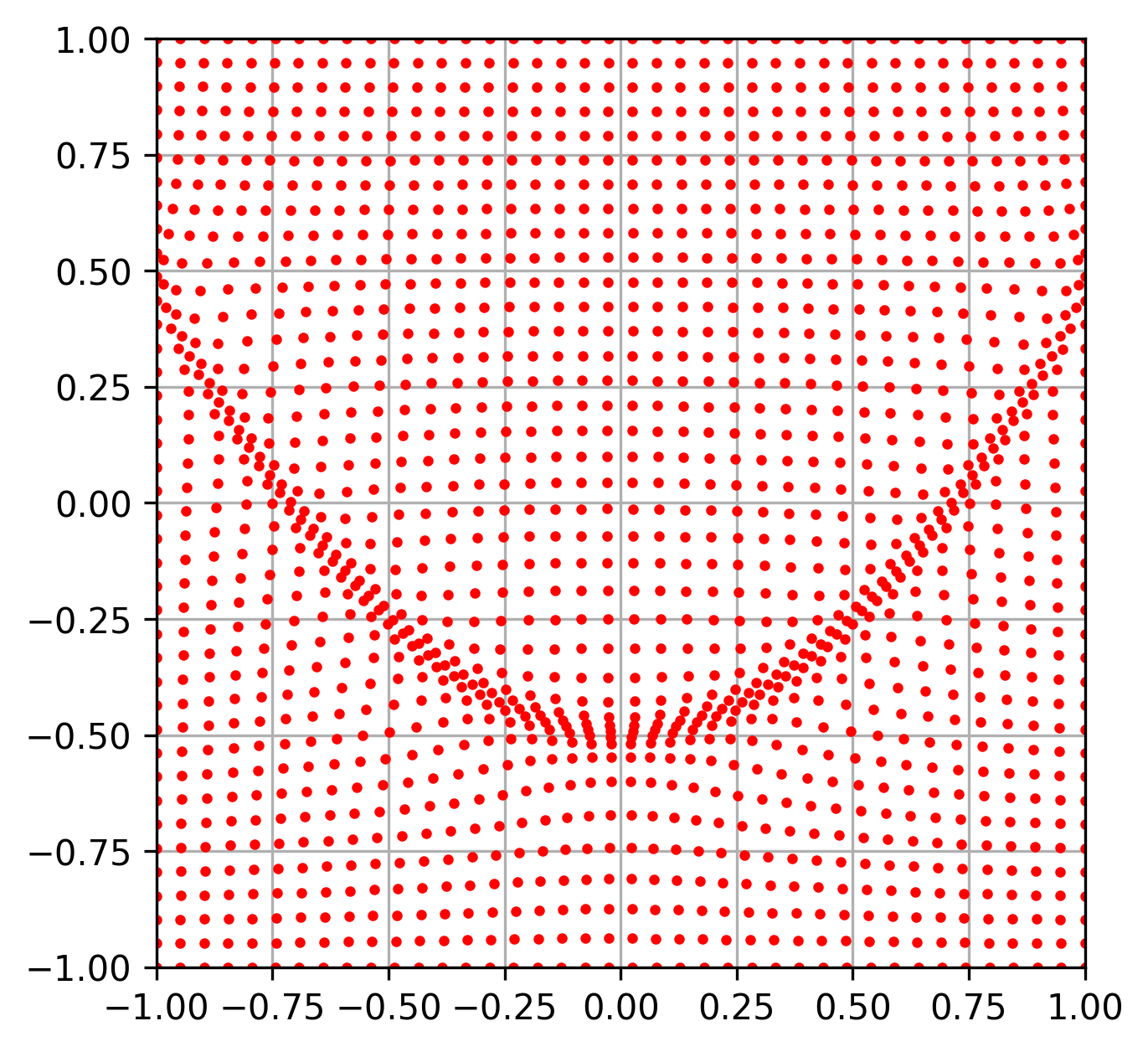}}
\subfloat[4 iterations]{\includegraphics[width = 0.24\textwidth]{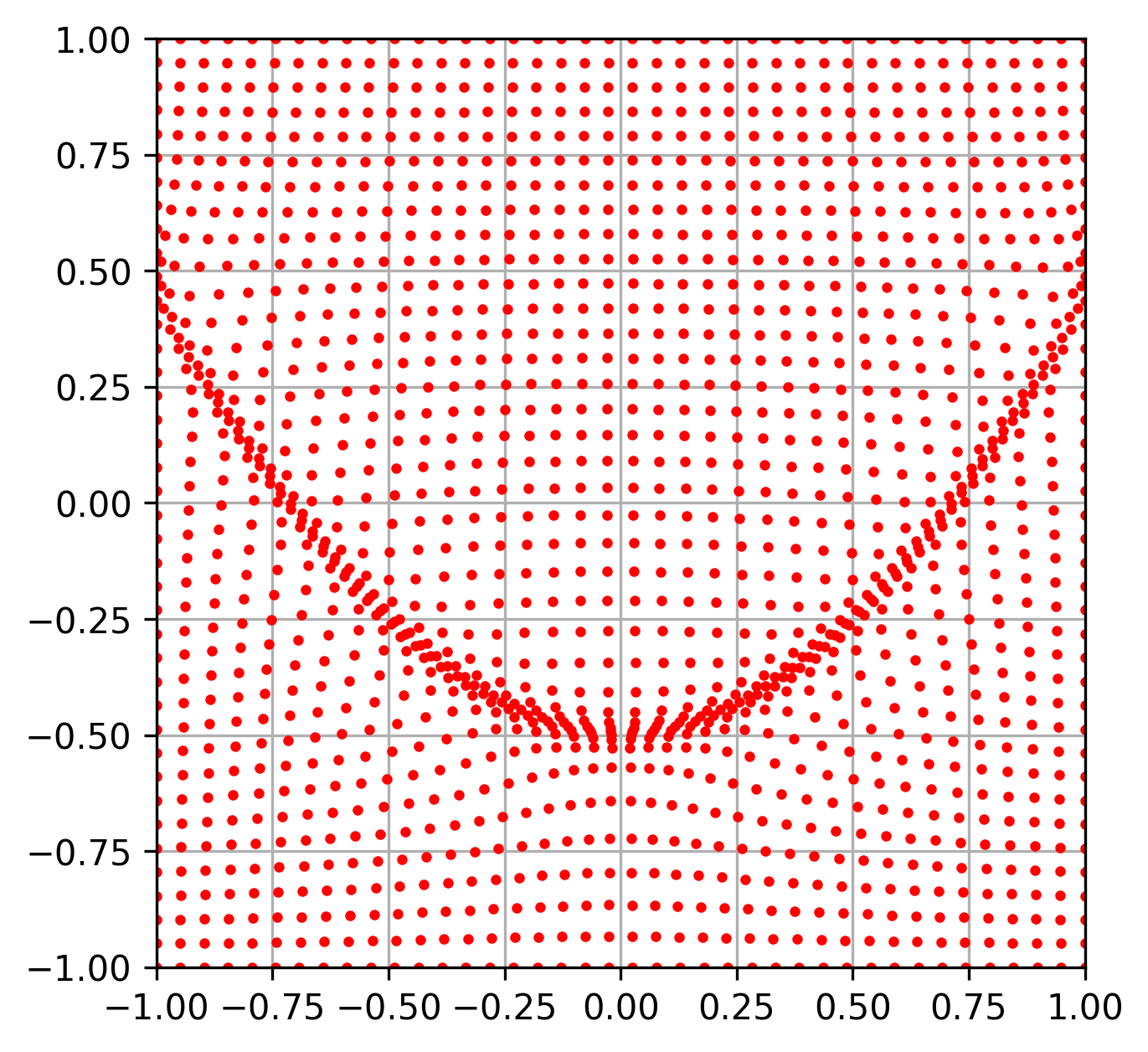}}
\caption{Points distribution  for $u=e^{-100{(-x^2+y+0.5)^2}}$  and $w =1+u$  after several iterations. (a) after 1 iteration;  (b) after 2 iterations;  (c) after 3 iterations; (d) after 4 iterations.}
\label{fig:wang_Fig6b}
\end{figure}

In Fig \ref{fig:wang_Fig5}, \ref{fig:wang_Fig6a} and \ref{fig:wang_Fig6b}, we reproduce the numerical results in \cite{ren-wang} about mesh iterations. In these figures, the monitor function is given and only iteration number varies.
%Each figure shows the results given a fixed analytic solution and monitoring function, only increasing the number of iterations. 
It can be seen that as the iterations increase, the distribution of points behaves more and more concentrated which is due to the monitor function. In fact, 
%as long as the optimization algorithm can find the minimum of the loss function (Eq \ref{eq:MMPDE_loss}),  %a suitable monitoring function $w(\bk;\bx)$, 
this iterative process with a well-chosen monitor function can be repeated until the concentration at the sampling point is satisfactory \cite{ren-wang}, as long as the loss function Eq \eqref{eq:MMPDE_loss} can be optimized. %%Since our later proofs focus on error estimation for deep learning solvers, we directly give Assumption \ref{asu:rho} for MMPDE-Net.
%% \begin{assumption}
%%     \label{asu:rho}
%%     By choosing a suitable monitoring function $w(\bk;\bx)$ or increasing the number of iterations of MMPDE-Net, it is possible to allow the sampling points to be arbitrarily concentrated in the region where the solution is complex.
%% \end{assumption}
%\begin{assumption}
%    \label{asu:rho}
%    The optimization algorithm always finds the optimal parameters $\theta$ to find the minimum of the loss function, i.e., $\min\limits_{\theta} \mathcal{L}oss_{MMPDE}(\xi,\eta;\theta)$
%\end{assumption}
% @@@@@@@@@@@@@@@@@@@@@@@@@@
% @@@@@@@@@@@@@@@@@@@@@@@@@@

\section{Implement of MS-PINN}
\label{sec:Implement of MS-PINN}
   
\subsection{MS-PINN}
\label{sec:MS-PINN}

Since MMPDE-Net is essentially a coordinate transformation mapping, it can be combined with many deep learning solvers, in this work we choose PINN, which is one of the most widely used neural networks.
     
%Let us now consider how to combine these two neural networks. 
It is not straightforward to %inappropriate to simply
embed MMPDE-Net into PINN, for example, by adding loss Eq \eqref{eq:MMPDE_loss} to the loss of PINN, which would not only make the loss function more difficult to optimize, but also reduce the efficiency of training. The combination should be a mutually beneficial relationship: PINN will feed MMPDE-Net the prior information about the solution to help to define the monitor function; MMPDE-Net will feed new sampling points back to PINN to output better approximate solutions.

%According to the above idea, 
Denote $\theta_P$ and $\theta_{MM}$ as the parameters of PINN and MMPDE-Net respectively, the training of MS-PINN  is divided into three stages: 

\begin{itemize}
    \item The pre-training with PINN. The input is the initial sampling points $\mathcal{X}$, and after a small number $N_1$ epochs of training, the output is the information of the solution on the initial training points $u(\mathcal{X};\theta_P^{N_1}),\nabla u(\mathcal{X};\theta_P^{N_1}),...$. 
    \item The adaptive sampling with MMPDE-Net. The input is the initial sampling points $\mathcal{X}$, at the same time, we use the outputs of PINN from the pre-training stage to define the monitor function $w[u(\mathcal{X};\theta_P^{N_1}),\nabla u(\mathcal{X};\theta_P^{N_1}),...)]$, abbreviated by $w(u(\mathcal{X};\theta_P^{N_1}))$, and the output is the new sampling points  after the coordinate transformation $\tilde{\mathcal{X}}=\tilde {\mathcal{X}} (\mathcal{X};\theta_P^{N_1},\theta_{MM}^{N_2})$.
    \item The formal training with PINN. The input is the new sampling points $\tilde{\mathcal{X}}$ and after $N_3$ epochs of training, the output is the solution at the new sampling points $u(\tilde{\mathcal{X}};\theta_P^{N_3})$.
\end{itemize}

\textcolor{black}{It is worth noting that we do not reinitialize the parameters of the PINN during the formal training stage. Instead, we utilize the transfer learning technique in machine learning to inherit the parameters obtained from the pre-training, which does help the convergence of the loss function.} Moreover, the loss functions of MMPDE-Net and PINN are independent; the loss function of MMPDE-Net is in the form of Eq \eqref{eq:MMPDE_loss} which consists of only one loss term based on MMPDE; the loss function of PINN is in the form of Eq \eqref{eq:L2_empiricalloss}, which generally consists of a residual loss term and a boundary loss term.

\begin{algorithm}[htbp]
\caption{Algorithm of MS-PINN in the 2D case}\label{alg:MMPDE-PINN}

\textbf{Symbols:} Maximum epoch number of pre-training, MMPDE-Net and  formal training: $N_1$, $N_2$ and $N_3$; the numbers of total points in $\Omega$  and $\partial \Omega$: $M_r$, and $M_b$; initial training points $\mathcal{X} :=\left\{(\xi_k,\eta_k)\right\}_{k=1}^{M_r} \subset \Omega$ and $\mathcal{X}_b :=\left\{(\xi_k,\eta_k)\right\}_{k=1}^{M_b} \subset \partial\Omega$; points in the test set $\mathcal{X}_T :=\left\{(\xi_k^T,\eta_k^T)\right\}_{k=1}^{M_T} \subset \Omega \cup \partial\Omega$; the number of test set points $M_T$; $\theta_P$ and $\theta_{MM}$ are the parameters of PINN and MMPDE-Net respectively; parameter $\tau$ in Eq \eqref{eq:MMPDE_loss}.

\textbf{Pre-training:}\\
Input $\mathcal{X}$ and  $\mathcal{X}_b$ into PINN. 

Initialize the output $u(\mathcal{X},\mathcal{X}_b;\theta_P^0)$.

\For{$i =0:N_1-1$}{
 
$\mathcal{L}oss[u(\mathcal{X},\mathcal{X}_b;\theta_{P}^i)] := \mathcal{L}oss_{res}(\mathcal{X};\theta_{P}^i) + \mathcal{L}oss_{bou}(\mathcal{X}_b;\theta_{P}^i)$;

Update $\theta_{P}^{i+1}$ by descending the gradient of $\mathcal{L}oss[u(\mathcal{X},\mathcal{X}_b;\theta_{P}^i)]$.
}
Output $u(\mathcal{X};\theta_{P}^{N_1}),\nabla u(\mathcal{X};\theta_{P}^{N_1}),...$.

Construct the monitor function $w[u(\mathcal{X};\theta_{P}^{N_1}),\nabla u(\mathcal{X};\theta_{P}^{N_1}),...)]$, abbreviated as $w(u(\mathcal{X};\theta_{P}^{N_1}))$.

\textbf{Adaptive Sampling:}\\

Input $\mathcal{X}$ into MMPDE-Net.

Initialize the output $\tilde {\mathcal{X}} := \tilde {\mathcal{X}} (\mathcal{X};\theta_{MM}^0)$.

\For{$i =0:N_2-1$}{
    
    $\mathcal{L}oss[\tilde {\mathcal{X}} (\mathcal{X};\theta_{MM}^i);w(u(\mathcal{X};\theta_{P}^{N_1}))] = 
    \mathcal{L}oss_{MMPDE} (\mathcal{X};\theta_{MM}^i)$ 
 (Eq \eqref{eq:MMPDE_loss});

    Update $\theta_{MM}^{i+1}$ by descending the gradient of $\mathcal{L}oss[\tilde {\mathcal{X}} (\mathcal{X};\theta_{MM}^{i});w(u(\mathcal{X};\theta_{P}^{N_1}))]$.
}

Output the new traning points $\tilde{\mathcal{X}} =\tilde {\mathcal{X}} (\mathcal{X};\theta_{MM}^{N_2}) = \left\{(x_k,y_k)\right\}_{k=1}^{M_r}$.

\textbf{Formal training:}\\
Input  $\tilde{\mathcal{X}}$ and $\mathcal{X}_b$ into PINN. 

Inheriting parameters from pre-training $\theta_{P}^{0} = \theta_{P}^{N_1}$.

Initialize the output $u(\tilde{\mathcal{X}},\mathcal{X}_b;\theta_P^0)$.

\For{$i =0:N_3-1$}{
 
$\mathcal{L}oss[u(\tilde {\mathcal{X}},\mathcal{X}_b;\theta_{P}^i)] := \mathcal{L}oss_{res}(\tilde {\mathcal{X}};\theta_{P}^i)  + \mathcal{L}oss_{bou}
(\mathcal{X}_b;\theta_{P}^i)$;

Update $\theta_{P}^{i+1}$ by descending the gradient of $\mathcal{L}oss[u(\tilde {\mathcal{X}},\mathcal{X}_b;\theta_{P}^i)]$.
}

Input  test set $\mathcal{X}_T$ into PINN. 

Output $u(\mathcal{X}_T;\theta_{P}^{N_3})$.
\end{algorithm}

\begin{figure}[htbp]
\centering
\includegraphics[width = 0.85\textwidth]{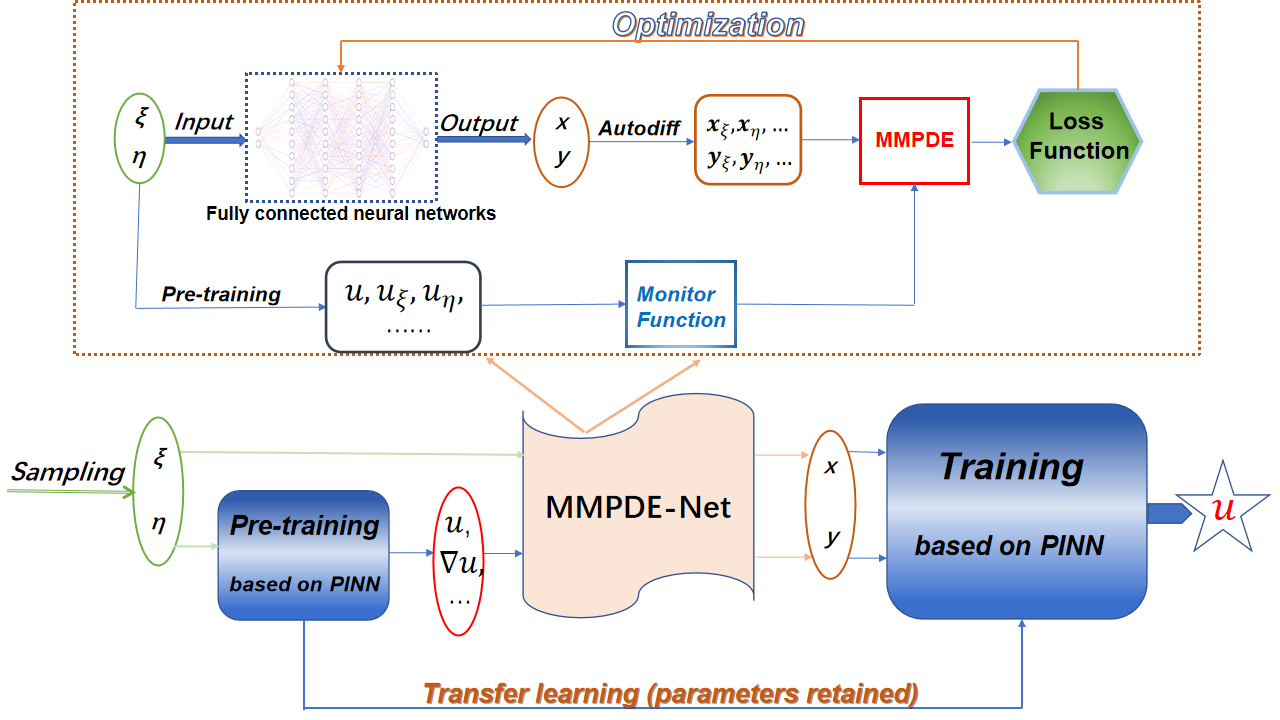}
\caption{Flow chart of MS-PINN.}
\label{fig:MMPDE-PINN}
\end{figure}

MS-PINN is described in Algorithm \ref{alg:MMPDE-PINN} for the two-dimensional problem, and the corresponding flowchart is given in Fig \ref{fig:MMPDE-PINN}. In fact, the iterative algorithm in Section \ref{sec:MMPDE-Net-Iterations} can be applied to MS-PINN as well, by simply changing the single MMPDE-Net training in Algorithm \ref{alg:MMPDE-PINN} to multiple iterations of MMPDE-Net training (like Algorithm \ref{alg:MMPDE-Net-iterations}). We will not discuss further about this part, more details can be found in the numerical examples in Section \ref{sec:MS-PINN with the iterative MMPDE-Net algorithm}.
	% @@@@@@@@@@@@@@@@@@@@@@@@@@
 	% @@@@@@@@@@@@@@@@@@@@@@@@@@
\subsection{Error Estimation}
\label{sec:Error estimation}
In this section, we will discuss the error estimate to verify the MS-PINN.
%we will give some mathematical analysis to prove the effectiveness of the algorithm for our proposed MS-PINN.
We consider the problem Eq \eqref{eq:gen_PDE} and assume that both $\mathcal{A}$ and $\mathcal{B}$ are linear operators. %The loss function of the conventional PINN is defined in Eq \eqref{eq:L2_empiricalloss}, which consists of a residual loss term and a boundary loss term. Since the boundary points are not involved in our proposed algorithm in Section \ref{sec:Implement of MMPDE-Net}, 
%we only discuss the residual loss term.
According to the description in Section \ref{sec:MMPDE-Net}, traditional PINN usually samples collocation points based on a uniform distribution $U(\bx)$, and the collocation points output by MMPDE-Net are subject to a new distribution, denoted by $\rho_{MM}(\bx)$.
In order to derive the error estimation for deep learning solvers, Assumption \ref{asu:rho} for MMPDE-Net is given.
%Since our later proofs focus on error estimation for deep learning solvers, we directly give Assumption \ref{asu:rho} for MMPDE-Net.

\begin{assumption}
   \label{asu:rho}
   The optimization algorithm always finds the optimal parameters $\theta$ to minimize the loss function, i.e., $\min\limits_{\theta} \mathcal{L}oss_{MMPDE}(\xi,\eta;\theta)$.
\end{assumption}

%The $\Vert  \cdot \Vert_{2}$ norm does not meet our needs now, and we need a new norm to reflect 
Since the sampling distribution is very important, we define the norm $\Vert \cdot \Vert_{2,\rho}$ as follows 

\begin{equation}\label{eq:L_2rho}
     \Vert  f(\bx) \Vert_{2,\rho} = \left(\int_{\Omega} |f(\bx)|^2 \rho(\bx)d\bx \right)^{\frac{1}{2}}, \quad \bx \in \Omega, 
\end{equation}
where $f(\bx) \in L^2(\Omega)$, $\rho(\bx)$ is a PDF. 
%and $\int_{\Omega}\rho(\bx)d\bx = 1 $ , $\rho(\bx) \geq 0$.
We now consider the PINN loss function as follows
%under the $\Vert  \cdot \Vert_{2,\rho}$ norm.

\begin{equation}\label{eq:L2rho_loss}
    \left\{
    \begin{aligned}
         &\mathcal{L}(\bx;\theta) =  \alpha_1\mathcal{L}_{res}(\bx;\theta) + \alpha_2\mathcal{L}_{bou}(\bx;\theta), \\
         &\mathcal{L}_{res}(\bx;\theta) = \Vert r(\bx;\theta)\Vert_{2,\rho,\Omega}^2
         =\int_{\Omega} |r(\bx;\theta)|^2 \rho(\bx) d\bx,\quad \bx \in \Omega,\\
         &\mathcal{L}_{bou}(\bx;\theta)= \Vert b(\bx;\theta)\Vert_{2,\rho,\partial\Omega}^2
         =\int_{\partial\Omega} |b(\bx;\theta)|^2 \rho(\bx) d\bx,\quad \bx \in \partial\Omega.\\
     \end{aligned}
     \right.
\end{equation}
Without loss of generality, in all subsequent derivations, we choose $\alpha_1=\alpha_2=1$. 
%Now, let's discuss the residual loss term $\mathcal{L}_{res}(\bx;\theta)$ and 
We use Monte Carlo method to approximate the residual loss term $\mathcal{L}_{res}(\bx;\theta)$, which is
\begin{equation}\label{eq:L2rho_lossres}
    \begin{aligned}
         \mathcal{L}_{res}(\bx;\theta)
         = \int_{\Omega} |r(\bx;\theta)|^2 \rho(\bx) d\bx
         = \mathbb{E}_{\bx \sim \rho} \left[|r(\bx;\theta)|^2 \right]
         \approx\frac{1}{M_{r}}\sum_{i = 1}^{M_{r}} |r(\bx_i;\theta)|^2,
         \quad \bx \sim \rho(\bx).
     \end{aligned}
\end{equation}
%where$M_r$ is the number of sampling points.
The discrete form of the $\Vert  \cdot \Vert_{2,\rho}$ norm is defined as

\begin{equation}\label{eq:eq:L_2rho_discrete}
    \Vert  r(\bx;\theta) \Vert_{2,\rho,M_r}^2 = \frac{1}{M_{r}}\sum_{i = 1}^{M_{r}} |r(\bx_i;\theta)|^2, \quad \bx \sim \rho(\bx).
\end{equation}

Now we will prove that $\Vert  r(\bx;\theta) \Vert_{2,\rho_{MM},M_r}^2 \leq \Vert  r(\bx;\theta) \Vert_{2,U,M_r}^2$.
%by the definitions and assumptions already given.
%Before the proof, we introduce some notations. 
Without loss of generality, we divide the computational domain $\Omega$ into two parts based on the complexity of the function, i.e., $\Omega = \Omega_1 +\Omega_2$,  $k_1 < k_2$, which denote the complexity of $\Omega_1$ and $\Omega_2$ respectively. We suppose $\vert \Omega_1 \vert = \int_{\Omega_1} d\bx > \vert \Omega_2 \vert =\int_{\Omega_2} d\bx$ and $M_r$ is the total number of sampling points in $\Omega$. If the sampling points are chosen from the uniform distribution, the numbers of sampling points in $\Omega_1$ and $\Omega_2$ are $\frac{b_1}{2} M_r$ and $\frac{b_2}{2}M_r \in \mathbb{N^+}$ respectively, where $b_1 + b_2 =2$. After training by MMPDE-Net, the number of sampling points in the two subregions changes to $\frac{a_1}{2} M_r$ and $\frac{a_2}{2}M_r \in \mathbb{N^+}$, respectively, where $a_1 + a_2 =2$.

\begin{theorem}
\label{thm:1}
Suppose that there exists $\hat{\theta}$ such that $\Vert  r(\bx;\hat{\theta}) \Vert_{2,\rho_{MM},M_r}^2 = \mathbb{E}_\theta\left[\Vert  r(\bx;\theta) \Vert_{2,\rho_{MM},M_r}^2\right]$, under Assumption \ref{asu:F} and \ref{asu:rho}, we have

$$ \Vert  r(\bx;\hat{\theta}) \Vert_{2,\rho_{MM},M_{r}}^2 \leq \Vert  r(\bx;\hat{\theta}) \Vert_{2,U,M_{r}}^2.$$

%if   $ \{0< b_2 < 1 < b_1 < 2\}$  or      $ \{0< b_1 \leq 1 \leq  b_2 < 2$ %\ and \ $(\frac{b_1}{b_2})^q \geq (\frac{k_1}{k_2})^p \}$ .

% If \ $ 0< b_1 \leq 1 \leq  b_2 < 2$ \ and \ $(\frac{b_1}{b_2})^q \geq (\frac{k_1}{k_2})^p$,  there always exist 1 and a2 such that 
% $$\mathcal{L}^{MM}_{r,M_r} = \Vert  r(x_i;\theta(\bk,M_r)) \Vert_{2,M_{r},\rho}^2 \leq \Vert  r(x_i;\theta(\bk,M_r)) \Vert_{2,M_{r},U}^2 = \mathcal{L}_{r,M_r}$$
\end{theorem}
\textcolor{black}{
\begin{proof}
See Appendix.
\end{proof}
}

Since $\mathcal{A}$ and $\mathcal{B}$ in Eq \eqref{eq:gen_PDE} are supposed to be linear operators, similar as in \cite{das-pinns}, the following assumption is given, which is the stability bound.

\begin{assumption}
    \label{asu:Normrelations}
%Assume that $\forall \bu \in L^2(\Omega)$ satisfy
  Let  $\mathbb{H}$ be a Hilbert space, the following assumption holds
    \begin{equation}\label{eq:Normrelations}
        C_1 \Vert \bu \Vert_{2,\rho} \leq \Vert \mathcal{A}(\bu) \Vert_{2,\rho,\Omega} +  \Vert \mathcal{B}(\bu) \Vert_{2,\rho,\partial\Omega} \leq C_2 \Vert \bu \Vert_{2,\rho}, \  \forall \bu \in \mathbb{H},
    \end{equation}
    where $\mathbb{H}$ is defined on $\Omega$, $C_1$ and $C_2$ are positive constants.
    
\end{assumption}

In order to give an upper bound of the error, we need to introduce the the Rademacher complexity.

\begin{definition}\label{de:Rademacher}
    Let $X := \{x_i\}_{i=1}^{n}$ be a collection of i.i.d. random samples, the Rademacher complexity of the function class $\mathcal{F}$ can be defined as 
    $$\mathcal{R}_n(\mathcal{F}) = \mathbb{E}_{X,\epsilon}\left( \sup_{f \in \mathcal{F}} \vert \frac{1}{n} \sum_{i=1}^{n} \epsilon_i f(x_i)\vert \right),$$
    where $(\epsilon_1,\epsilon_2,...,\epsilon_n)$ are i.i.d. Rademacher variables, i.e., $P(\epsilon_i=1) =  P(\epsilon_i=-1)=0.5 $.

\end{definition}

% Let $\bx = \{x_i\}_{i=1}^{M_r}$ be a collection of i.i.d. random samples from the PDF $\rho_{MM}$, which is regarded as the input to the neural networks.  
Let $\bu^*$ be the exact solution of Eq \eqref{eq:gen_PDE}, same as in \cite{shin2023error}, it is reasonable to assume that
\begin{assumption}
    \label{asu:r-boundary}
    There exists a positive number $0 < b \in \mathbb{R}$ such that
    \begin{equation}\label{eq:r-boundary}
    \sup_{i} \Vert  \mathcal{A}[\bu_i(\bx)] - f(x) \Vert_{\infty} = b 
    \end{equation}
     where $\left\{\bu_i(\bx)\right\}_{i\geq1} \subset \mathbb{H}$  is a sequence of approximate solutions , i.e., $\lim\limits_{i \to +\infty} \Vert  \bu_i(\bx) - \bu^*(\bx) \Vert_{2,\rho_{MM}} = 0$.     
\end{assumption}

Now we consider the approximate solution of the neural networks $\mathcal{F}^{NN}_u = \left\{\bu(\bx;\theta) : \theta \in \mathbb{R}^{NN} \right\}$. Similarly as in \cite{shin2023error}, using Assumption \ref{asu:r-boundary} and the Universal Approximation Theorem (\cite{hornik1989multilayer},\cite{hornik1990universal}), the following function class can be defined 
$$\mathcal{F}_r  = \left\{ \bu(\bx;\theta): \bu(\bx;\theta) \in \mathbb{H} \cap \mathcal{F}^{NN}_u \ \mbox{and} \ \Vert \mathcal{A}[\bu(\bx;\theta)]-f(\bx)  \Vert_{\infty,\Omega} \leq b \right\},$$ 
which is not empty.

%  Let $\{x_i\}_{i=1}^{M_r}$ be a collection of i.i.d. random samples from the PDF $\rho_{MM}$, which is regarded as the input to the neural networks. Due to the randomness of $\theta$, there is a class of approximate solutions $\mathcal{F}_u = \left\{\bu_1(\bx;\theta),\bu_2(\bx;\theta),...,\bu_n(\bx;\theta),...\right\}$. We define the function class $\mathcal{F}_r  = \left\{\vert r_1(\bx;\theta)\vert^2,\vert r_2(\bx;\theta)\vert^2,...,\vert r_n(\bx;\theta)\vert^2,...\right\}$
% and give Assumption \ref{asu:r-boundary}, where $r_i(\bx;\theta)=\mathcal{A}[\bu_i(\bx;\theta)]-f(\bx), \forall \bu_i(\bx;\theta)  \in \mathcal{F}_u$.
% \begin{assumption}
%     \label{asu:r-boundary}
%     The function class $\mathcal{F}_r$ is a b-uniformly bounded class, that is, there exists a positive number $0 < b \in \mathbb{R}$ such that
%     \begin{equation}\label{eq:r-boundary}
%         \vert r_i(\bx;\theta) \vert^2 \leq  b, \quad  \forall \vert r_i(\bx;\theta) \vert^2 \in \mathcal{F}_r, \   \forall \bx  \in \Omega.
%     \end{equation}

% \end{assumption}

According to \cite{wainwright2019high}, we use Lemma \ref{lem:Uniform laws of large numbers-Rademacher} and propose the following theorem about the Rademacher complexity.
\begin{lemma}
    \label{lem:Uniform laws of large numbers-Rademacher}
    Let $\mathcal{F}$ be a b-uniformly bounded class of integrable real-valued functions with domain $\Omega$, for any postive integer $n \geq 1$ and any scalar $\delta \geq 0$, we have
    $$\sup_{f \in \mathcal{F}} \left\vert \frac{1}{n} \sum_{i=1}^{n} f(x_i) - \mathbb{E}[f(x)]\right \vert \leq 2\mathcal{R}_n(\mathcal{F}) + \delta, $$
    with probability at least $1-exp\left(-\frac{n\delta^2}{2b^2}\right)$. 
    
\end{lemma}

\begin{theorem}
\label{thm:2}
Suppose that Assumption \ref{asu:Normrelations} and Assumption \ref{asu:r-boundary} hold, let $\bu^*$ be the exact solution of Eq \eqref{eq:gen_PDE}. Given $\delta > 0$, $\{x_i\}_{i=1}^{M_r}$ is a collection of i.i.d. random samples from the PDF $\rho_{MM}$
and $\forall \bu(\bx;\theta) \in \mathcal{F}_r$, the following error estimate holds

\begin{equation}\label{eq:errorestimate}
    \Vert \bu^*(\bx)-\bu(\bx;\theta) \Vert_{2,\rho_{MM}} \leq  \frac{\sqrt{2}}{C_1} \left(\Vert  r(\bx;\theta) \Vert_{2,\rho_{MM},M_r} ^2 + 
     \Vert b(\bx;\theta)\Vert_{2,\rho_{MM},\partial\Omega}^2 + 2\mathcal{R}_{M_r}(\mathcal{F}_r) + \delta \right)^{\frac{1}{2}},
\end{equation}
with probability at least $1-\exp\left(-\frac{M_r\delta^2}{2b^2}\right)$.
\end{theorem}
\textcolor{black}{
\begin{proof}
See Appendix.
\end{proof}
}

\begin{corollary}\label{cor:1}
    Suppose that Theorem \ref{thm:1} and Theorem \ref{thm:2} hold, then 
    \begin{equation}\label{eq:errorestimate_cor}
    \begin{aligned}
     \Vert \bu^*(\bx)-\bu(\bx;\hat{\theta}) \Vert_{2,\rho_{MM}}
     &\leq \frac{\sqrt{2}}{C_1} \left(\Vert  r(\bx;\hat{\theta}) \Vert_{2,\rho_{MM},M_r} ^2 + 
     \Vert b(\bx;\hat{\theta})\Vert_{2,\rho_{MM},\partial\Omega}^2 + 2\mathcal{R}_{M_r}(\mathcal{F}_r) + \delta \right)^{\frac{1}{2}} \\
     &\leq \frac{\sqrt{2}}{C_1} \left(\Vert  r(\bx;\hat{\theta}) \Vert_{2,U,M_r} ^2 + 
     \Vert b(\bx;\hat{\theta})\Vert_{2,U,\partial\Omega}^2 + 2\mathcal{R}_{M_r}(\mathcal{F}_r) + \delta \right)^{\frac{1}{2}},
     \end{aligned} 
\end{equation}
    with probability at least $1-\exp\left(-\frac{M_r\delta^2}{2b^2}\right)$.
\end{corollary}
\textcolor{black}{
\begin{proof}
See Appendix.
\end{proof}
}
It is observed from Corollary \ref{cor:1} that 
the upper bound of the error in $\Vert \cdot \Vert_{2,\rho_{MM}}$ can be reduced by MMPDE-Net with probability at least $1-\exp\left(-\frac{M_r\delta^2}{2b^2}\right)$.

%using MMPDE-Net reduces the upper bound on the error under the $\Vert \cdot \Vert_{2,\rho_{MM}}$ Norm.
% % @@@@@@@@@@@@@@@@@@@@@@@@@@

% @@@@@@@@@@@@@@@@@@@@@@@@@@
\section{Numerical Experiments}
\label{sec:Numericalsurvey}

In this section, numerical experiments are presented to verify our method. A  two-dimensional Poisson equation with low regularity, which has one peak is discussed in Section \ref{sec:OnePeak_2D}. We compare our method with other methods and show the performance of MS-PINN with iterative MMPDE-Net. %The first example in Section \ref{sec:OnePeak_2D} is a Also, we compare our method with other methods and give the performance of MS-PINN using iterative MMPDE-Net. %The example in Section \ref{sec:DoublePeaks_2D} is still the two-dimensional Poisson equation, but with two peaks and more difficult to learn.
In Section \ref{sec:DoublePeaks_2D}, numerical solutions of two-dimensional Poisson equation with two peaks are presented.
Numerical results in Section \ref{sec:Burgers_1D} show that our algorithm is effective for solving Burgers equation, whose solution has large gradient, including both forward (Section\ref{sec:Burgers_1D_Forward}) and inverse (Section \ref{sec:Burgers_1D_Inverse}) problems. The last example in Section \ref{sec:Burgers_2D} is the two-dimensional Burgers problem, which has shock wave solution in finite time.
%the one-dimensional Burgers equation is used as the third example, and we would like to observe whether our algorithm is still effective in regions of sharp gradients.

\subsection{Symbols and parameter settings}
\label{sec:Symbols and parameter settings}
%In this section, we introduce some notation and discuss how to measure the output of neural networks.
The analytical solution or reference solution is denoted by $u^*$.
% We use $u^*$ to denote the true solution.
%and $u^{\text{Ref}}$ to denote the true and reference solution.
The $u^{\text{PINN}}$ and $u^{\text{MS-PINN}}$ are used to denote the approximate solutions obtained by PINN and MSPINN respectively. 
%Taking the true solution (or the reference solution) as the goal, we define the following relative error in infinite norm as the error measure (Eq.\ref{eq:measure}) in this paper.
According to Eq \eqref{eq:eq:L_2rho_discrete}, the relative errors in the test set $\left\{\bx_i\right\}_{i=1}^{M_t}$ sampled from the equispaced uniform distribution are defined as follows
\begin{equation}
    \label{eq:measure}
    \hspace{-0.3cm}
    \begin{array}{r@{}l}
        \begin{aligned}
            e_\infty(u) &= \frac{\Vert u^* - u\Vert_{\infty,U,M_{t}}}{\Vert u^*\Vert_{\infty,U,M_{t}}},\\
            e_2(u) &= \frac{\Vert u^* - u\Vert_{2,U,M_{t}}}{\Vert u^*\Vert_{2,U,M_{t}}}.
        \end{aligned}
    \end{array}
\end{equation}
Parameter settings in MS-PINN are given in Table \ref{tab:parameter-MMPDENet} if we do not specify otherwise. For 
Poisson equation and forward Burgers equation, the Adam method is used to do the optimization. For the inverse Burgers equation and two dimensional Poisson equation, the Adam method is implemented firstly and then the LBFGS method is used to optimize the loss function. Since there is three stages in MS-PINN, we have given the training epochs in two stages in Table \ref{tab:parameter-MMPDENet}, while the training epochs in MMPDE-Net is specified in each case. All simulations are carried on NVIDIA A100(80G).
 %It is worth noting that there is no doubt that we would have achieved much nicer results if we had adjusted the parameters of MMPDE-Net separately for each of the examples in the later sections. However, in order to demonstrate the effectiveness of our algorithm and to minimize the impact of the parameters, we did not do so.  
 %MMPDE-Net in all the examples uses the parameters in Table \ref{tab:parameter-MMPDENet}. Note that, 

% \begin{table}
% 	\centering
% 	\caption{Key parameter settings in MS-PINN.}
% 	\label{tab:parameter-MMPDENet} 
% 	\setlength{\tabcolsep}{3.mm}{
% 		\begin{tabular}{|c|c|c|c|c|c|c|c|}%{p{1cm}p{2.2cm}p{1.2cm}p{1.2cm}p{1.2cm}p{1.2cm}p{1.2cm}p{1.6cm}}
% 			\hline\noalign{\smallskip}
% 			 Activation   &  Initialization  & Optimization& learning  &MMPDE-Net & Poisson & Burgers\\
% 			 function  & method   & method &  rate &Size & PINN Size& PINN Size\\
% 			\hline
% 			 tanh  & Xavier & Adam/LBFGS & 0.0001/1 &$20 \times 8$ &$40 \times 4$&$20 \times 8$\\ 
% 			\hline

% 		\end{tabular}
% 	}
% \end{table}

\begin{table}
	\centering
	\caption{Default settings for main parameters in MS-PINN.}
	\label{tab:parameter-MMPDENet} 
	\setlength{\tabcolsep}{1.mm}{
		\begin{tabular}{|c|c|c|c|c|c|}%{p{1cm}p{2.2cm}p{1.2cm}p{1.2cm}p{1.2cm}p{1.2cm}p{1.2cm}p{1.6cm}}
			\hline\noalign{\smallskip}
			 Torch  & Activation &  Initialization  & Optimization& Learning&  Eq \eqref{eq:MMPDE_loss}  \\
			 Version&  Function  &         Method   &       Method &  Rate & $\tau$ \\
			\hline
			2.0.1 &tanh  & Xavier & Adam/LBFGS & 0.0001/1 &0.1\\
			\hline
            MMPDE-Net &  Poisson & Burgers   & MMPDE-Net&Pre-     &  Formal \\
            Size      & PINN Size& PINN Size & Training&Training &  Training\\
            (neurons/layer  & (neurons/layer  & (neurons/layer & (epochs)&  (epochs) & (epochs)\\
            $\times$ hidden layers) & $\times$ hidden layers) & $\times$ hidden layers) &   &  &\\
            \hline
            $20 \times 8$  & $40 \times 4$&$20 \times 8$&20000 &20000 & 40000 \\
            \hline
		\end{tabular}
	}
\end{table}
 
% @@@@@@@@@@@@@@@@@@@@@@@@@@
\subsection{Two-dimensional Poisson equation with one peak}
\label{sec:OnePeak_2D}
For the following two-dimensional Poisson equation
\begin{equation}
	\label{eq:2d_Poisson}
	\hspace{-0.3cm}
	\begin{array}{r@{}l}
		\left\{
		\begin{aligned}
			 -\Delta u(x,y) & = f(x,y), \quad (x,y) \ \mbox{in} \ \Omega, \\
   %= [-1,1]^2, \\
                  u(x,y) & = g(x,y),  \quad  (x,y) \ \mbox{on} \  \partial \Omega,
%			& u(x,-1) = u(x,1) = e^{-1000(x^2+1)}, \quad  x \in [-1,1] \\ 
%			& u(-1,y) = u(1,y) = e^{-1000(1+y^2)}, \quad  y \in [-1,1]\\   
		\end{aligned}
		\right.
	\end{array}
\end{equation}
where $\Omega = (-1,1)^2$,  the analytical solution with one peak at $(0,0)$ is chosen as
\begin{equation}
    \label{eq:2d_Peaksolution}
    \hspace{-0.3cm}
    \begin{array}{r@{}l}
        \begin{aligned}
            u = e^{-1000(x^2+y^2)}.
        \end{aligned}
    \end{array}
\end{equation}
The Dirichlet boundary condition $g(x,y)$ on $\partial \Omega$ and function $f(x,y)$ are given by the analytical solution.
\subsubsection{MS-PINN}
As shown in Fig \ref{fig:MMPDE-PINN}, we firstly use PINN for pre-training to obtain the preliminary solution information, secondly input it into MMPDE-Net together with the training points, finally input the new training points obtained from the second stage into PINN for the formal training. The monitor function in MMPDE-Net for Eq \eqref{eq:2d_Poisson} is chosen as 

\begin{equation}
    \label{eq:2d_monitorfunction}
    \hspace{-0.3cm}
    \begin{array}{r@{}l}
        \begin{aligned}
            w = \sqrt{1+ 1000u^2 + \lvert \nabla u \rvert^2}.
        \end{aligned}
    \end{array}
\end{equation}

In order to compare our method with other five methods, we sample $100\times100$ points in $\Omega$ and 1000 points on $\partial \Omega$ as the training set and $400\times400$ points as the test set for all the methods.
%where we use equispaced uniform sampling method.} 
%residual training set and test set from two uniform meshes. 
%On $\partial \Omega$, we also sample 300 uniformly distributed boundary training points.
%from the boundaries of the uniform mesh. Here are some parameter settings for PINN, which has 4 hidden layers with 40 neurons per layer.
%using tanh as the activation function. As for optimization, the Adam optimization method is used and the initial learning rate is 0.0001.
% In PINN, there is 4 hidden layers with 40 neurons per layer. 
In MS-PINN, we train  MMPDE-Net 20000 epochs, and the epochs of pre-training formal PINN training are shown in Table \ref{tab:parameter-MMPDENet}. Five different methods are implemented: (1) PINN that samples $100 \times 100$ uniformly distributed points and trains 60000 epochs; 
%training points from a uniform mesh
(2) PINN-RAR method that pre-trains 20000 epochs using $90 \times 90$ training points, then 1900 more training points are added by RAR method (\cite{lu2021deepxde},\cite{wu2023comprehensive}) and trains 40000 epochs; 
(3) PINN-RAD method that pre-trains 20000 epochs using $100 \times 100$ training points, then 10000 points are adaptively sampled by RAD method (\cite{wu2023comprehensive},\cite{nabian2021efficient},\cite{zapf2022investigating}) and trains 40000 epochs; 
\textcolor{black}{(4) PINN-ES method that pre-trains 20000 epochs using $100 \times 100$ training points, then 10000 points are adaptively sampled by evolutionary sampling method \cite{daw2022rethinking}  and trains 40000 epochs; 
(5) DAS-PINN method that pre-trains 20000 epochs using $90 \times 90$ training points, then 1900 points are added by DAS-G method \cite{das-pinns} and trains 40000 epochs. }

\begin{figure}[htbp]
\centering
\subfloat[analytical solution]{\includegraphics[width = 0.30\textwidth]{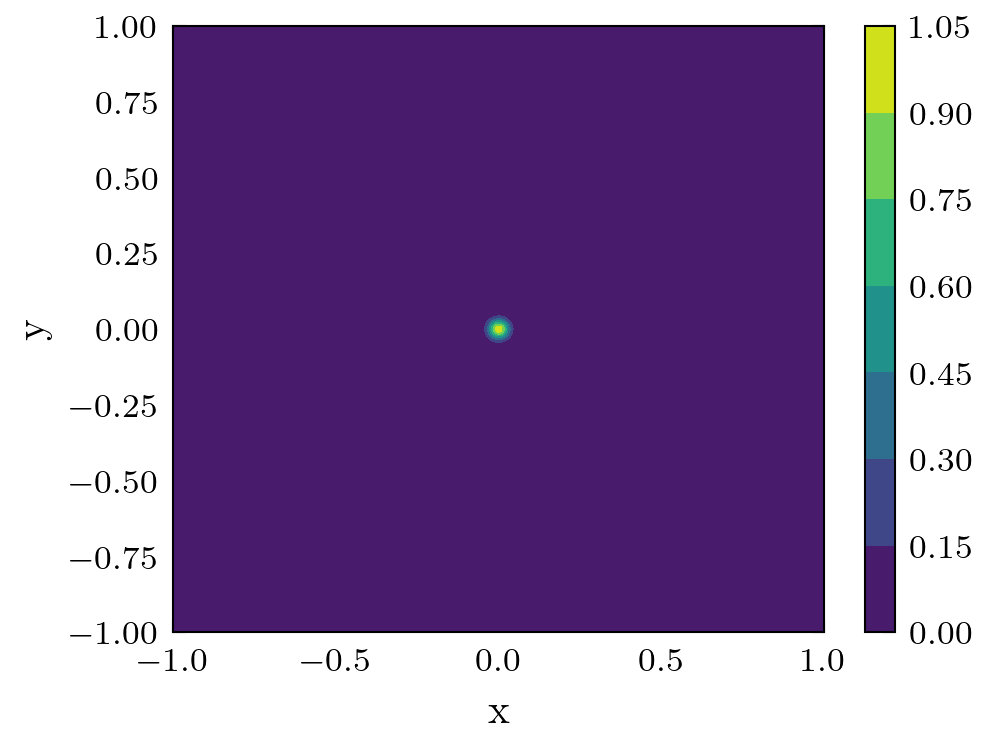}}
\subfloat[approximate solution]{\includegraphics[width = 0.30\textwidth]
{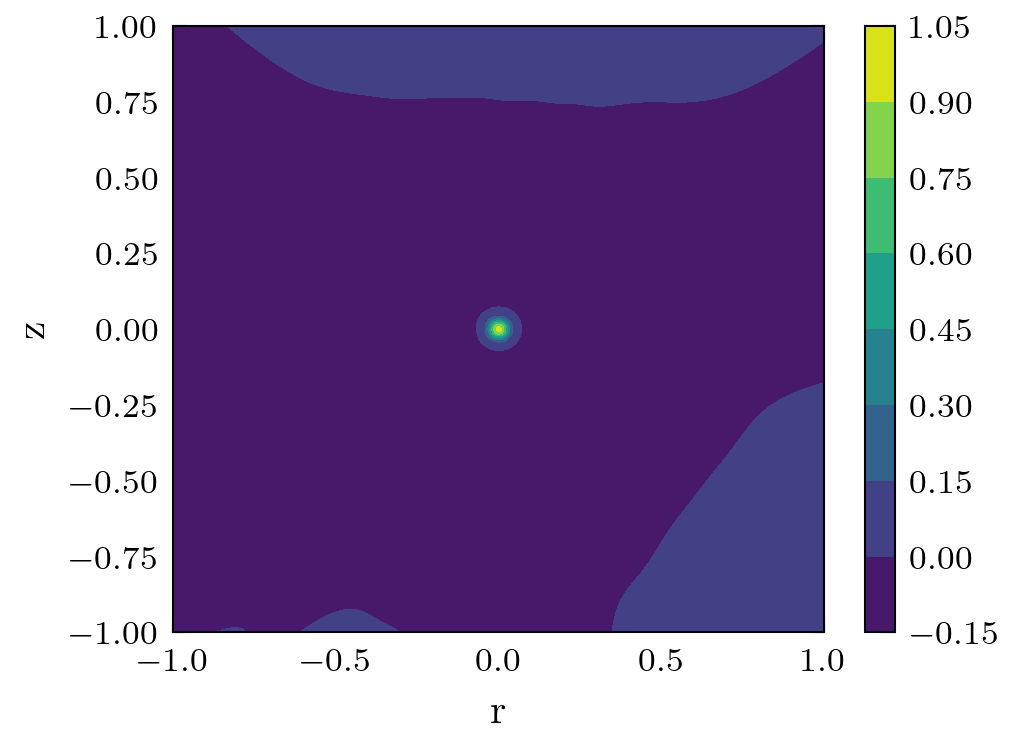}}
\subfloat[absolute error]{\includegraphics[width = 0.30\textwidth]{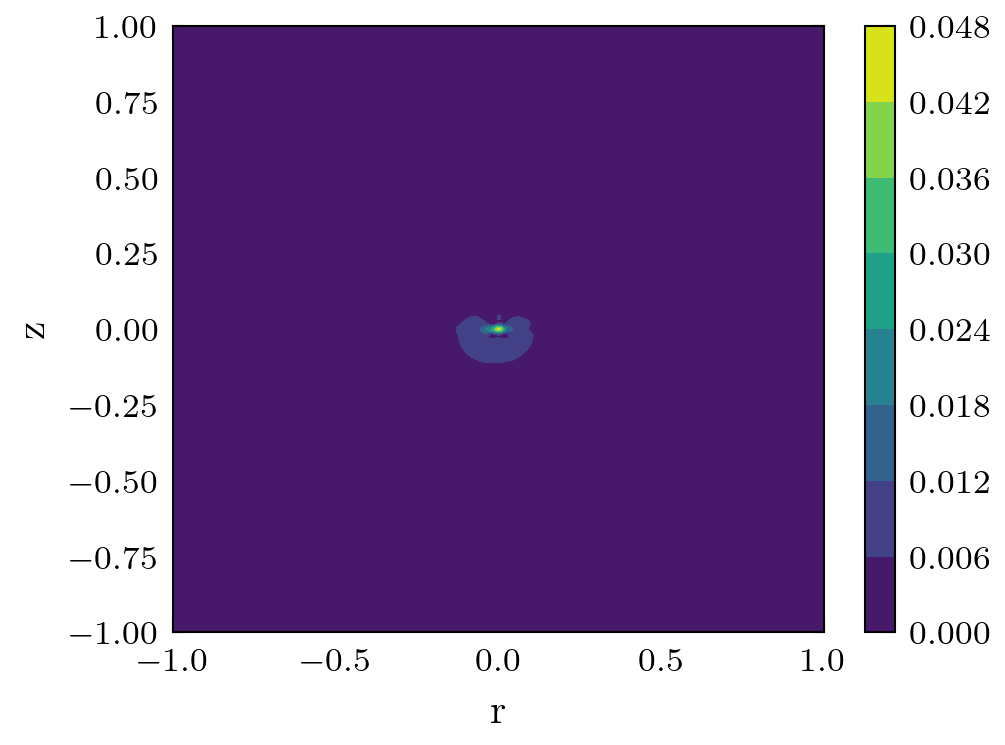}}
\caption{The result of PINN for the two-dimensional Poisson equation with one peak. (a) the analytical solution $u^*$ of the equation; (b) the approximate solution $u^{\text{PINN}}$ of PINN; (c) the absolute error $\vert u^* - u^{\text{PINN}} \vert$.}
\label{fig:PINN_Peak2D}
\end{figure}

\begin{figure}[htbp]
\centering
\subfloat[training points sampled by RAR]{\includegraphics[width = 0.27\textwidth]{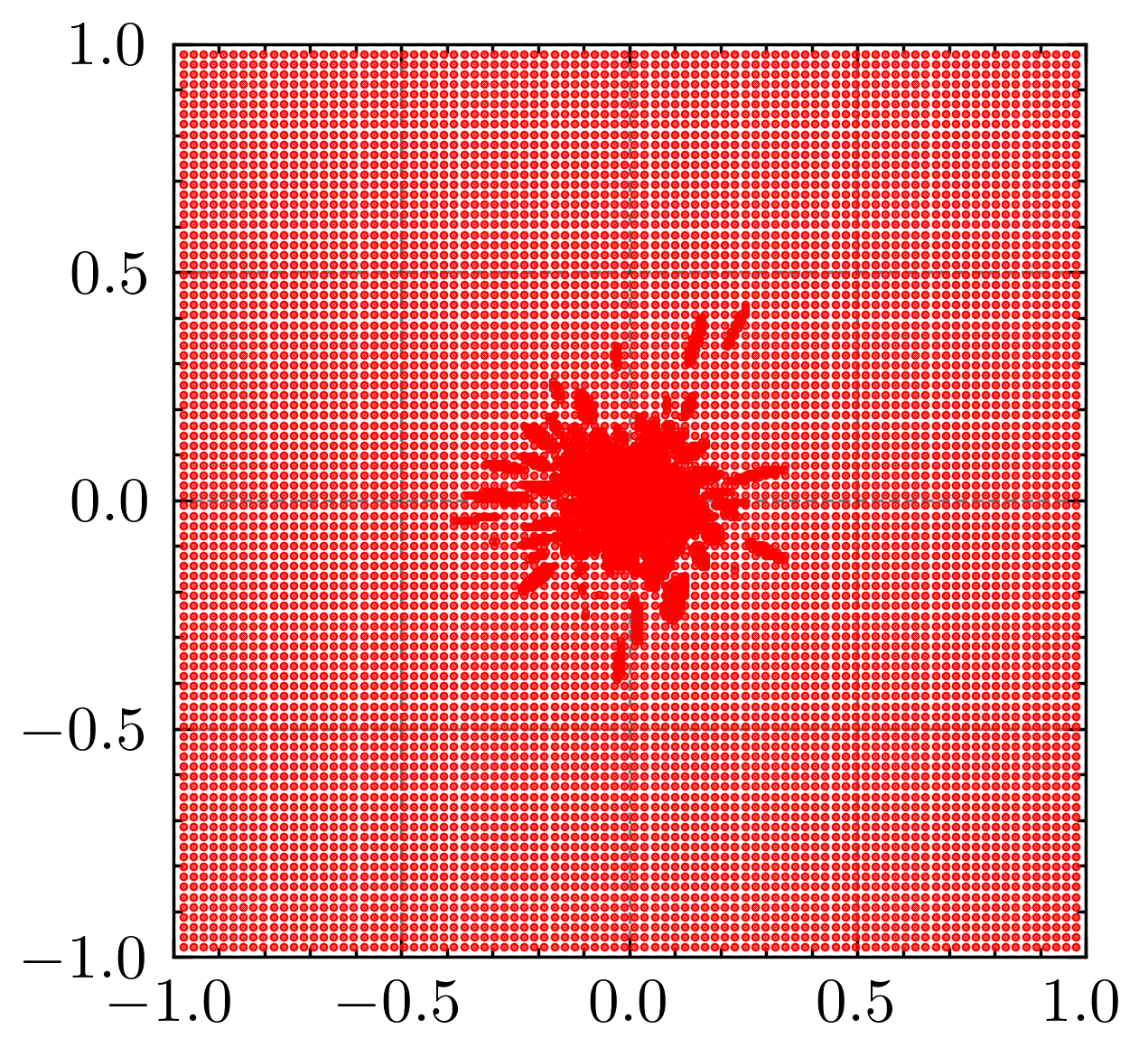}}
\subfloat[approximate solution]{\includegraphics[width = 0.31\textwidth]
{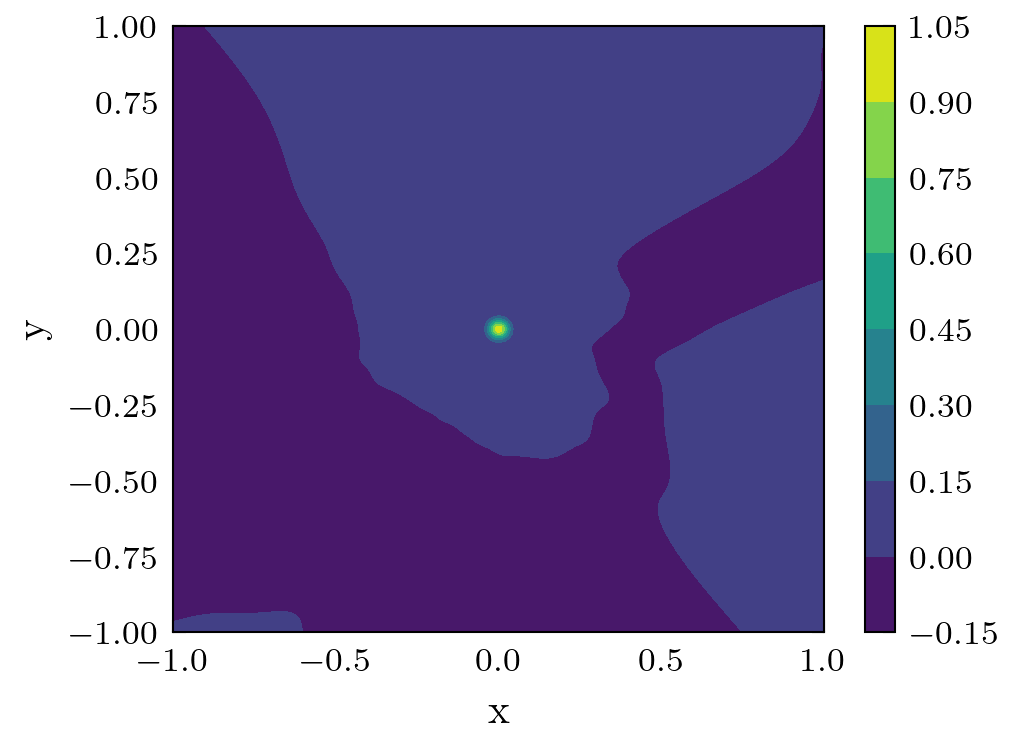}}
\subfloat[absolute error]{\includegraphics[width = 0.31\textwidth]{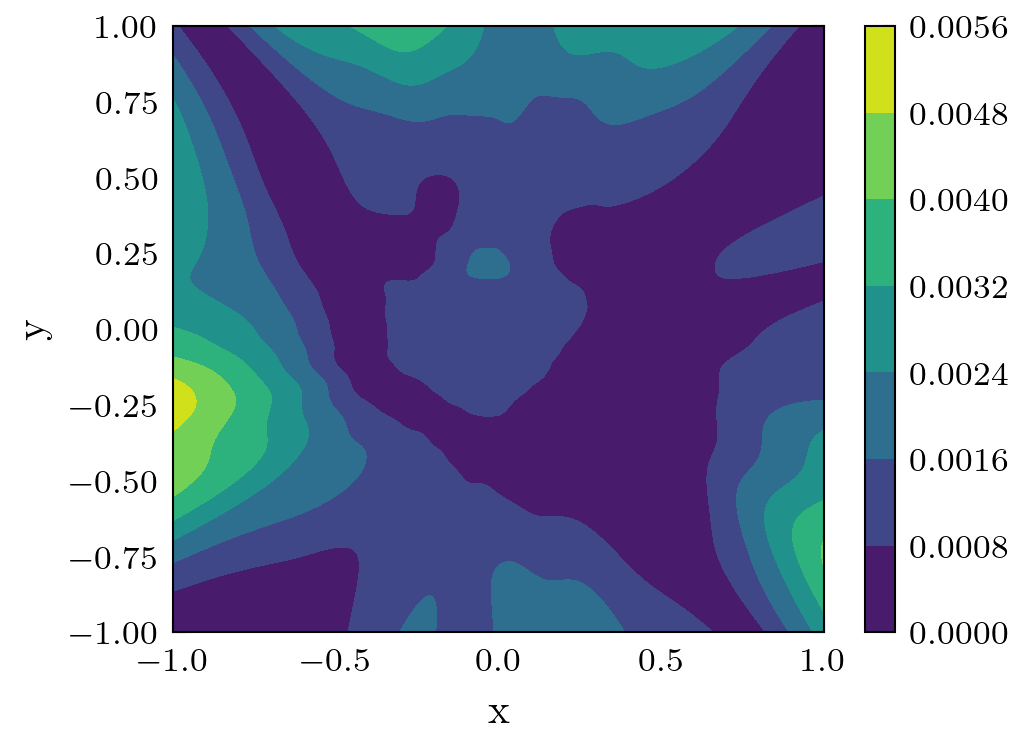}}
\caption{The result of PINN-RAR for the two-dimensional Poisson equation with one peak. (a) training points sampled by RAR method; (b) the approximate solution $u^{\text{PINN-RAR}}$; (c) the absolute error $\vert u^* - u^{\text{PINN-RAR}} \vert$.}
\label{fig:Peak2D_RAR}
\end{figure}

\begin{figure}[htbp]
\centering
\subfloat[training points sampled by RAD]{\includegraphics[width = 0.27\textwidth]{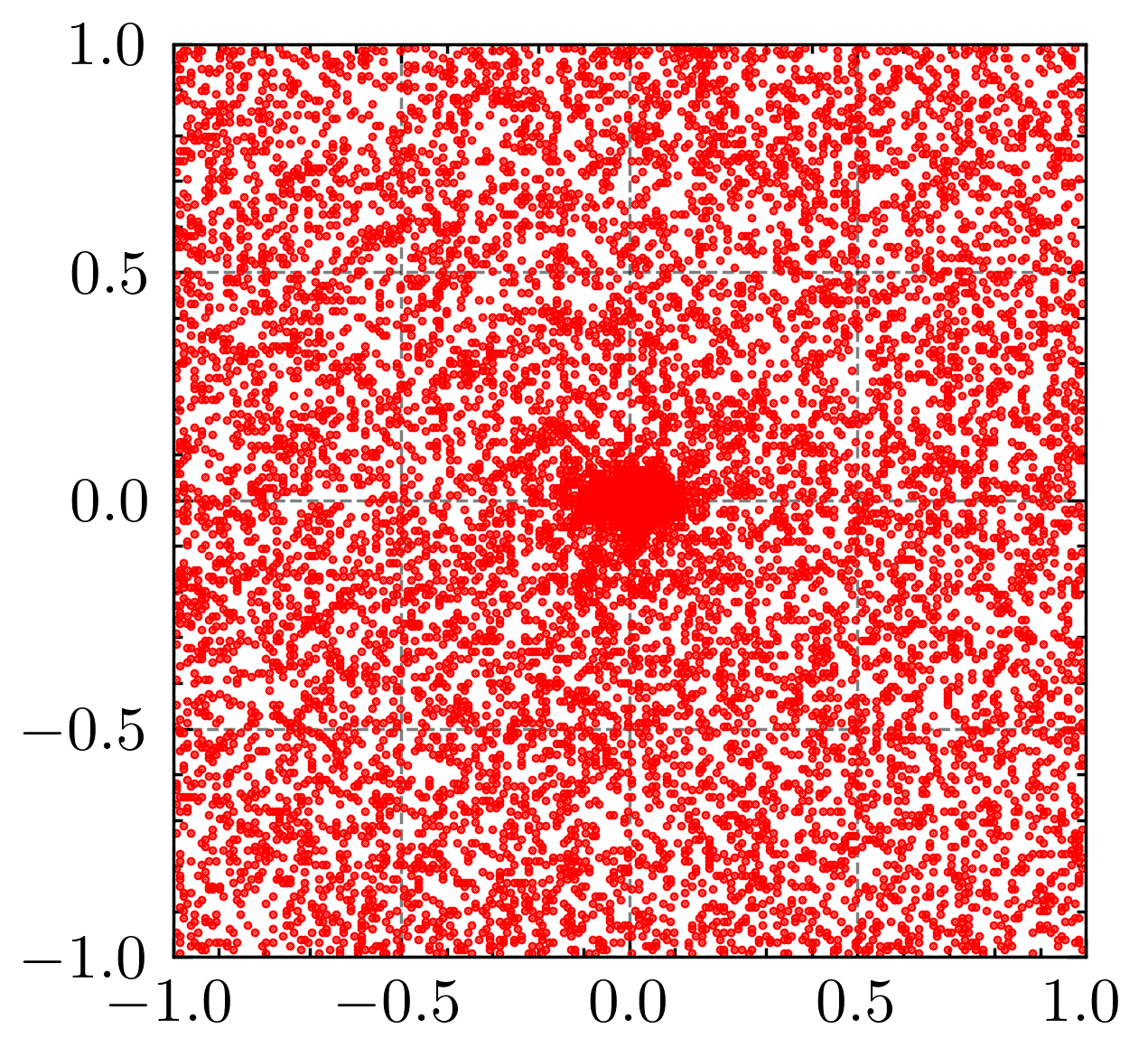}}
\subfloat[approximate solution]{\includegraphics[width = 0.31\textwidth]
{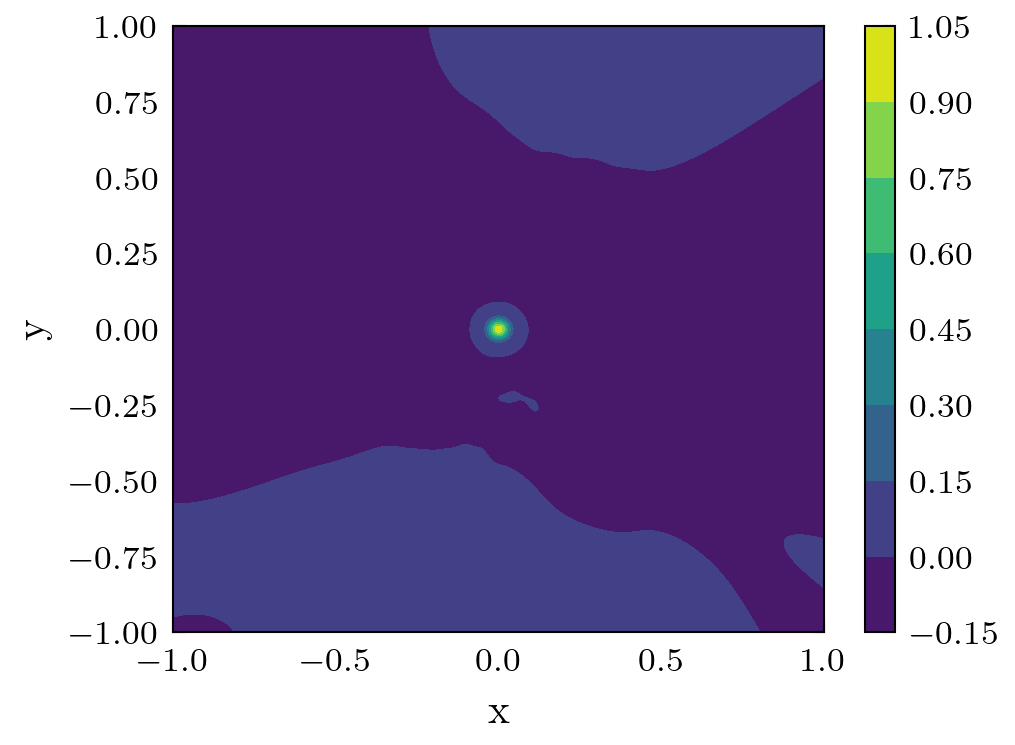}}
\subfloat[absolute error]{\includegraphics[width = 0.31\textwidth]{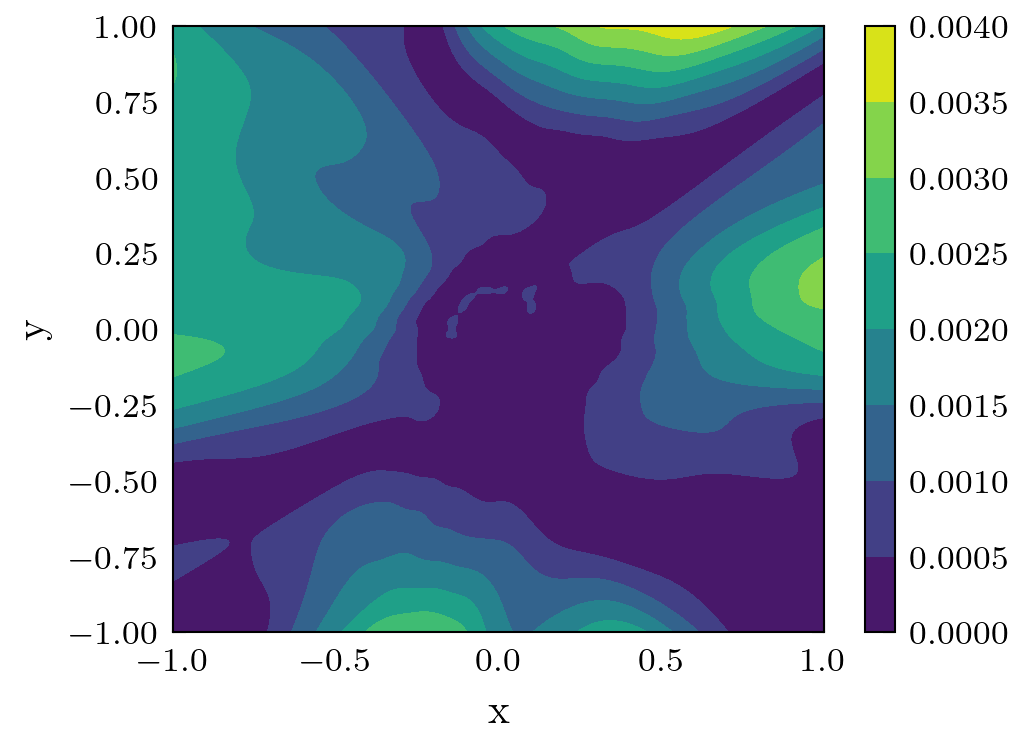}}
\caption{The result of PINN-RAD for the two-dimensional Poisson equation with one peak. (a) training points sampled by RAD method (k=1,c=1); (b) the approximate solution $u^{\text{PINN-RAD}}$; (c) the absolute error $\vert u^* - u^{\text{PINN-RAD}} \vert$.}
\label{fig:Peak2D_RAD}
\end{figure}

\begin{figure}[htbp]
\centering
\subfloat[training points sampled by ES]{\includegraphics[width = 0.26\textwidth]{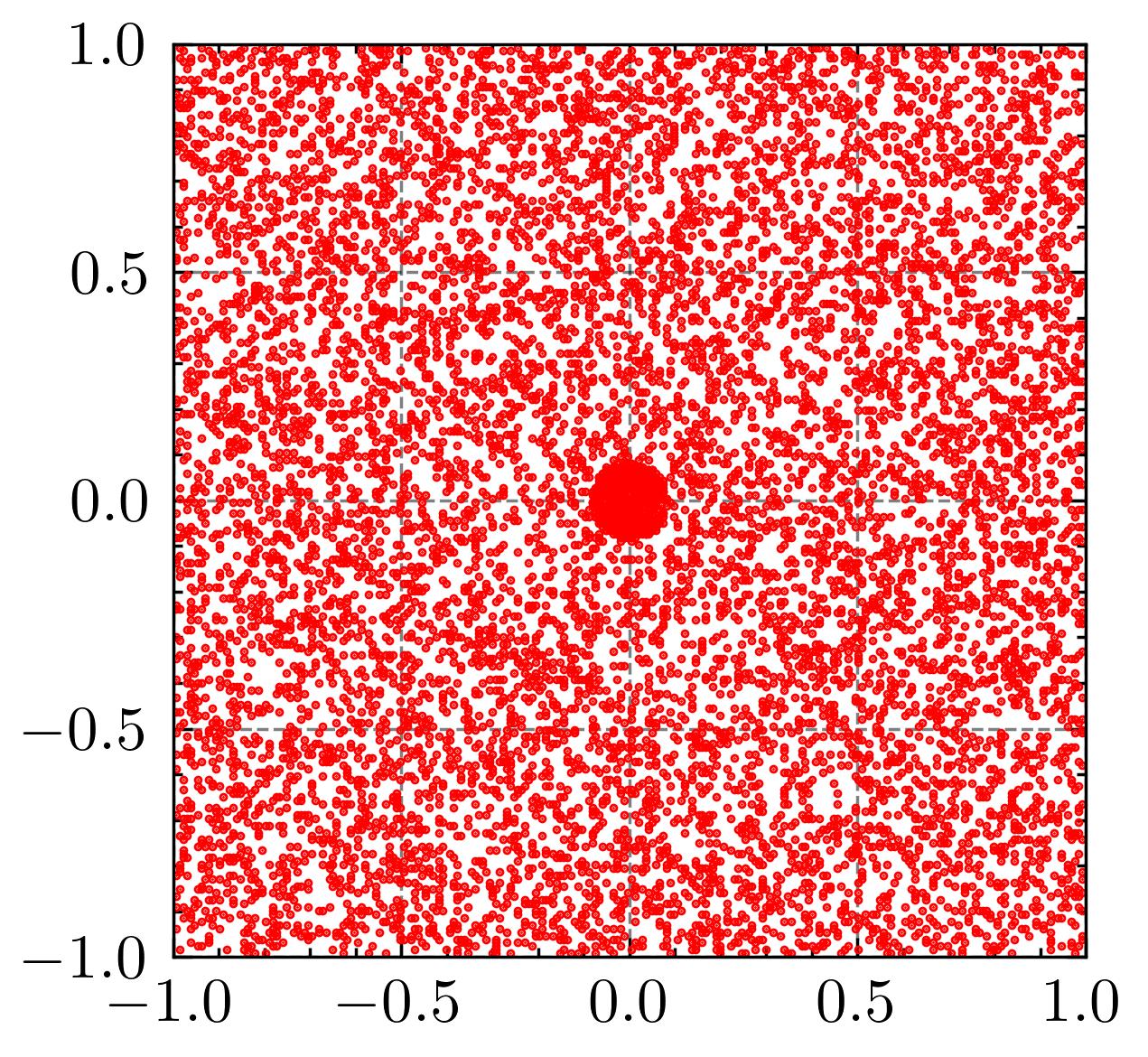}}
\subfloat[approximate solution]{\includegraphics[width = 0.31\textwidth]
{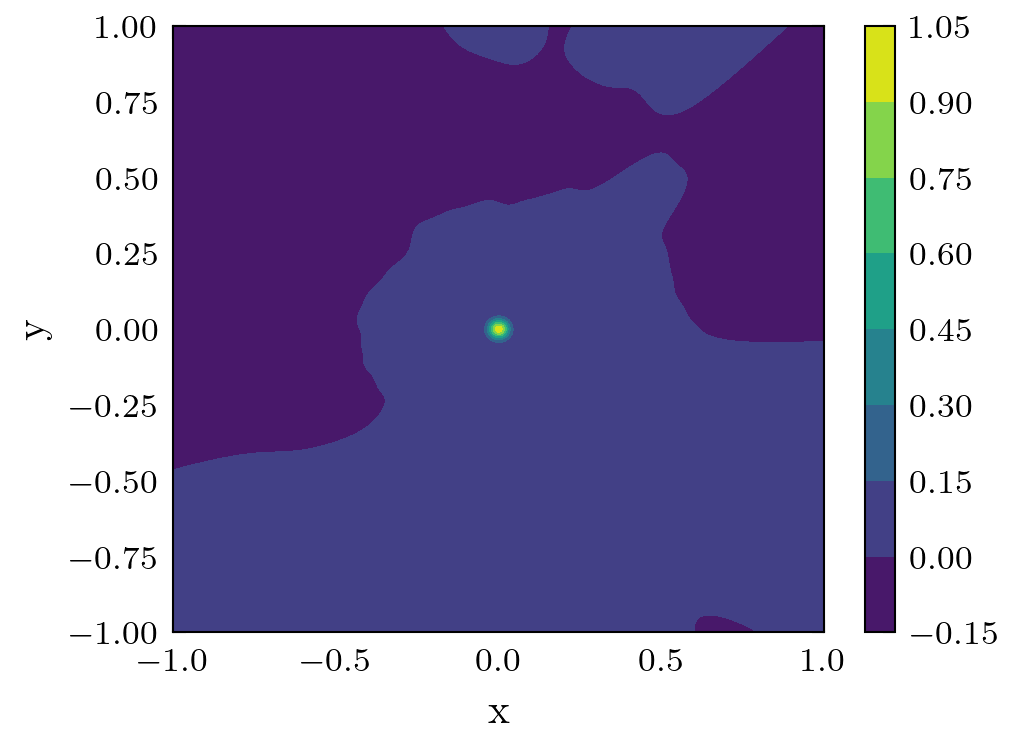}}
\subfloat[absolute error]{\includegraphics[width = 0.31\textwidth]{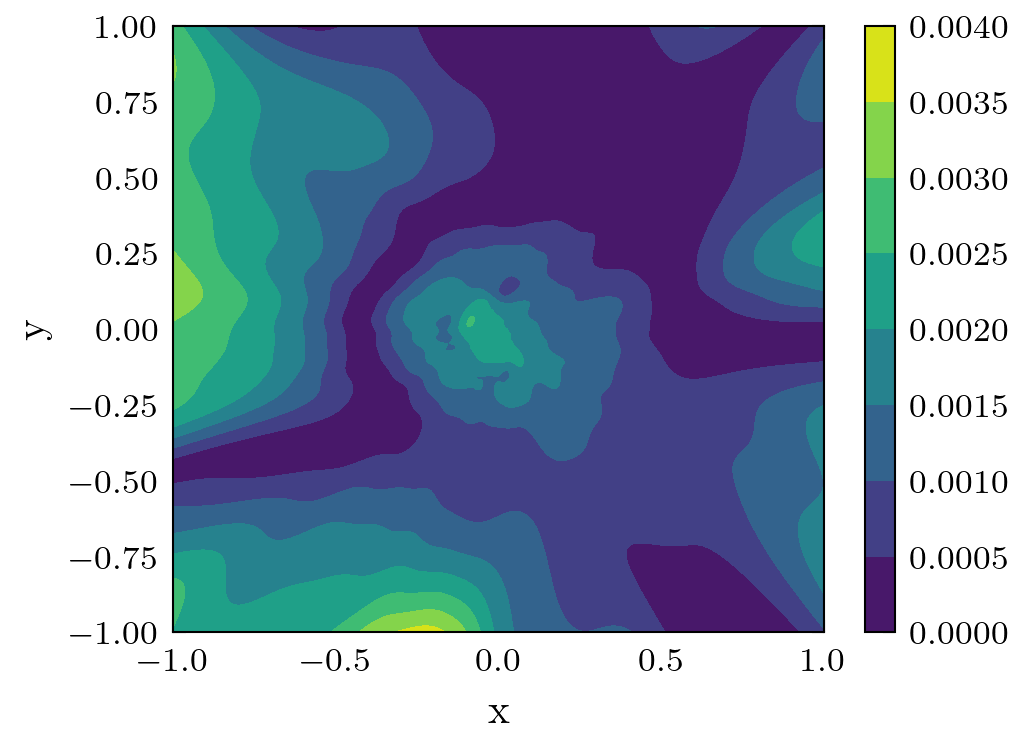}}
\caption{ \textcolor{black}{The result of PINN-ES for the two-dimensional Poisson equation with one peak. (a) training points sampled by evolutionary sampling method (evolved 10 times); (b) the approximate solution $u^{\text{PINN-ES}}$; (c) the absolute error $\vert u^* - u^{\text{PINN-ES}} \vert$.}}
\label{fig:Peak2D_ES}
\end{figure}

\begin{figure}[htbp]
\centering
\subfloat[training points sampled by DAS-G]{\includegraphics[width = 0.26\textwidth]{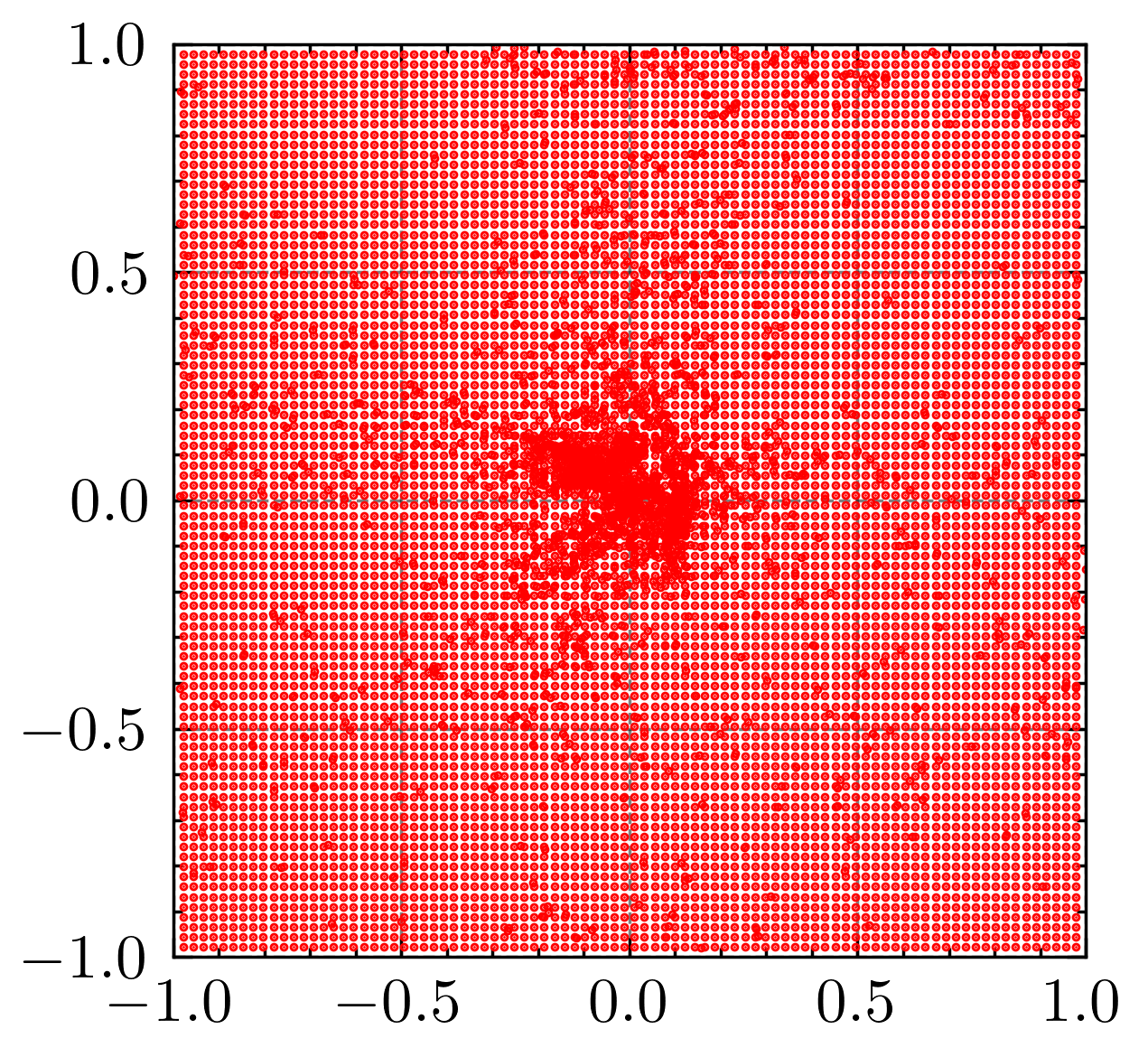}}
\subfloat[approximate solution]{\includegraphics[width = 0.31\textwidth]
{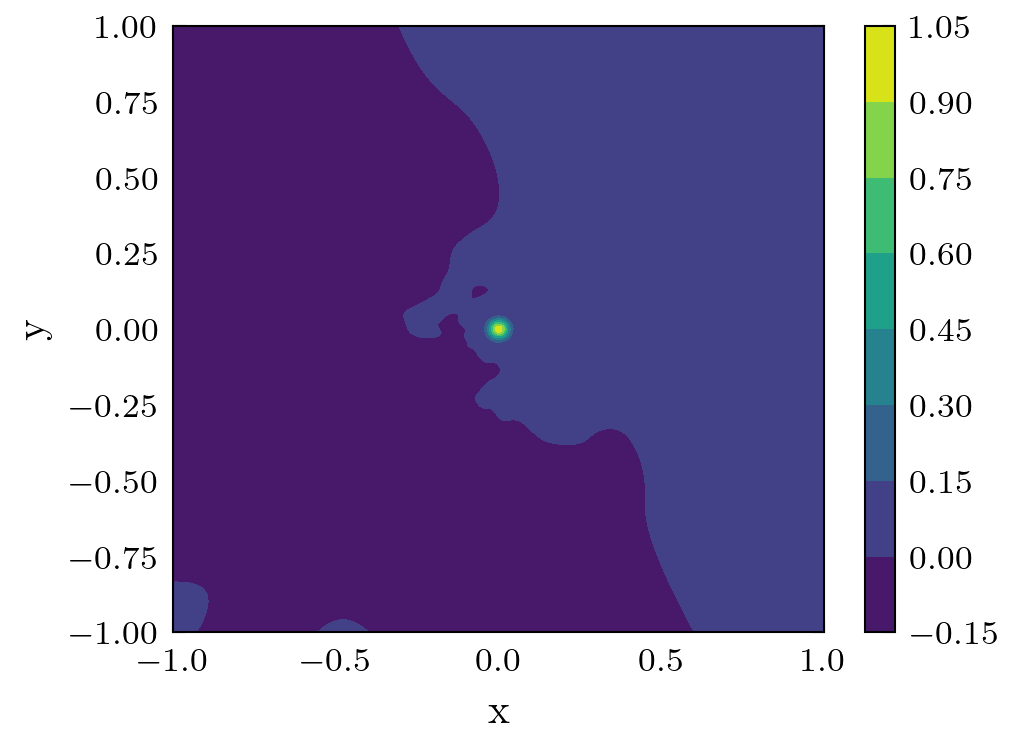}}
\subfloat[absolute error]{\includegraphics[width = 0.31\textwidth]{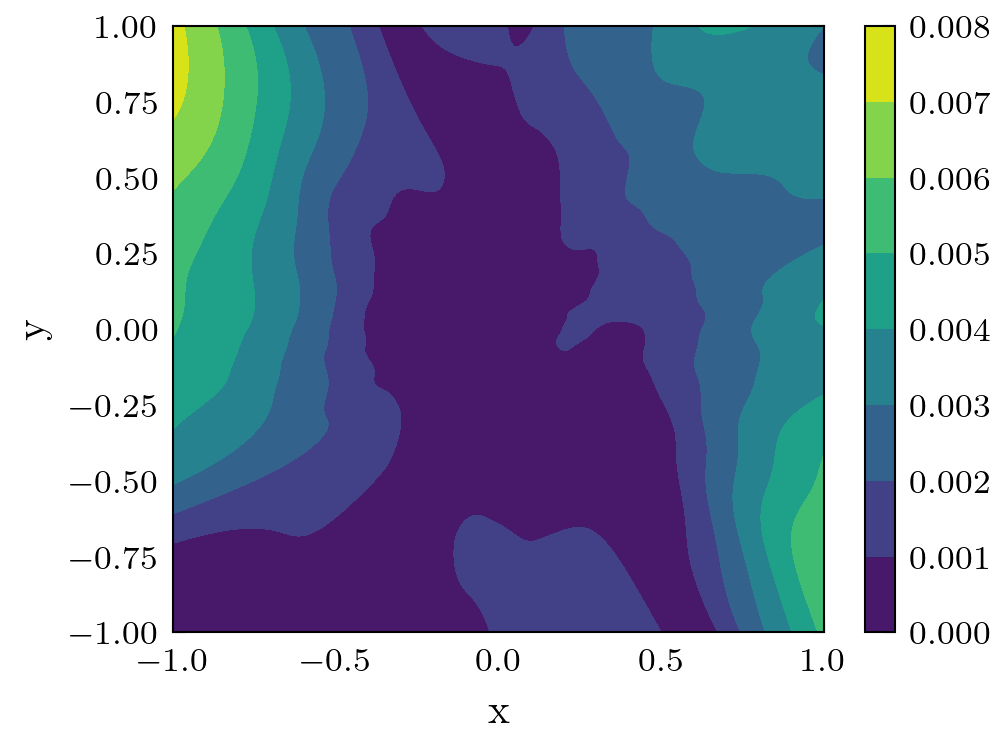}}
\caption{ \textcolor{black}{The result of DAS-PINN for the two-dimensional Poisson equation with one peak. (a) training points sampled by DAS-G method, the hyper parameters are the same as in Section 6.1.1 of \cite{das-pinns}; (b) the approximate solution $u^{\text{DAS-PINN}}$; (c) the absolute error $\vert u^* - u^{\text{DAS-PINN}} \vert$.}}
\label{fig:Peak2D_DAS}
\end{figure}

\begin{figure}[htbp]
\centering
\subfloat[new distribution of sampling points]{\includegraphics[width = 0.27\textwidth]{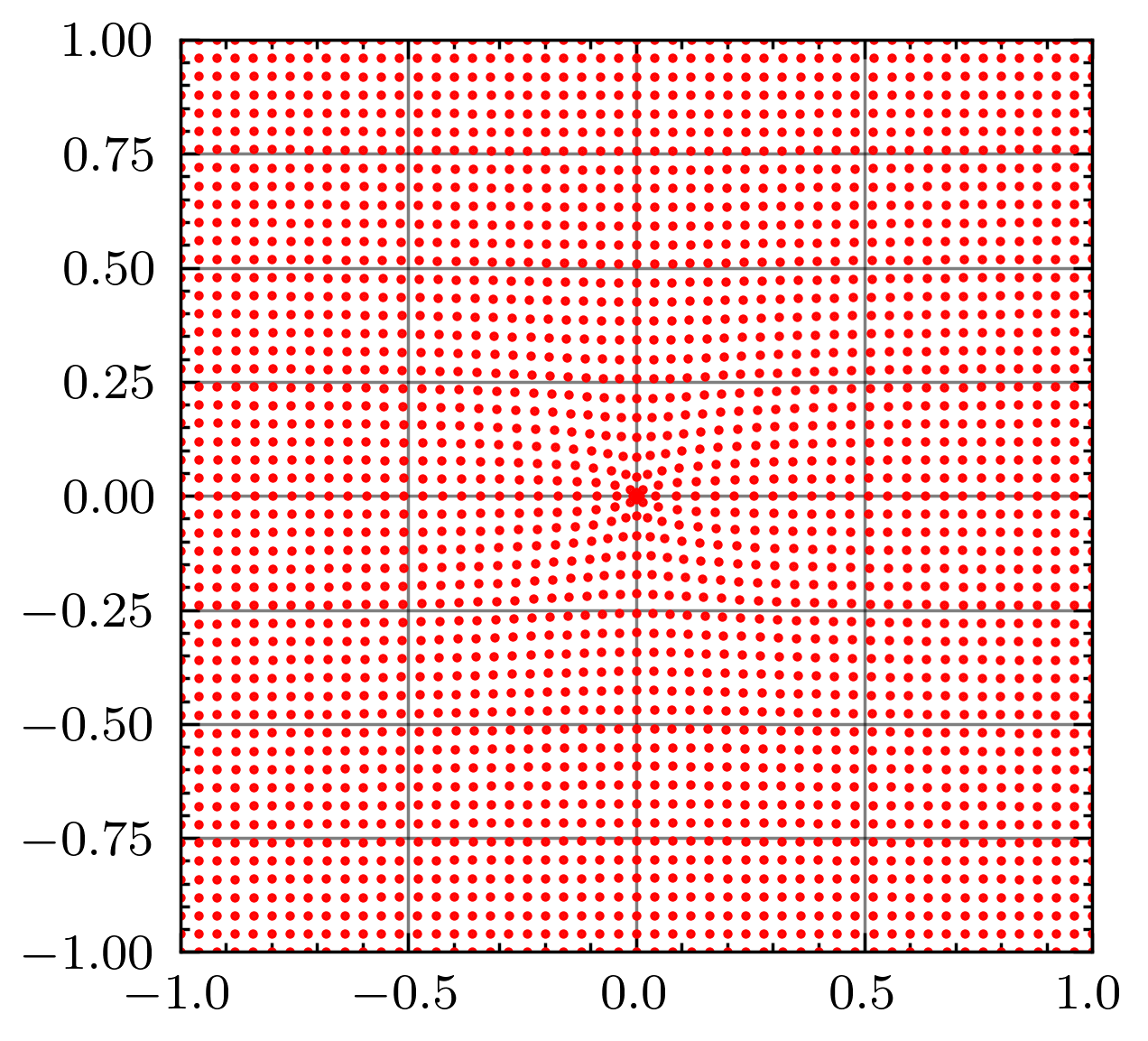}}
\subfloat[approximate solution]{\includegraphics[width = 0.31\textwidth]
{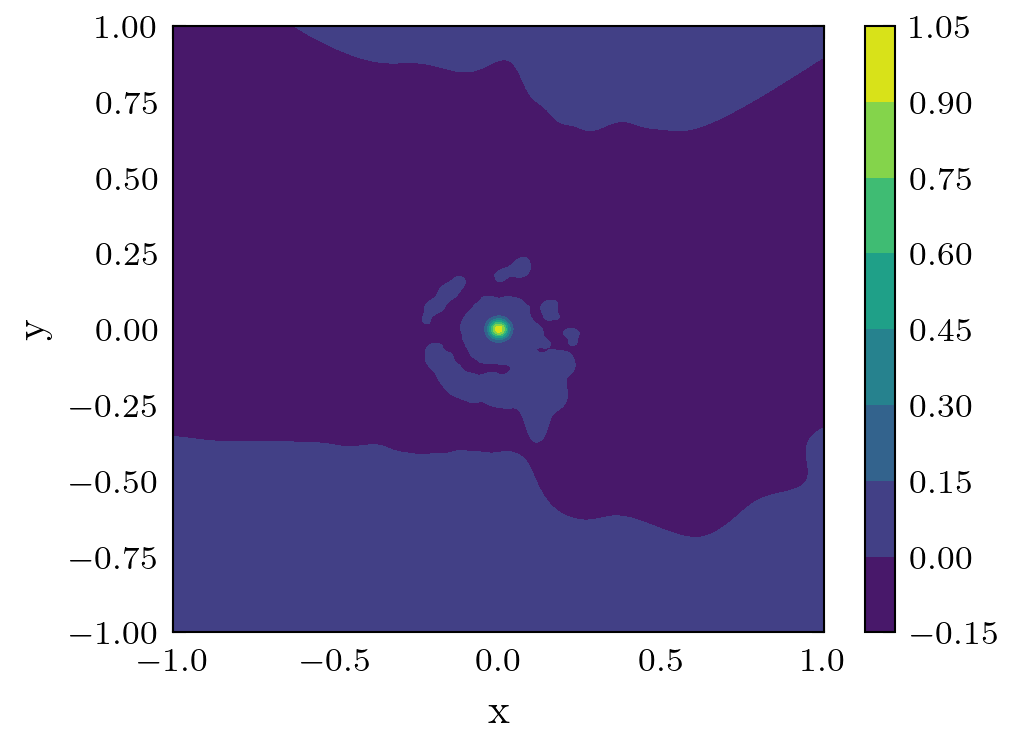}}
\subfloat[absolute error]{\includegraphics[width = 0.31\textwidth]{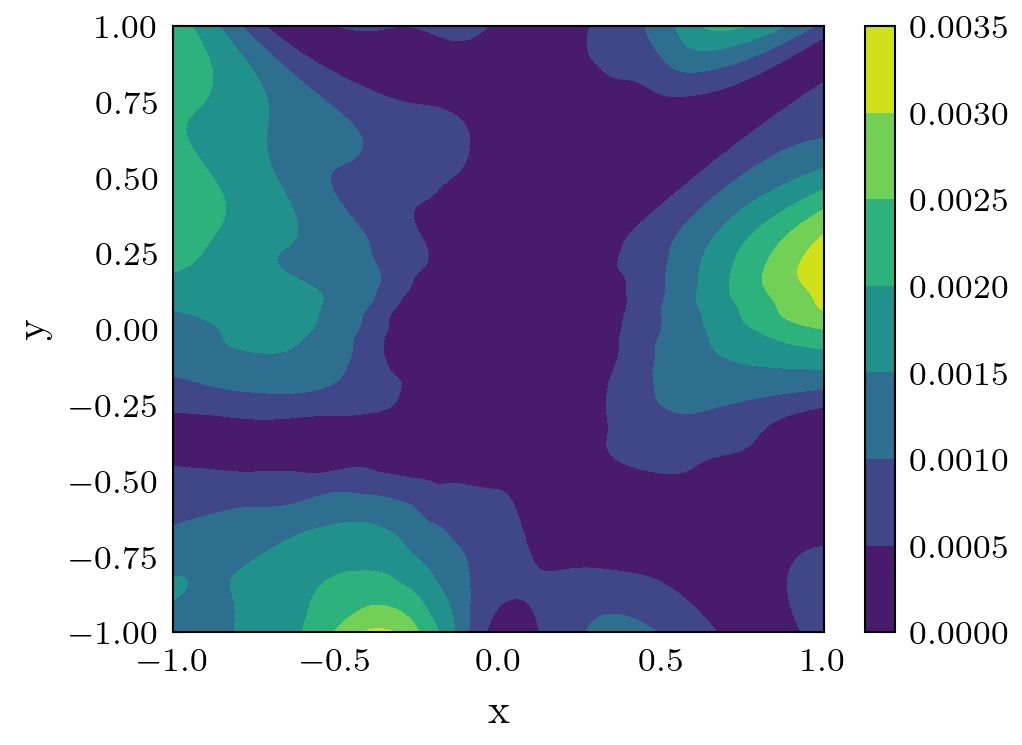}}
\caption{The result of MS-PINN for the two-dimensional Poisson equation with one peak. (a) new training points output by MMPDE-Net, for ease of viewing, we only show $50 \times 50$ points; (b) the approximate solution $u^{\text{MS-PINN}}$ of the MS-PINN; (c) the absolute error $\vert u^* - u^{\text{MS-PINN}} \vert$.}
\label{fig:MMPDE_PINN_Peak2D}
\end{figure}

\begin{figure}[htbp]
\centering
\subfloat[performance of $e_\infty(u)$ during training]{\includegraphics[width = 0.40\textwidth]{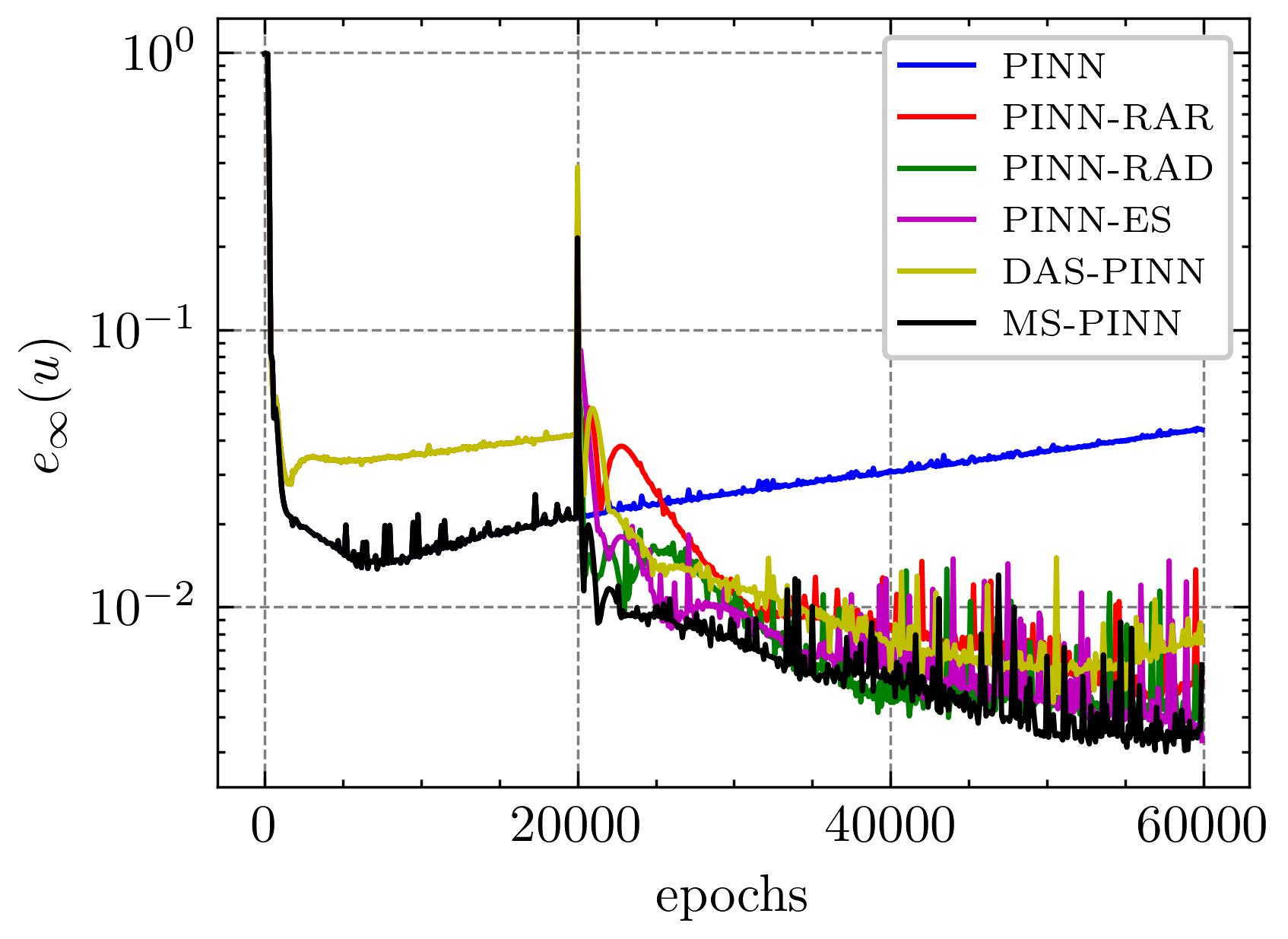}} \quad \quad \quad
\subfloat[performance of $e_2(u)$ during training]{\includegraphics[width = 0.40\textwidth]
{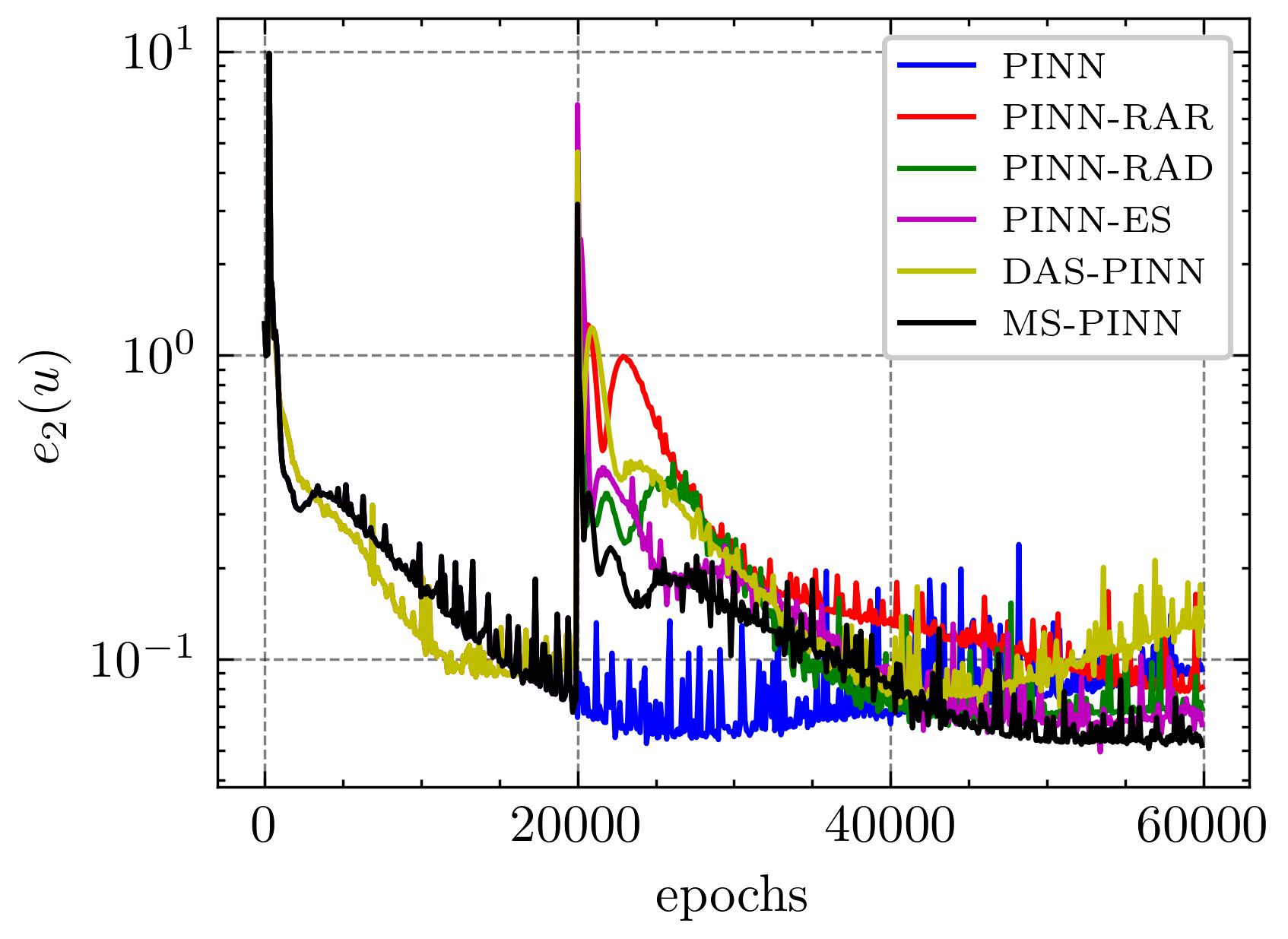}}
\caption{(a) the performance of relative error $e_\infty(u)$ with different training epochs, (b) the performance of relative error $e_2(u)$ with different training epochs.}
\label{fig:Peak2D_OnePeak_ErrorEpochs}
\end{figure}

The numerical results of PINN are given in Fig \ref{fig:PINN_Peak2D}. 
%(a) the analytical solution (Eq.\ref{eq:2d_Peaksolution}), (b) the approximate solution of PINN, and (c) their absolute error.
\textcolor{black}{The distribution of the training points sampled by RAR method and the approximation solution of PINN-RAR are shown in Fig \ref{fig:Peak2D_RAR}.} The distribution of the training points sampled by RAD method and the approximation solution of PINN-RAD are shown in Fig \ref{fig:Peak2D_RAD}. \textcolor{black}{The distribution of the training points sampled by evolutionary sampling method and the approximation solution of PINN-ES are shown in Fig \ref{fig:Peak2D_ES}. The distribution of the training points sampled by DAS method and the approximation solution of DAS-PINN are shown in Fig \ref{fig:Peak2D_DAS}.}
The distribution of the training points generated by MMPDE-Net and the approximation solution of MS-PINN are shown in Fig \ref{fig:MMPDE_PINN_Peak2D}. 
%In Fig \ref{fig:MMPDE_PINN_Peak2D}, we give the numerical computation results of MS-PINN, (a) is the coordinates of the new training points output by MMPDE-Net, and for ease of viewing, we only show 50 × 50 points, (b) is the approximation solution of MS-PINN, and (c) is the absolute error between the analytical solution and the approximation solution of MS-PINN.
The performance of the six different methods in terms of relative errors on the test set during the training process is given by Fig \ref{fig:Peak2D_OnePeak_ErrorEpochs}. It is observed that there is a big jump of errors at epoch $= 20000$ in Fig \ref{fig:Peak2D_OnePeak_ErrorEpochs}, which is due to that pre-training with PINN works in the first stage, MMPDE-Net takes effect after epoch $= 20000$ and points are redistributed meanwhile. Similar phenomena arise in other cases. Moreover, a comparison of errors using different methods is given in Table \ref{tab:Peak_2D_Results}. 
From these results, it is observed that MS-PINN can obtain better results than the other methods.
%We implement PINN with $100\times100$ sampling points, PINN-RAR with $100\times100$ sampling  points, PINN-RAD with $100\times100$ sampling  points and MS-PINN with $100\times100$ sampling  points. It can be seen that more sampling points can reduce the error of the computation and numerical results of MS-PINN is more accurate than that of the others.
% PINN with 200×200 sampling  points 
%multiplying the number of points sampled from the equispaced uniform grid, RAR method and our method all reduce the error by an order of magnitude, MS-PINN works better.

\begin{table}[h]
\scriptsize
\centering
\caption{
Comparison of errors using different methods.
}
\setlength{\tabcolsep}{3.mm}{
\begin{tabular}{|c|c|c|c|c|c|c|}%{p{1cm}p{2.2cm}p{1.2cm}p{1.2cm}p{1.2cm}p{1.2cm}p{1.2cm}p{1.6cm}}
\hline\noalign{\smallskip}
Relative      &  $100 \times 100$ & $100 \times 100$ & $100 \times 100$ & $100 \times 100$& $100 \times 100$& $100 \times 100$ \\
error        &  PINN &PINN-RAR& PINN-RAD& PINN-ES& DAS-PINN & MS-PINN\\
\hline
$e_\infty(u)$  & $4.374 \times 10^{-2}$  & $5.302 \times 10^{-3}$  &  $3.930 \times 10^{-3}$ & $3.753 \times 10^{-3}$&$7.466 \times 10^{-3}$&$3.389 \times 10^{-3}$\\
% \hline
% $e_\infty^\rho(u)$  & $1.661 \times 10^{-2}$  & $4.855 \times 10^{-3}$  & $2.416 \times 10^{-3}$ & $2.013 \times 10^{-3}$\\
\hline
$e_2(u)$  & $9.213 \times 10^{-2}$  & $7.908 \times 10^{-2}$&
$6.863 \times 10^{-2}$  & $6.593 \times 10^{-2}$&$1.279 \times 10^{-1}$&$5.425 \times 10^{-2}$\\
% \hline
% $e_2^\rho(u)$  & $3.814 \times 10^{-2}$  & $3.530 \times 10^{-2}$  & $1.779 \times 10^{-2}$ & $1.060 \times 10^{-2}$\\
\hline
\end{tabular}
}
\label{tab:Peak_2D_Results} 
\end{table}

% INFO:root:FINAL_RE_Linfinity: 4.373526e-02
% INFO:root:FINAL_RE_L2: 9.213140e-02

% INFO:root:FINAL_RE_Linfinity: 5.302363e-03
% INFO:root:FINAL_RE_L2: 7.907644e-02

% INFO:root:FINAL_RE_Linfinity: 3.930429e-03
% INFO:root:FINAL_RE_L2: 6.863460e-02

% INFO:root:FINAL_RE_Linfinity: 3.752895e-03
% INFO:root:FINAL_RE_L2: 6.592952e-02

% INFO:root:FINAL_RE_Linfinity: 7.465676e-03
% INFO:root:FINAL_RE_L2: 1.278690e-01

% INFO:root:RE_Linfinity: 3.389381e-03
% INFO:root:RE_L2: 5.425328e-02
% \begin{table}[h]
% \scriptsize
% \centering
% \caption{
% Comparison of errors using different methods.
% }
% \setlength{\tabcolsep}{7.mm}{
% \begin{tabular}{|c|c|c|c|c|c|}%{p{1cm}p{2.2cm}p{1.2cm}p{1.2cm}p{1.2cm}p{1.2cm}p{1.2cm}p{1.6cm}}
% \hline\noalign{\smallskip}
% Relative      &  $100 \times 100$ & $100 \times 100$ & $100 \times 100$ &$200 \times 200$ & $100 \times 100$ \\
% error        &  PINN &PINN-RAR& PINN-RAD&PINN  & MS-PINN\\
% \hline
% $e_\infty(u)$  & $1.593 \times 10^{-2}$  & $4.916 \times 10^{-3}$  &  $6.062 \times 10^{-3}$ & $2.447 \times 10^{-3}$& $2.038 \times 10^{-3}$\\
% % \hline
% % $e_\infty^\rho(u)$  & $1.661 \times 10^{-2}$  & $4.855 \times 10^{-3}$  & $2.416 \times 10^{-3}$ & $2.013 \times 10^{-3}$\\
% \hline
% $e_2(u)$  & $8.844 \times 10^{-2}$  & $8.879 \times 10^{-2}$&
% $8.927 \times 10^{-2}$ & $4.486 \times 10^{-2}$ & $2.686 \times 10^{-2}$\\
% % \hline
% % $e_2^\rho(u)$  & $3.814 \times 10^{-2}$  & $3.530 \times 10^{-2}$  & $1.779 \times 10^{-2}$ & $1.060 \times 10^{-2}$\\
% \hline
% \end{tabular}
% }
% \label{tab:Peak_2D_Results} 
% \end{table}
%\hline
%time(s)  & 1418  & 1450 & 2529 & 2999 \\ 

\subsubsection{MS-PINN with the iterative MMPDE-Net}
\label{sec:MS-PINN with the iterative MMPDE-Net algorithm}
For Eq \eqref{eq:2d_Poisson}, we implement MS-PINN with iterative MMPDE-Net (Algorithm \ref{alg:MMPDE-Net-iterations}).
The monitor function that we use is 
\begin{equation}
    \label{eq:2d_monitorfunction_iterations}
    \hspace{-0.3cm}
    \begin{array}{r@{}l}
        \begin{aligned}
            w = 1+ 2u.
        \end{aligned}
    \end{array}
\end{equation}

We sample $80\times80$ points in $\Omega$ and 324 points on $\partial \Omega$ as the training set and $400 \times 400$  points as the test set. Four experiments with different training strategies are used to validate the effectiveness of the iterative MMPDE-Net method: PINN, MS-PINN with one MMPDE-Net iteration, MS-PINN with three MMPDE-Net iterations and MS-PINN with five MMPDE-Net iterations (See Table \ref{tab:Peak_2D_iterations_comparison}).

\begin{table}[h]
\scriptsize
\centering
\caption{
Settings of four training strategies.
}
\setlength{\tabcolsep}{3.mm}{
\begin{tabular}{|c|c|c|c|c|c|c|}%{p{1cm}p{2.2cm}p{1.2cm}p{1.2cm}p{1.2cm}p{1.2cm}p{1.2cm}p{1.6cm}}
\hline\noalign{\smallskip}
Training &  Sampling & USE  & MMPDE-Net & MMPDE-Net& Pre-training &Formal training\\
strategy& points & MMPDE-Net & iterations& epochs& epochs & epochs\\
\hline
Strategy 1  & 6400+324  &  No & - & - &  -  &70000 \\
\hline
Strategy 2  & 6400+324  &  Yes & 1 & 20000&20000&50000\\
\hline
Strategy 3  & 6400+324  &  Yes &3&60000&20000&50000\\
\hline
Strategy 4  &6400+324  &  Yes&5&100000&20000&50000 \\
\hline
\end{tabular}
}
\label{tab:Peak_2D_iterations_comparison} 
\end{table}

\begin{figure}[htbp]
\centering
\subfloat[initial training points]{\includegraphics[width = 0.25\textwidth]{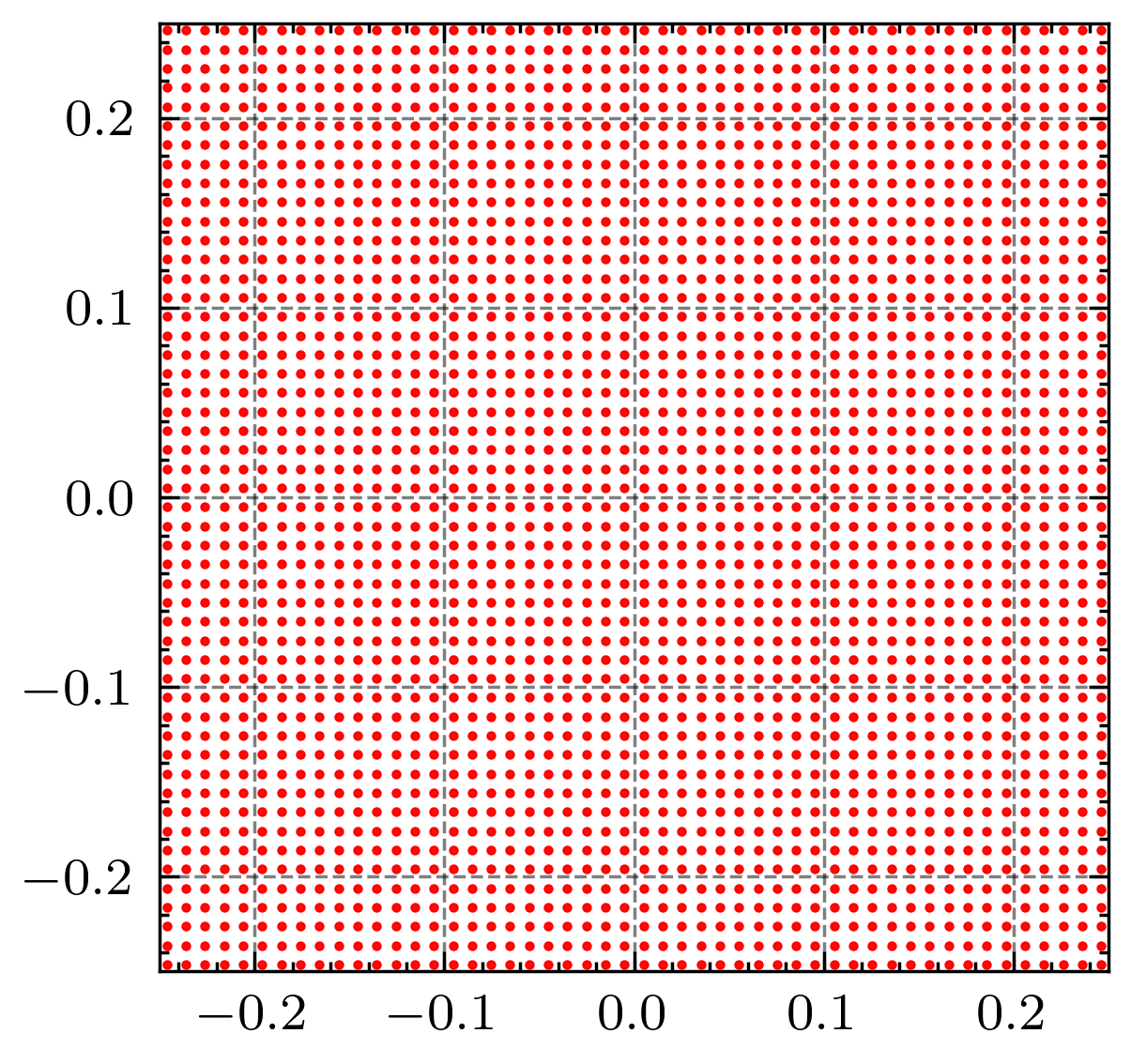}}
\subfloat[after 1 iteration]{\includegraphics[width = 0.25\textwidth]
{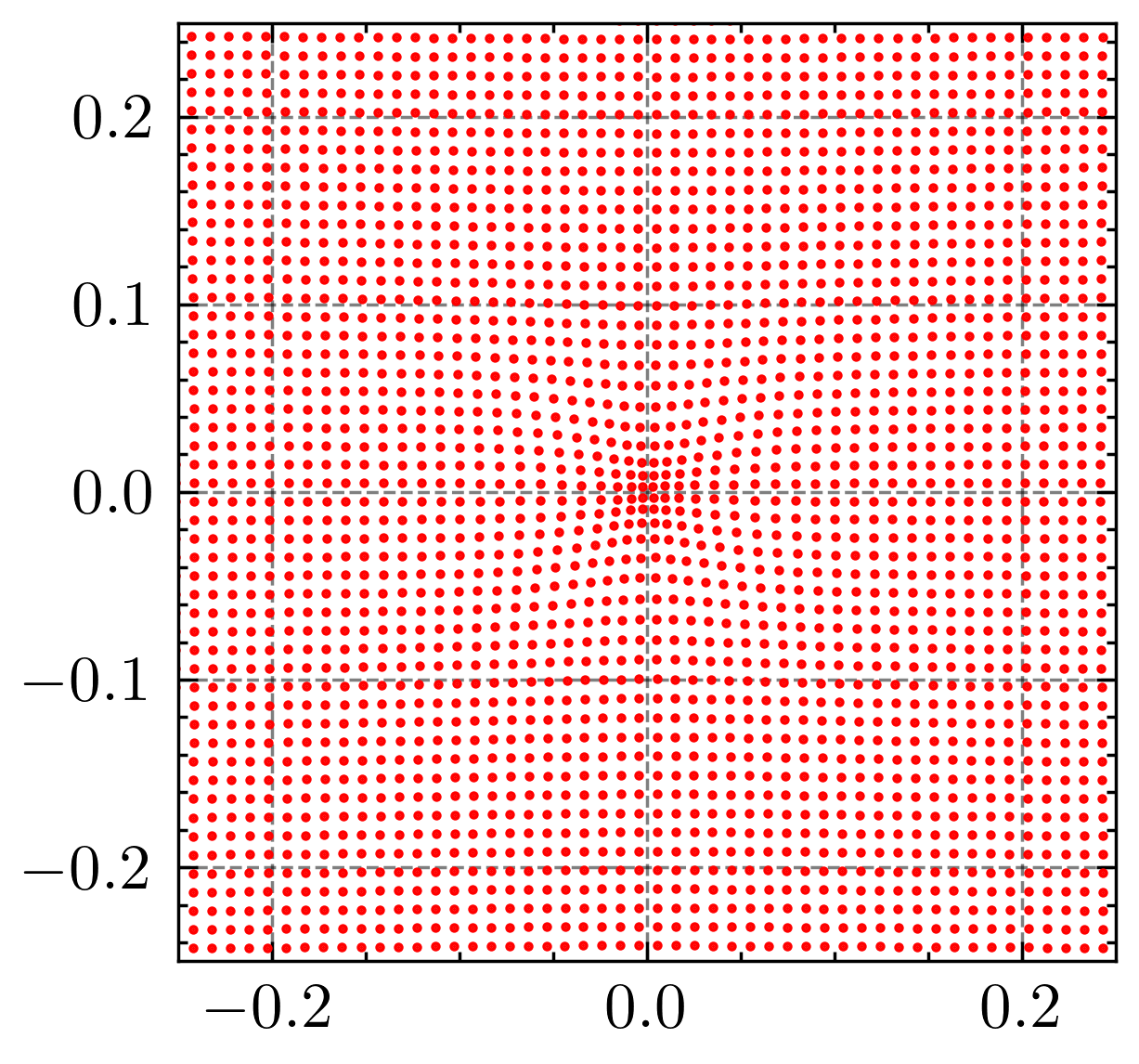}}
\subfloat[after 3 iterations]{\includegraphics[width = 0.25\textwidth]{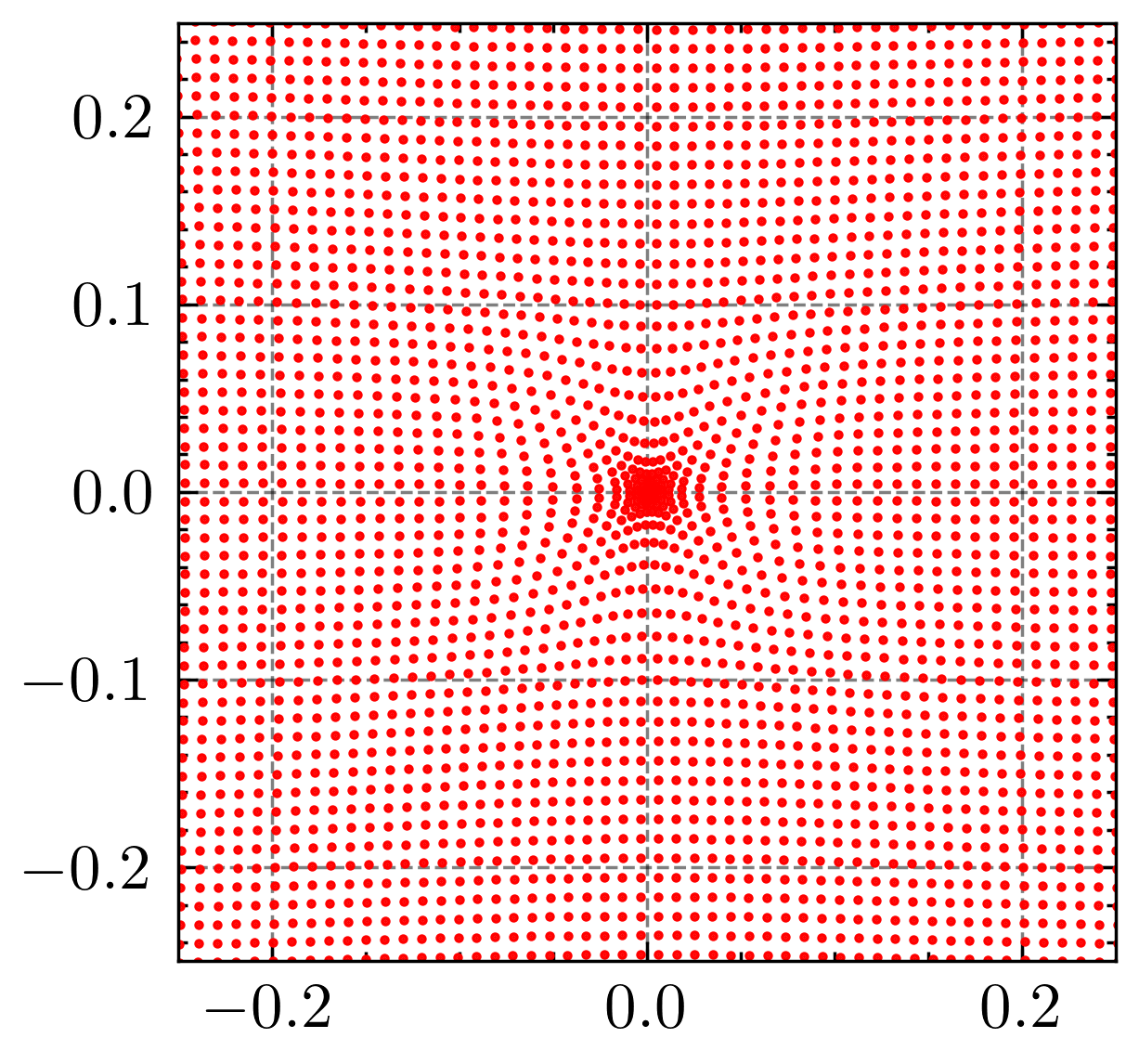}}
\subfloat[after 5 iterations]{\includegraphics[width = 0.25\textwidth]{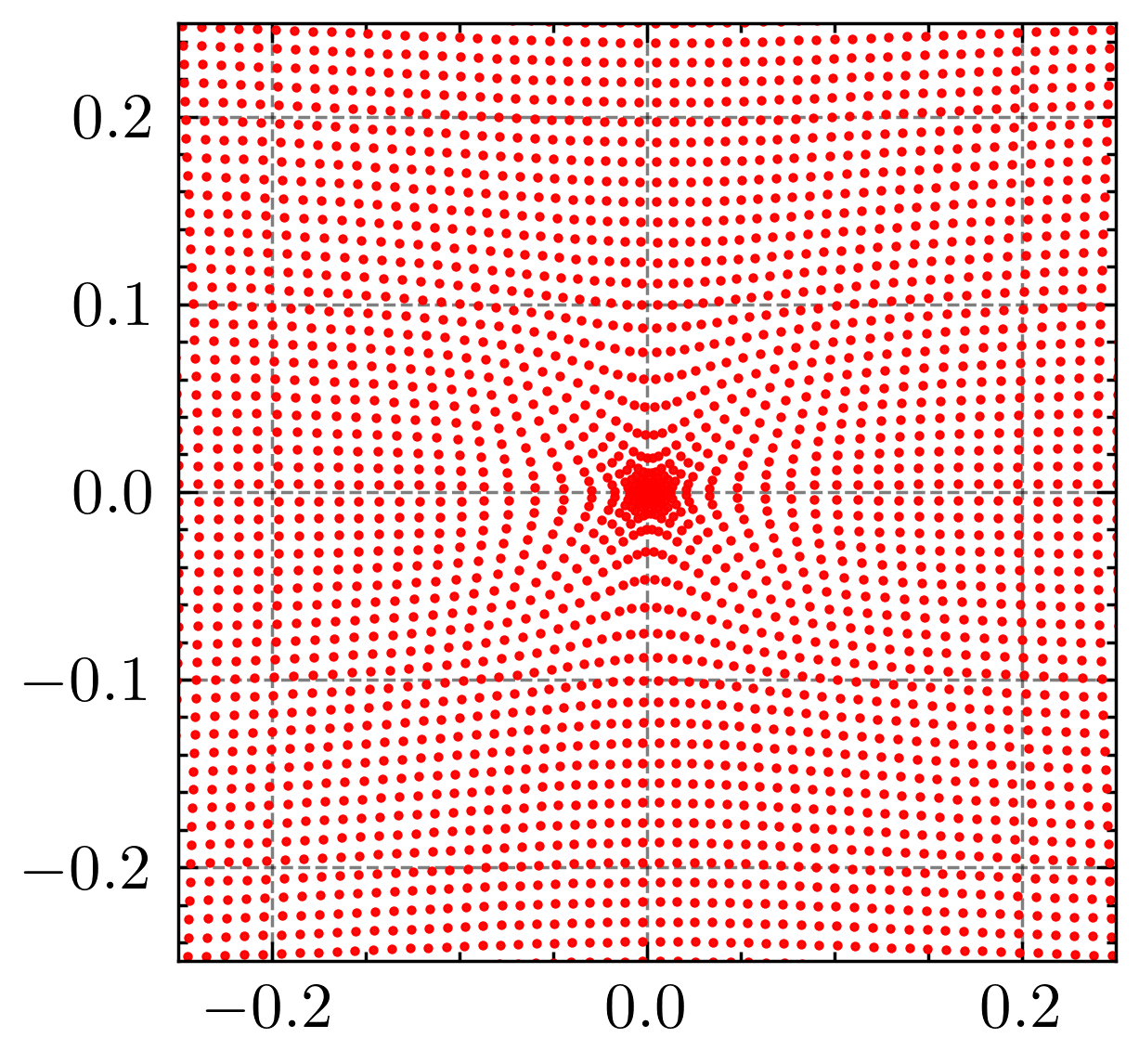}}
\caption{The results of iterative MMPDE-Net for the two-dimensional Poisson equation with one peak. To see the results clearly, we show %here only the coordinates of the point on 
the region $[-0.25,0.25]^2$ particularly.}
\label{fig:Peak2D-Iterations_mesh_part2}
\end{figure}

\begin{figure}[htbp]
\centering
\subfloat[performance of $e_\infty(u)$ during training]{\includegraphics[width = 0.40\textwidth]{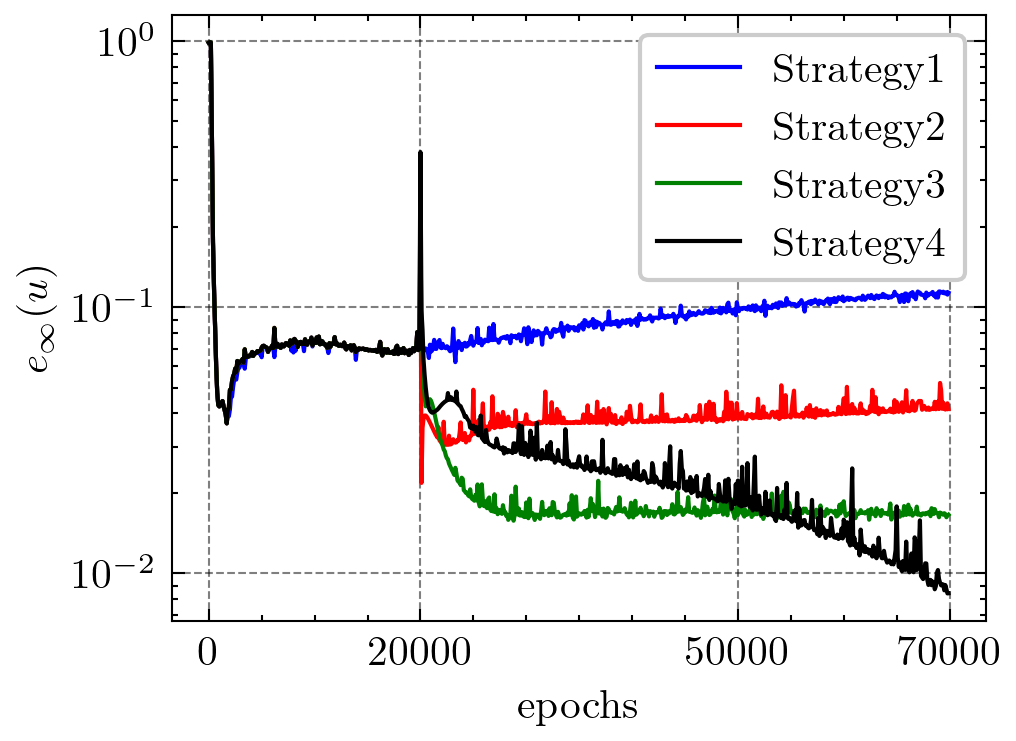}} \quad \quad \quad
\subfloat[performance of $e_2(u)$ during training]{\includegraphics[width = 0.40\textwidth]
{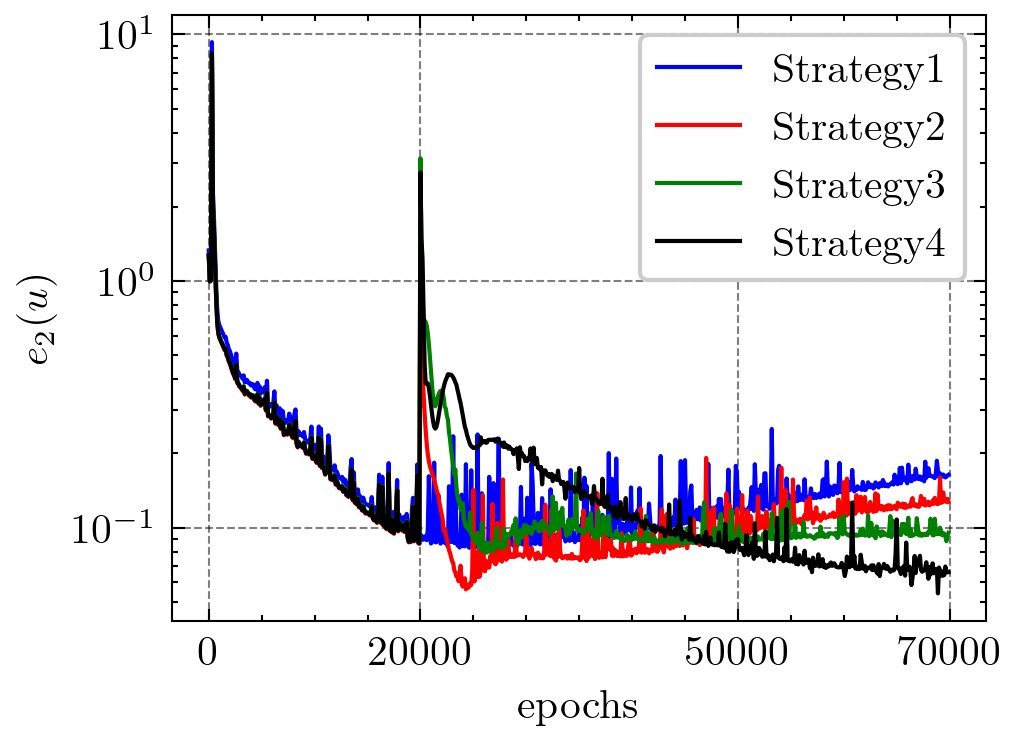}}
\caption{(a) the performance of relative error $e_\infty(u)$ with different training epochs; (b) the performance of relative error $e_2(u)$ with different training epochs.}
\label{fig:Peak2D-Iterations_ErrorEpochs}
\end{figure}

\begin{figure}[htbp]
\centering
\subfloat[PINN]{\includegraphics[width = 0.25\textwidth]{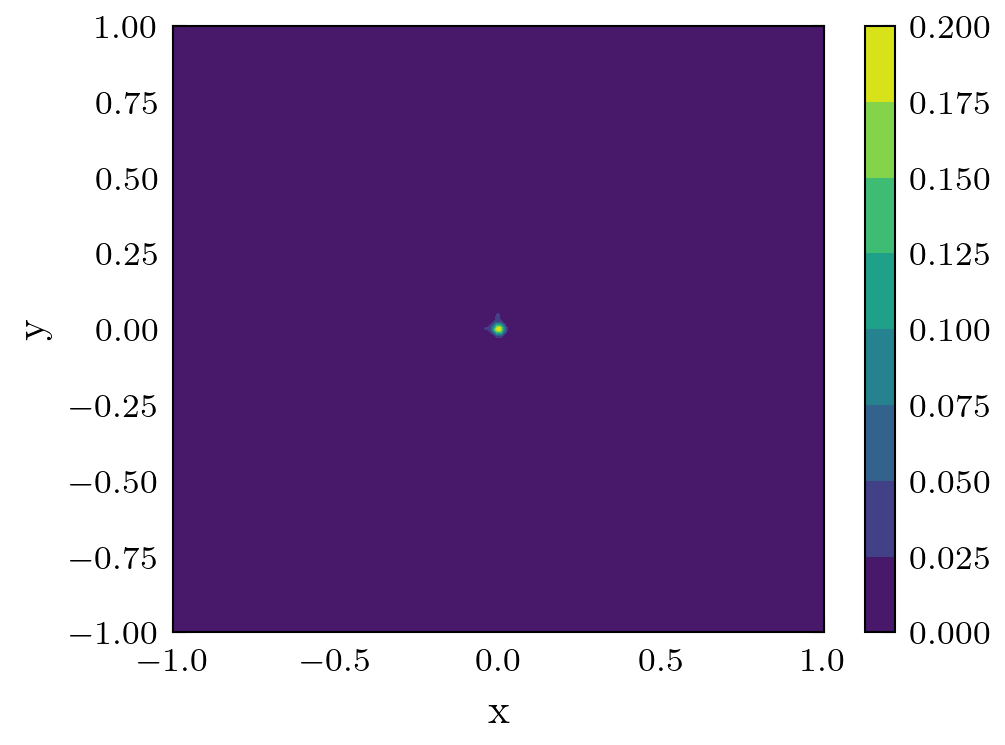}}
\subfloat[after 1 iteration]{\includegraphics[width = 0.25\textwidth]
{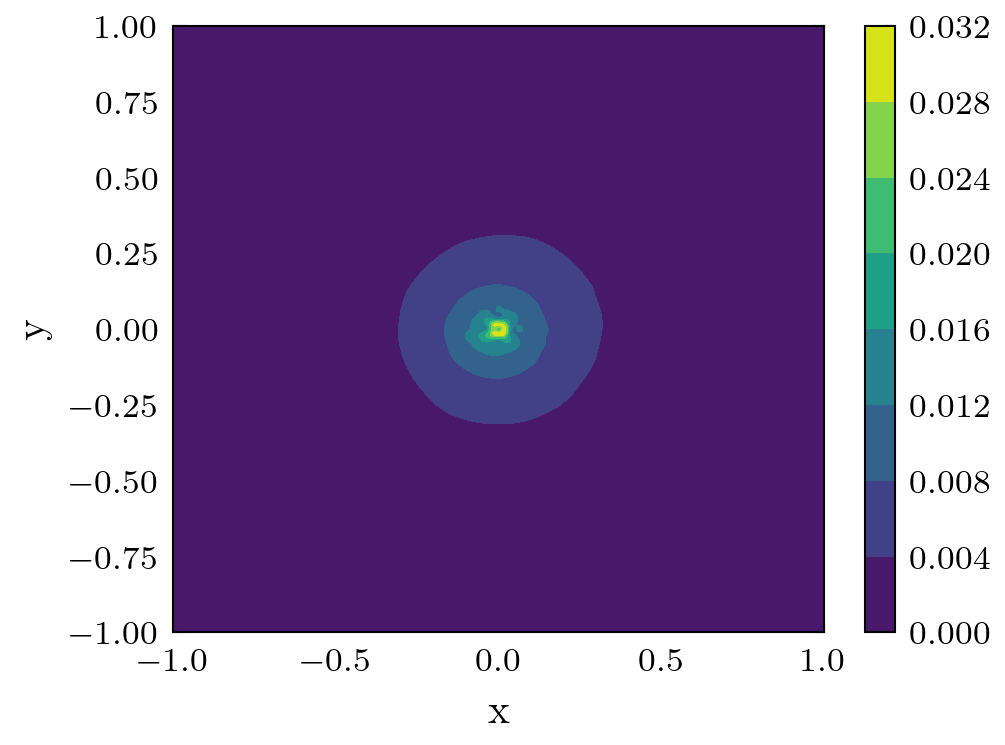}}
\subfloat[after 3 iterations]{\includegraphics[width = 0.25\textwidth]{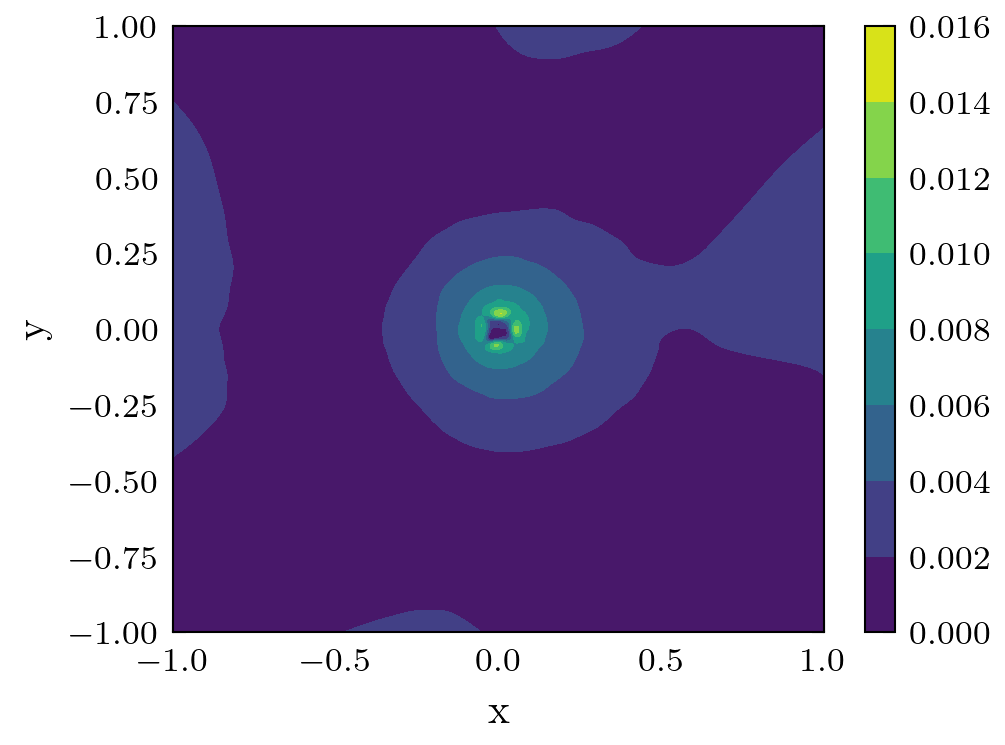}}
\subfloat[after 5 iterations]{\includegraphics[width = 0.25\textwidth]{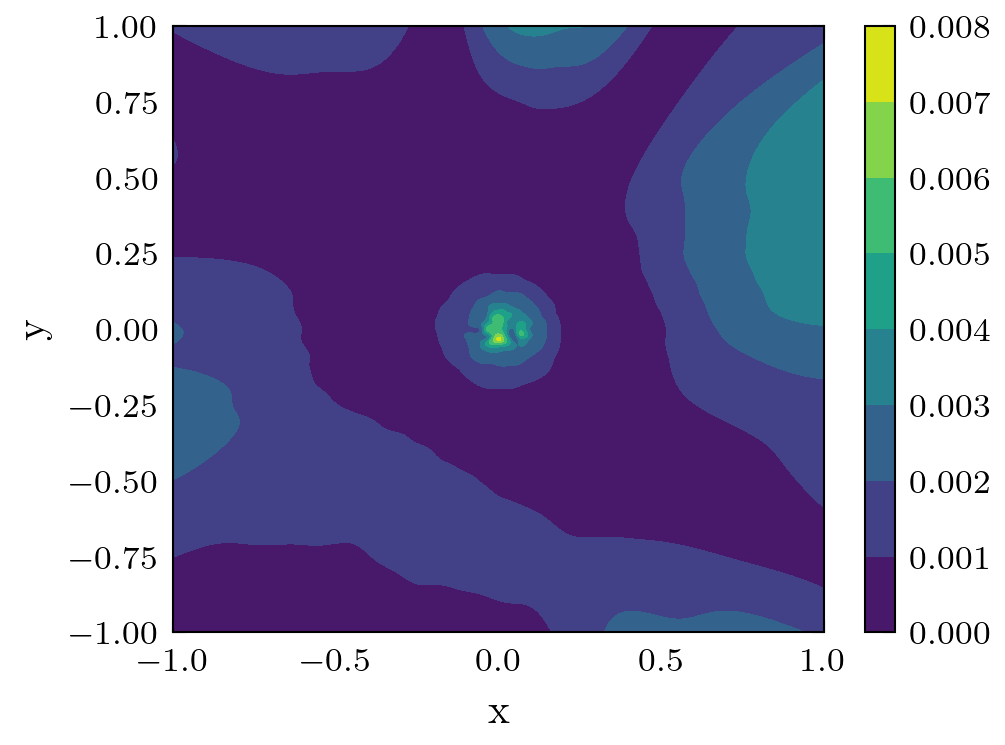}}
\caption{The absolute errors of MS-PINN using iterative MMPDE-Net for the two-dimensional Poisson equation with one peak. (a) the absolute error of PINN; (b) the absolute error of MS-PINN using MMPDE-Net with one iteration; (c) the absolute error of MS-PINN using MMPDE-Net with three iterations; (d) the absolute error of MS-PINN using MMPDE-Net with five iterations.}
\label{fig:Peak2D-Iterations_heatmap}
\end{figure}

\begin{figure}
\centering
\includegraphics[width = 0.45\textwidth]{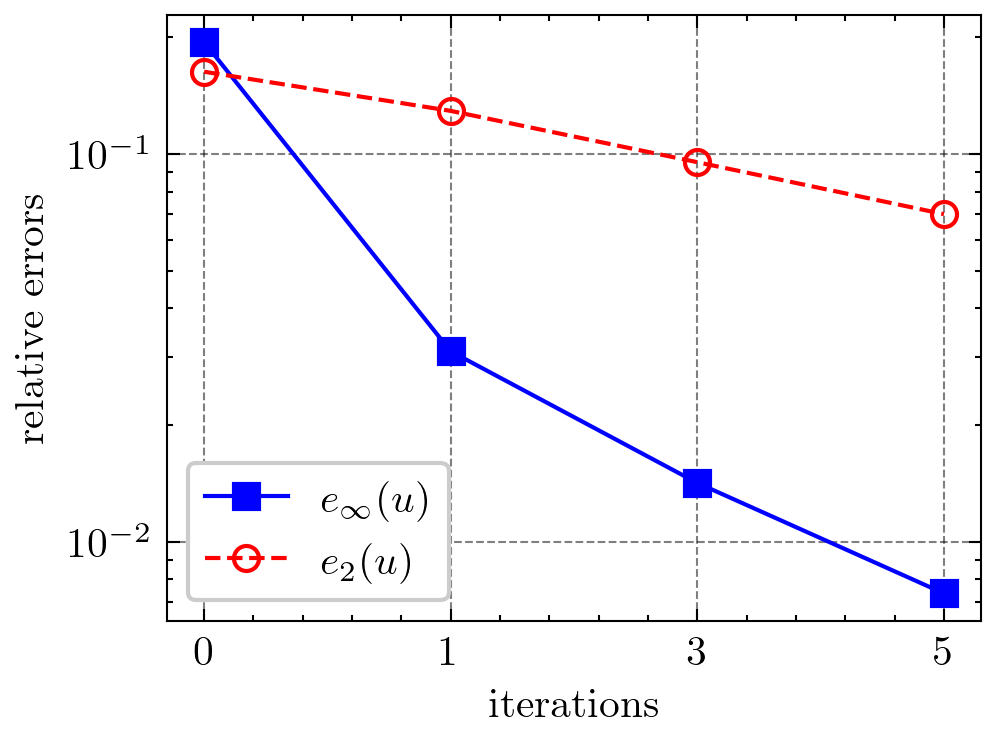}
\caption{The relative errors of MS-PINN with different  MMPDE-Net iterations.}
\label{fig:Peak2D-Iterations_Errorline}
\end{figure}

From Fig \ref{fig:Peak2D-Iterations_mesh_part2}, it can be seen that the sampling points obtained from MMPDE-Net are  concentrated much more on the place where the solution has large variations as the number of iterations increases.
%and for ease of viewing, we only show 50 × 50 points.
The performance of the four different strategies in terms of relative errors during the training process is given in Fig \ref{fig:Peak2D-Iterations_ErrorEpochs}. MMPDE-Net iteration can reduce the errors effectively, since it make the points move to the places where need more samplings.
From Fig \ref{fig:Peak2D-Iterations_heatmap} and Fig \ref{fig:Peak2D-Iterations_Errorline}, it can be seen that MMPDE-Net with iterations in MS-PINN can improve the numerical results, the relative errors decrease as the number of iterations increases.

% @@@@@@@@@@@@@@@@@@@@@@@@@@
\subsection{Two-dimensional Poisson equation with two peaks}
\label{sec:DoublePeaks_2D}

Consider the two-dimensional Poisson equation Eq \eqref{eq:2d_Poisson} in $\Omega = (-1,1)^2$,
the analytical solution with two peaks is given by

\begin{equation}
    \label{eq:DoublePeaks2d_solution}
    \hspace{-0.3cm}
    \begin{array}{r@{}l}
        \begin{aligned}
            u = e^{-1000 \left( x^2+(y-0.5)^2\right)}+ e^{-1000\left(x^2+(y+0.5)^2\right)}.
        \end{aligned}
    \end{array}
\end{equation}
The monitor function that we use in MMPDE-Net is 

\begin{equation}
    \label{eq:DoublePeaks2d_monitorfunction}
    \hspace{-0.3cm}
    \begin{array}{r@{}l}
        \begin{aligned}
            w = \sqrt{1+ 1000u^2 + \lvert \nabla u \rvert^2}.
        \end{aligned}
    \end{array}
\end{equation}

We sample $100 \times 100$  points in $\Omega$  and 400 points on $\partial \Omega$ as the training set and $400 \times 400$  points as the test set.  
%In $\Omega$, we sample 100×100 uniformly distributed points as the residual training set and 400×400 uniformly distributed points as the test set. On $\partial \Omega$, we also sample 400 uniformly distributed boundary training points.
The setting of main parameters in MS-PINN is given in Table \ref{tab:parameter-MMPDENet}. 
%Refer to Table \ref{tab:parameter-MMPDENet} for the setting of key parameters of MSPINN.
We compare the numerical results of MS-PINN  with the results obtained by PINN trained 60000 epochs to show the effectiveness of our method.

\begin{figure}[htbp]
\centering
\subfloat[analytical solution]{\includegraphics[width = 0.30\textwidth]{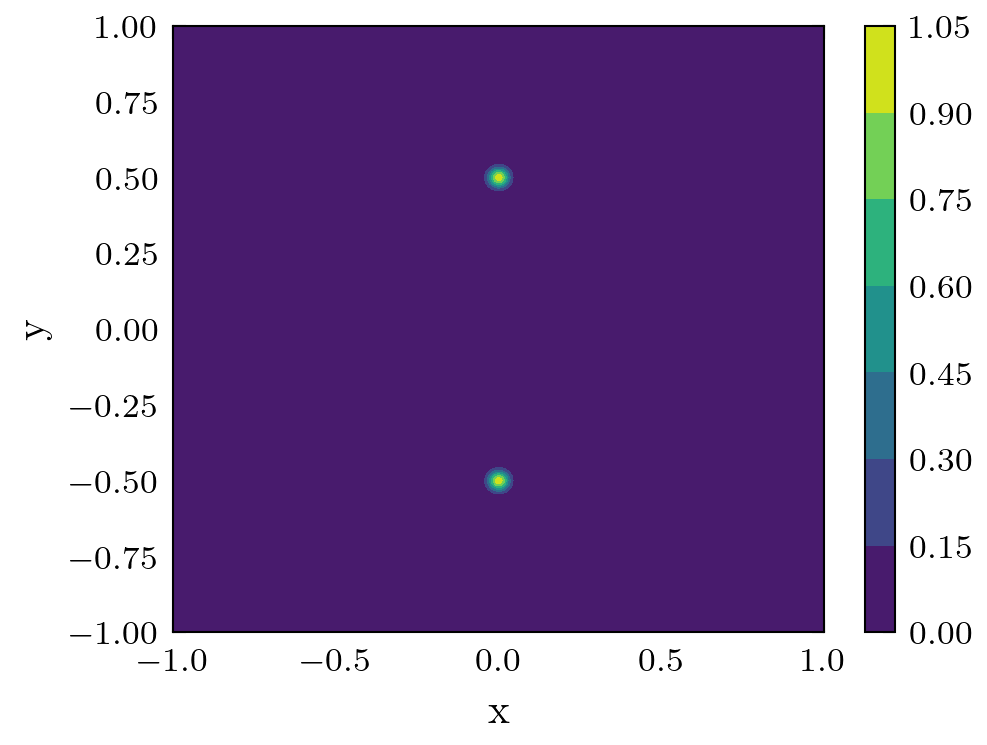}}
\subfloat[approximate solution]{\includegraphics[width = 0.30\textwidth]
{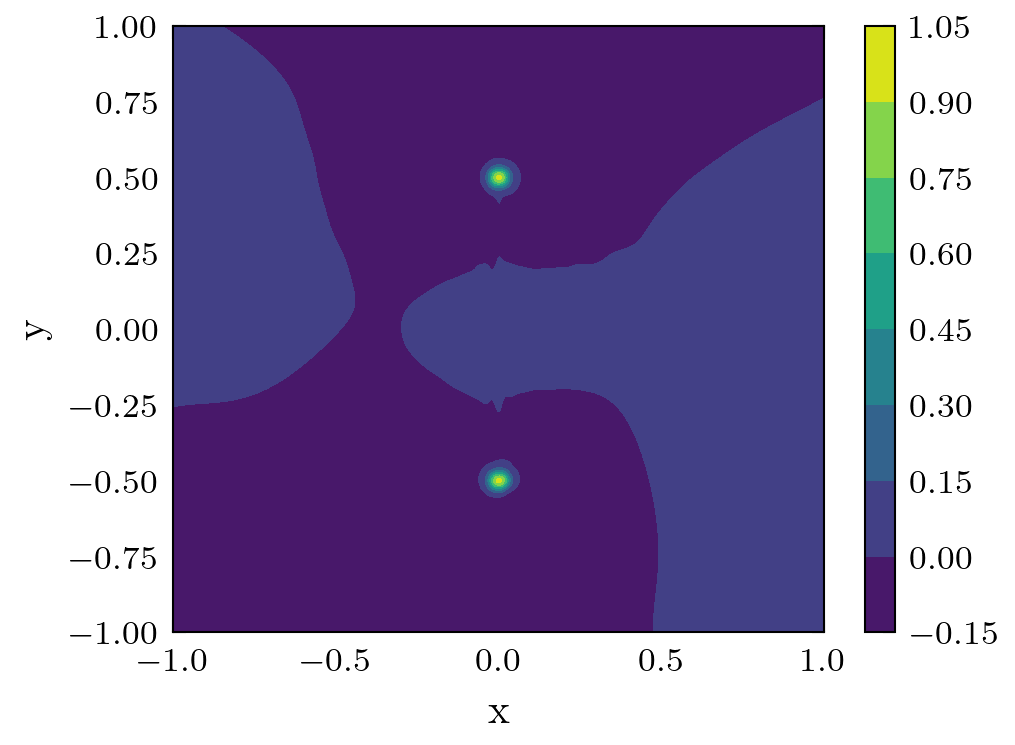}}
\subfloat[absolute error]{\includegraphics[width = 0.30\textwidth]{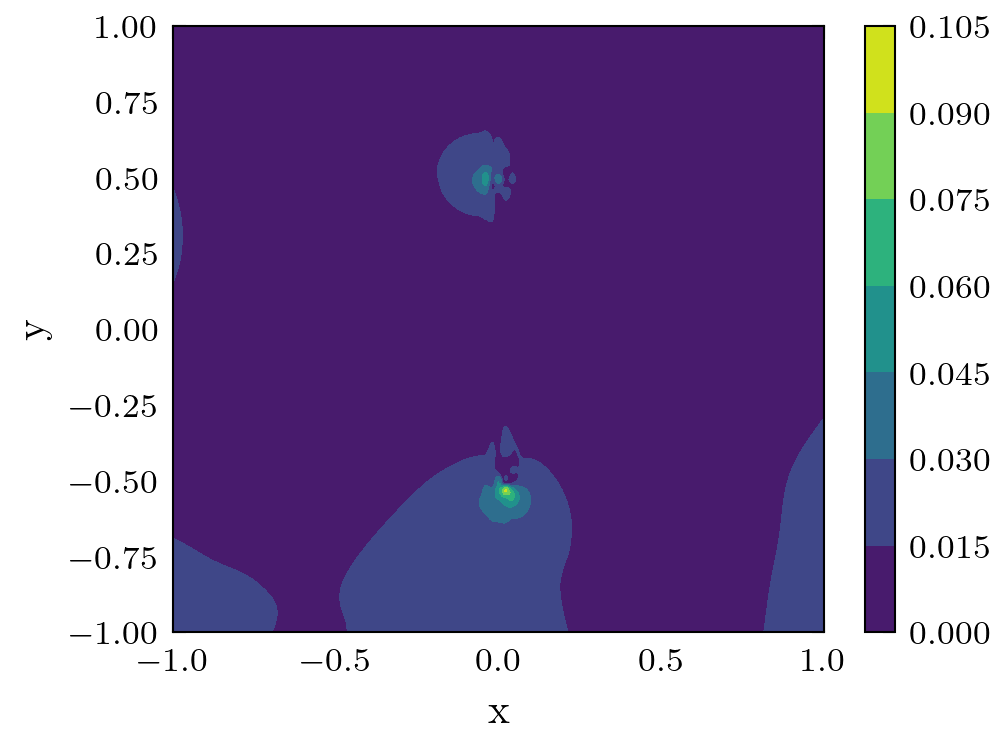}}
\caption{The result of PINN for the two-dimensional Poisson equation with two peaks. (a) the analytical solution $u^*$ of the equation; %\ref{eq:DoublePeaks2d_Poisson}, 
(b) the approximate solution $u^{\text{PINN}}$ of PINN; (c) the absolute error $\vert u^* - u^{\text{PINN}} \vert$.}
\label{fig:PINN_DoublePeaks2D}
\end{figure}

\begin{figure}[htbp]
\centering
\subfloat[new distribution of sampling points]{\includegraphics[width = 0.27\textwidth]{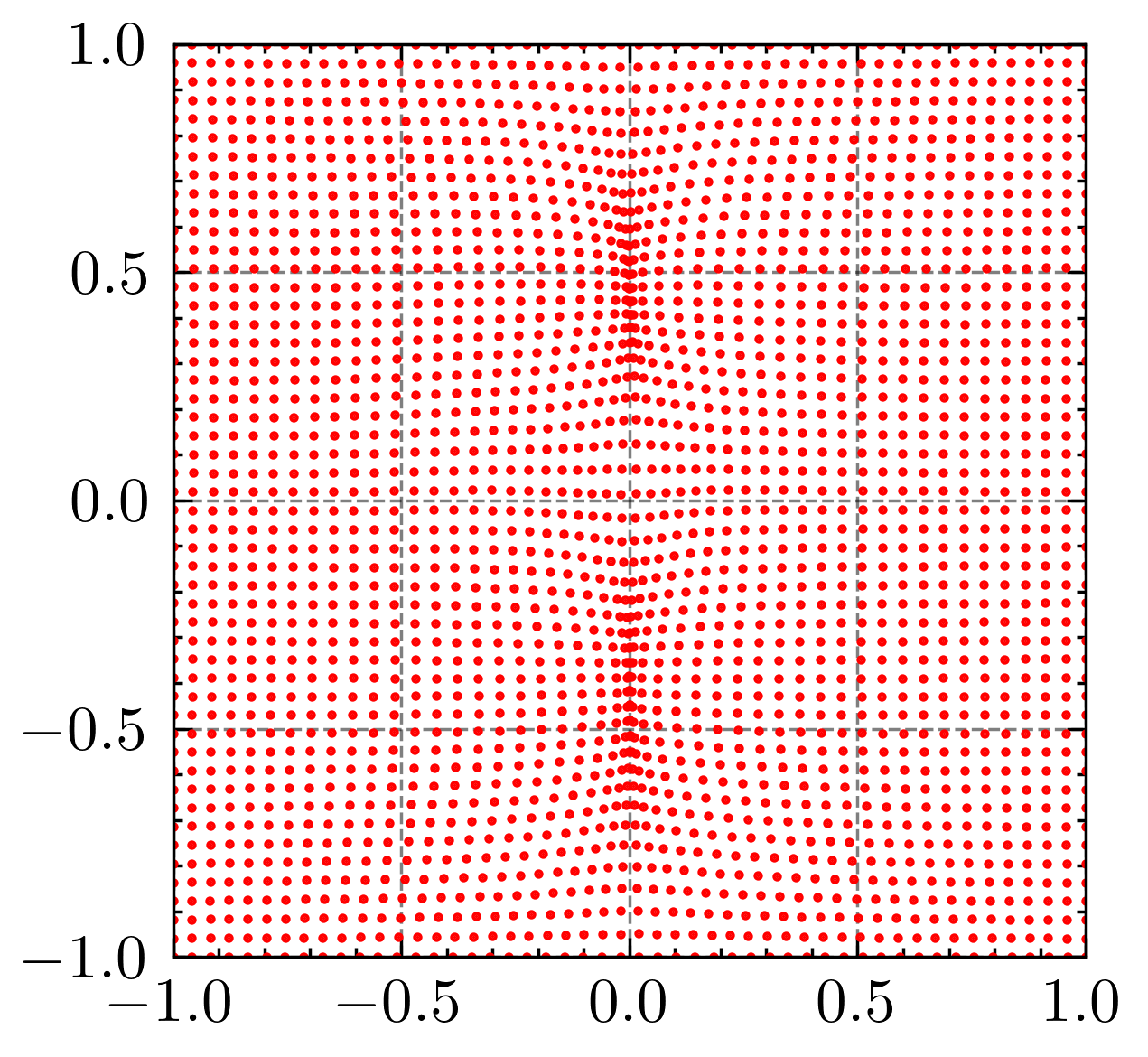}}
\subfloat[approximate solution]{\includegraphics[width = 0.31\textwidth]
{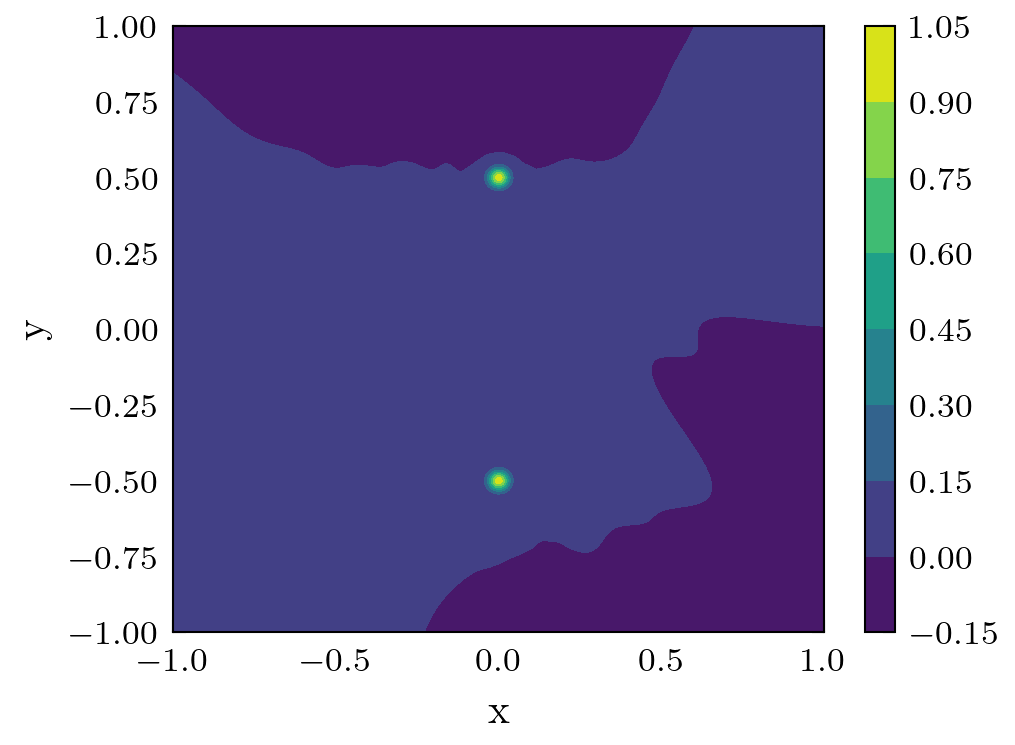}}
\subfloat[absolute error]{\includegraphics[width = 0.31\textwidth]{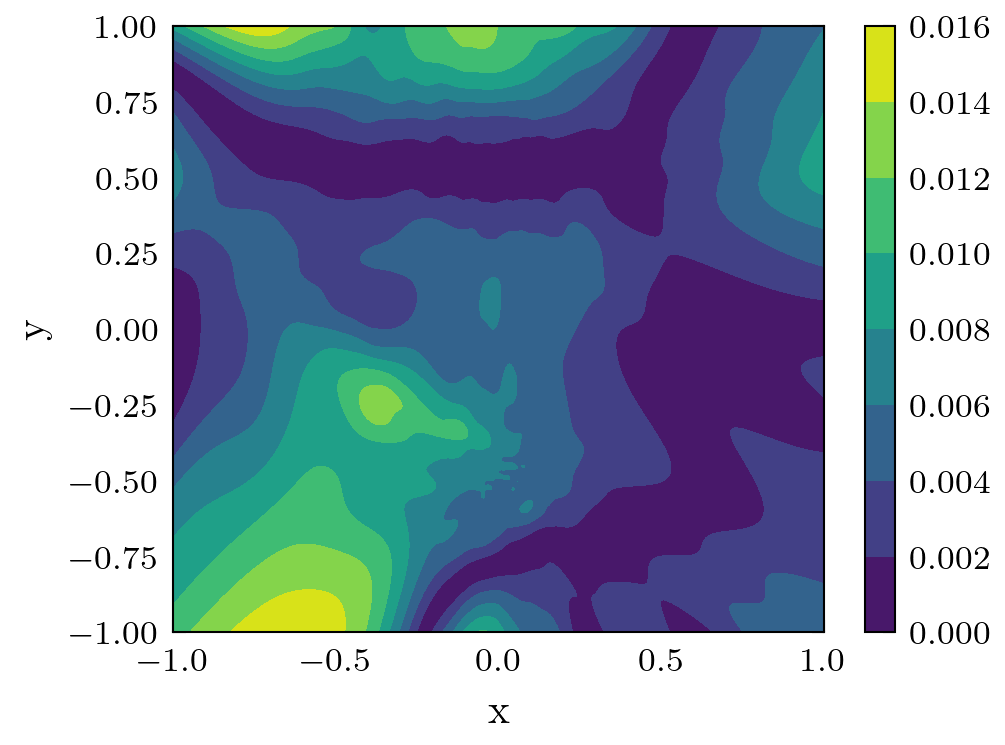}}
\caption{The results of MS-PINN for the two-dimensional Poisson equation with two peaks. (a) new training points obtained by MMPDE-Net; (b) the approximate solution $u^{\text{MS-PINN}}$ of the MS-PINN; (c) the absolute error $\vert u^* - u^{\text{MS-PINN}} \vert$.}
\label{fig:MMPDE_PINN_DoublePeaks2D}
\end{figure}

\begin{figure}[htbp]
\centering
\subfloat[performance of $e_\infty(u)$ during training]{\includegraphics[width = 0.40\textwidth]{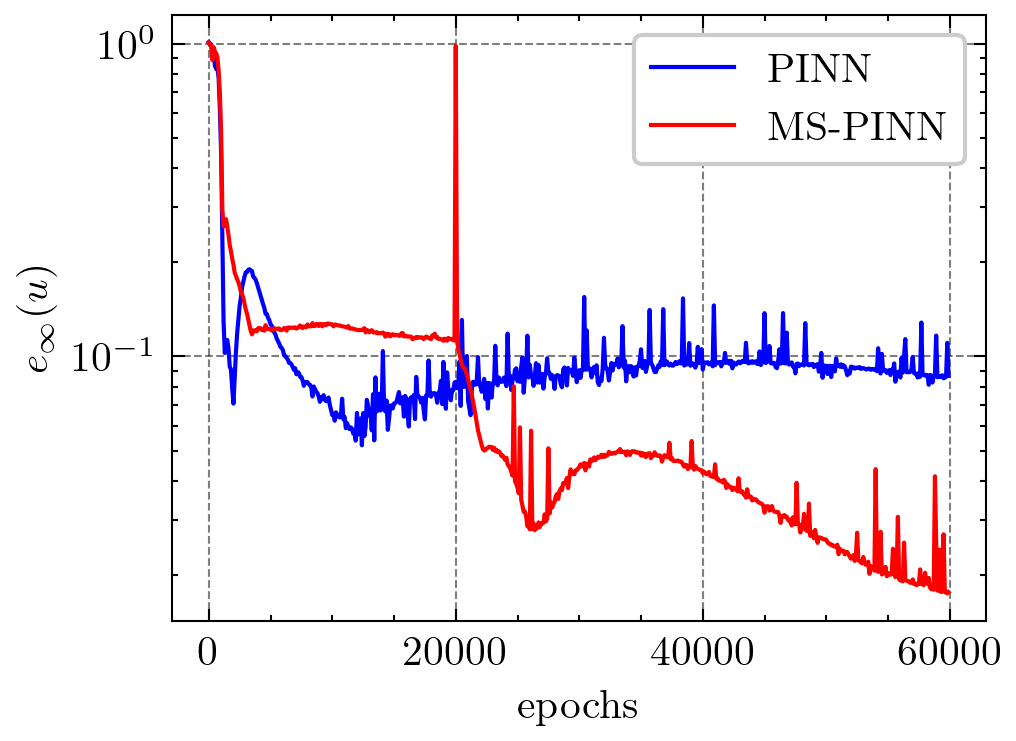}} \quad \quad \quad
\subfloat[performance of $e_2(u)$ during training]{\includegraphics[width = 0.40\textwidth]
{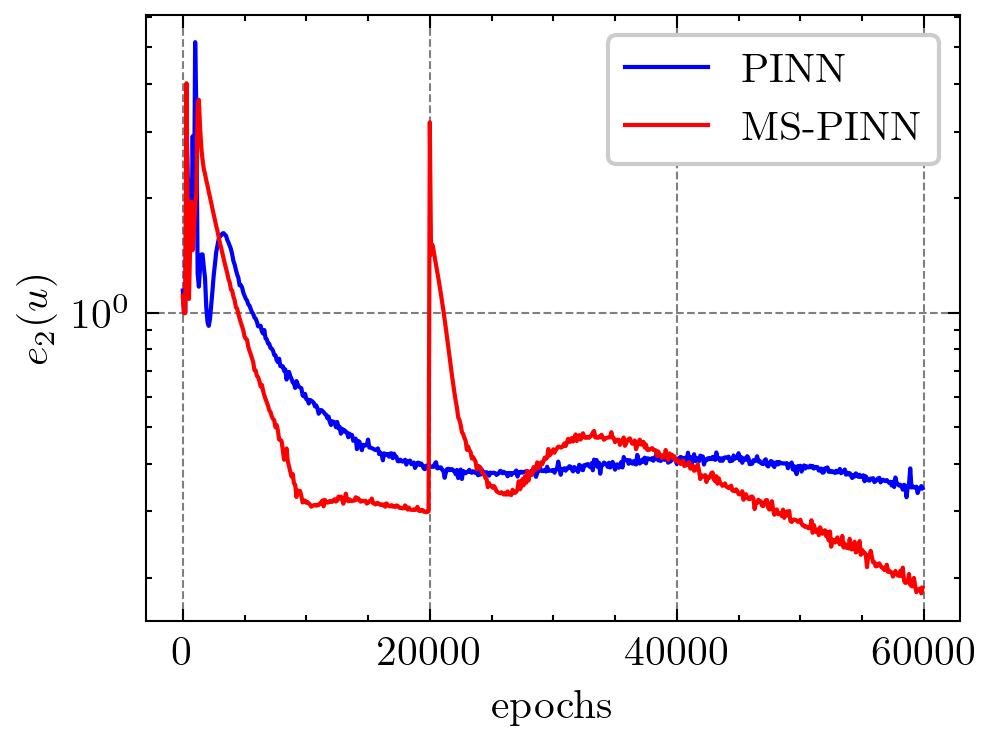}}
\caption{(a) the performance of relative error $e_\infty(u)$ with different training epochs; (b) the performance of relative error $e_2(u)$ with different training epochs.}
\label{fig:DoublePeaks2D_ErrorEpochs}
\end{figure}
The numerical results of PINN are given in Fig \ref{fig:PINN_DoublePeaks2D}. The distribution of the training points generated by MMPDE-Net and the approximation solution of MS-PINN are shown in Fig \ref{fig:MMPDE_PINN_DoublePeaks2D}. The performance of PINN and MS-PINN in terms of relative errors  during the training process is given in Fig \ref{fig:DoublePeaks2D_ErrorEpochs}, which shows that MS-PINN behaves better than PINN.

% @@@@@@@@@@@@@@@@@@@@@@@@@@
\subsection{One-dimensional Burgers equation}
\label{sec:Burgers_1D}

\subsubsection{Forward problem}
\label{sec:Burgers_1D_Forward}
Consider the following one-dimensional Burgers equation \cite{PINN}

\begin{equation}
	\label{eq:1d_Burgers}
	\hspace{-0.3cm}
	\begin{array}{r@{}l}
		\left\{
		\begin{aligned}
			& u_t+uu_{x}-\frac{0.01}{\pi}u_{xx} = 0, \quad  x \in (-1,1), \quad  t \in (0,1), \\
			& u(x,0) = -\sin(\pi x), \quad  x \in [-1,1], \\
            & u(-1,t) = u(1,t) = 0, \quad  t \in [0,1], \\
		\end{aligned}
		\right.
	\end{array}
\end{equation}
the analytical solution is given in \cite{basdevant1986spectral}. 
\begin{figure}[htbp]
\centering
\subfloat[$u^{\text{PINN}}$]{\includegraphics[width = 0.33\textwidth]{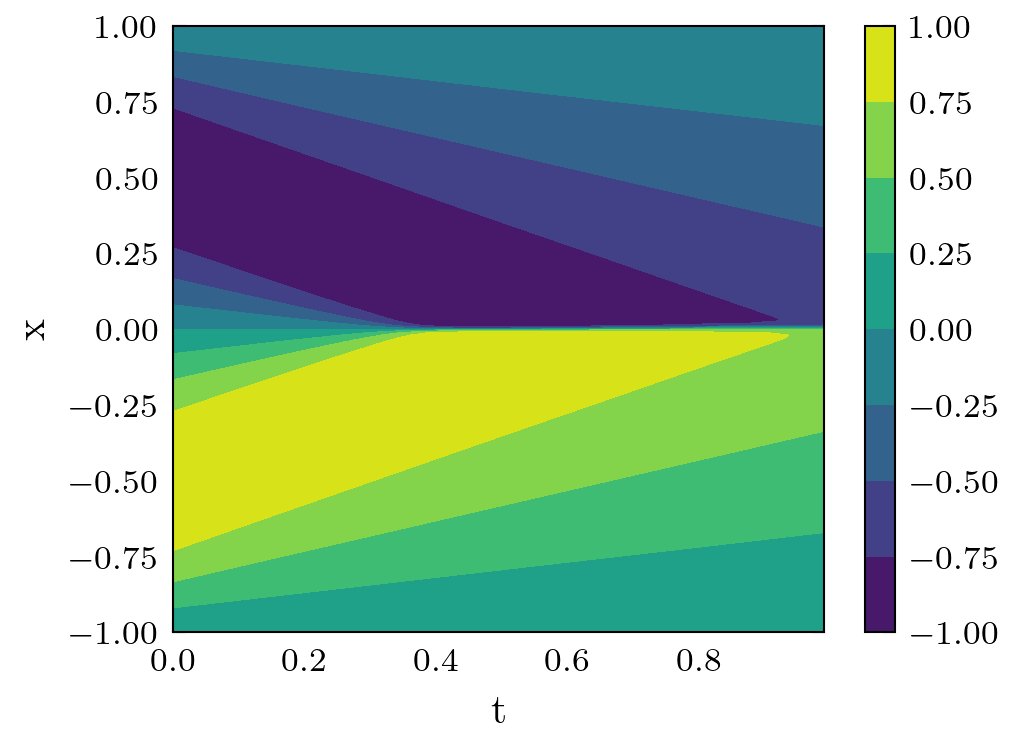}}
\subfloat[$u^{\text{MSPINN}}$]{\includegraphics[width = 0.33\textwidth]
{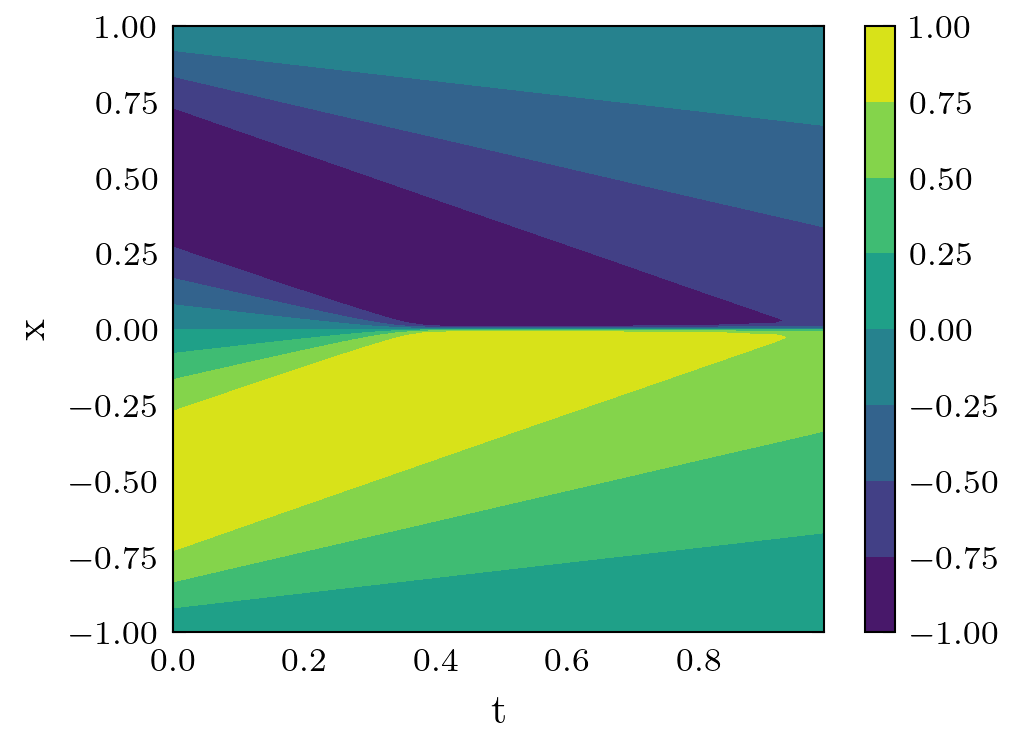}}
\subfloat[$u^{*}$]{\includegraphics[width = 0.33\textwidth]{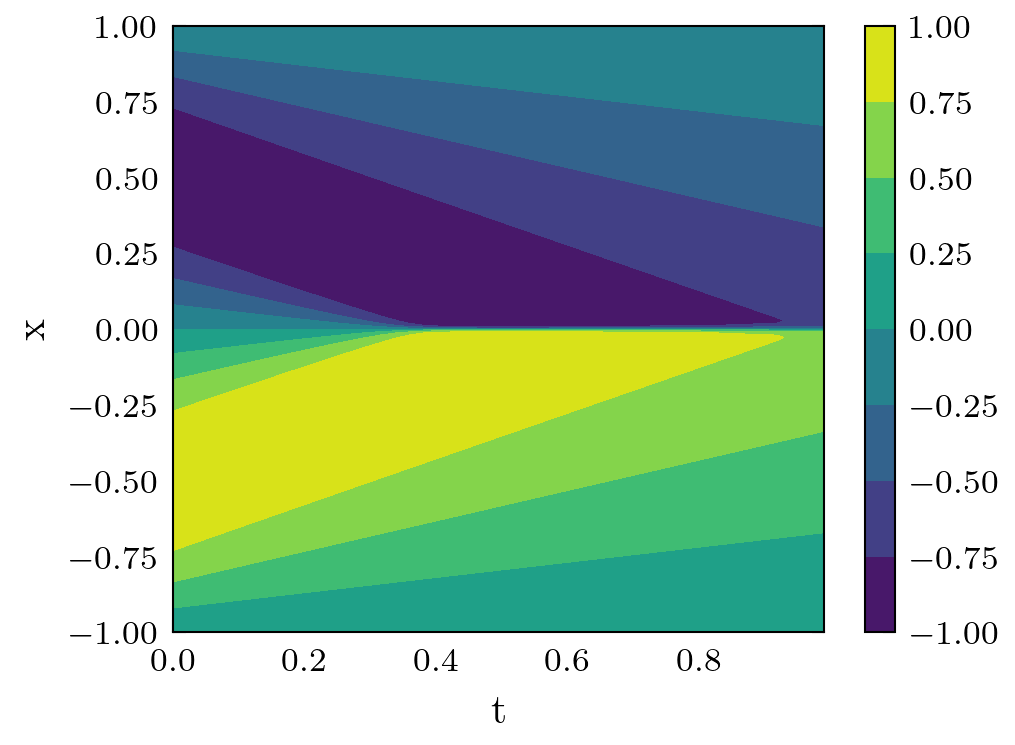}}
\caption{The solutions of the forward problem for the one-dimensional Burgers equation. (a) the solution $u^{\text{PINN}}$ predicted by PINN; (b) the solution $u^{\text{MSPINN}}$ predicted by MSPINN; (c) the analytic solution $u^{*}$.}
\label{fig:analyticsolution_Burgers1D}
\end{figure}
The monitor function that we use in MMPDE-Net is 

\begin{equation}
    \label{eq:1d_Bugers_monitorfunction}
    \hspace{-0.3cm}
    \begin{array}{r@{}l}
        \begin{aligned}
            w = \sqrt{1+u^2+(0.5u_x)^2}.
        \end{aligned}
    \end{array}
\end{equation}

In $[-1,1] \times [0,1]$, we sample $152 \times 100$ points as the training set and $256 \times 100$  points as the test set. The settings of main parameters are in Table  \ref{tab:parameter-MMPDENet}.
%Refer to Table \ref{tab:parameter-MMPDENet} for the setting of key parameters of MSPINN.
We compare the numerical results of MS-PINN  with the results obtained by PINN trained 60000 epochs to show the effectiveness of our method. 
\begin{figure}[htbp]
\centering
\subfloat[new distribution of sampling points]{\includegraphics[width = 0.40\textwidth]{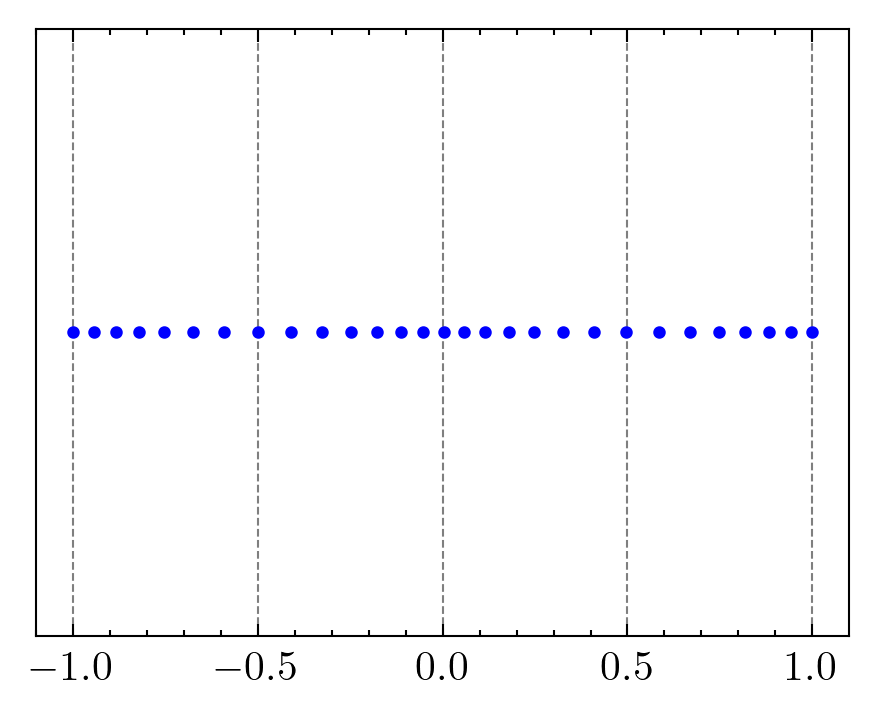}}
\subfloat[performance of the loss function]{\includegraphics[width = 0.40\textwidth]
{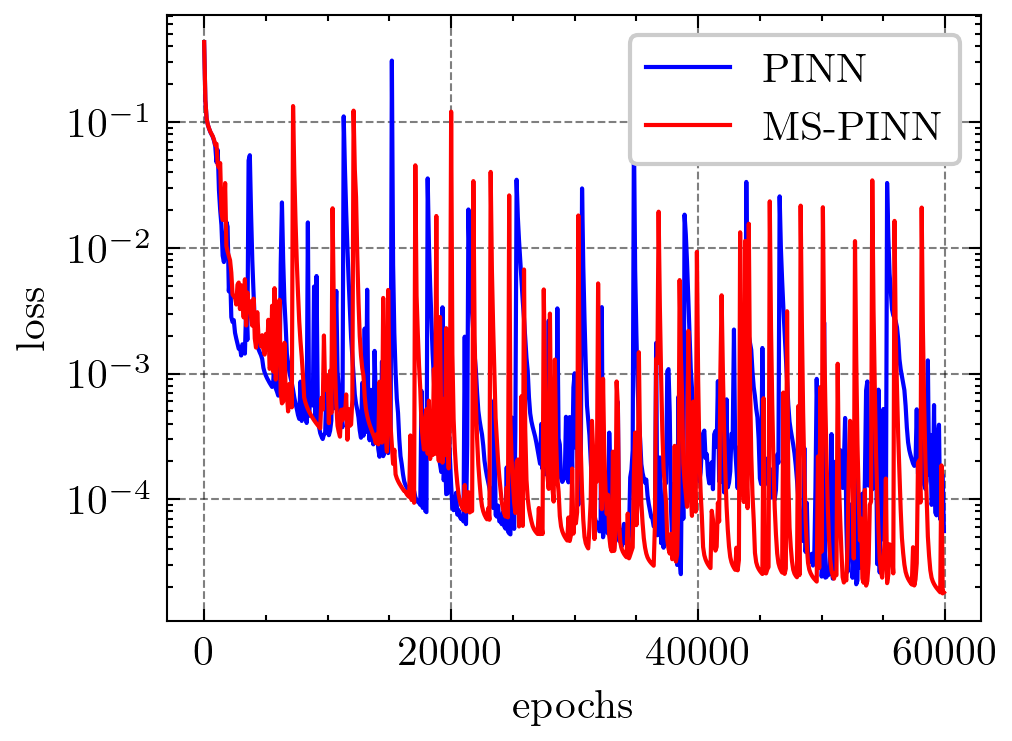}}
\caption{(a) the distribution of new training points obtained by MMPDE-Net for forward Burgers equation \eqref{eq:1d_Burgers}; (b) the loss function with different training epochs.}
\label{fig:Newmesh_Burgers1D}
\end{figure}

\begin{figure}[htbp]
\centering
\subfloat[t=0.2]{\includegraphics[width = 0.24\textwidth]{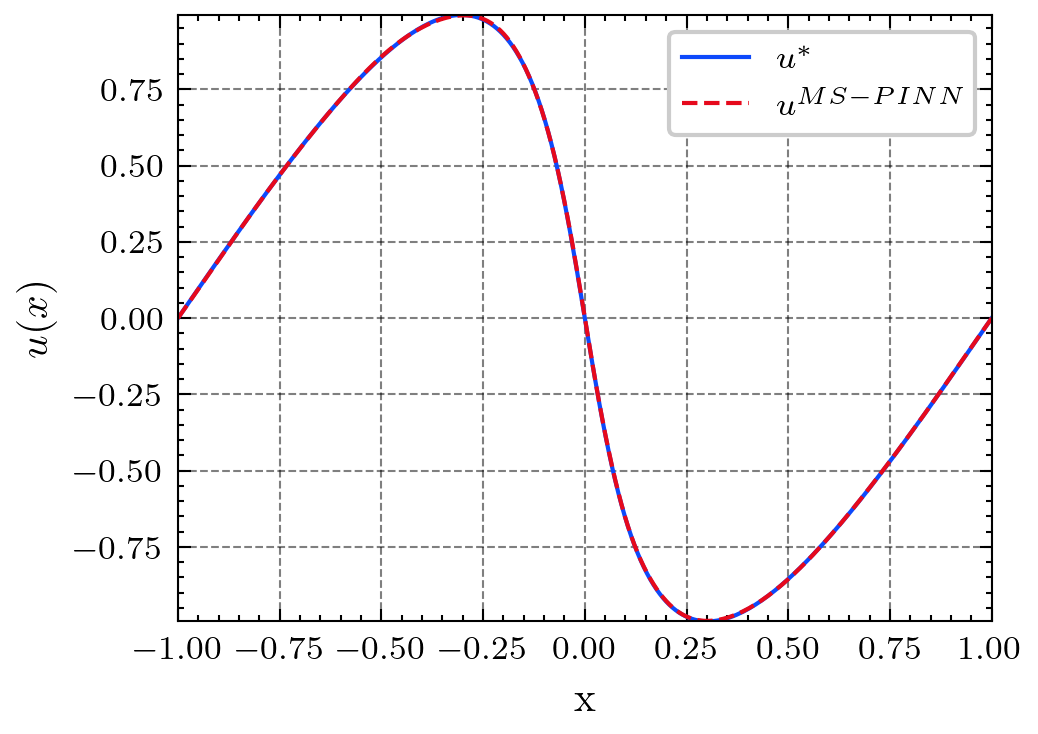}}
\subfloat[t=0.4]{\includegraphics[width = 0.24\textwidth]
{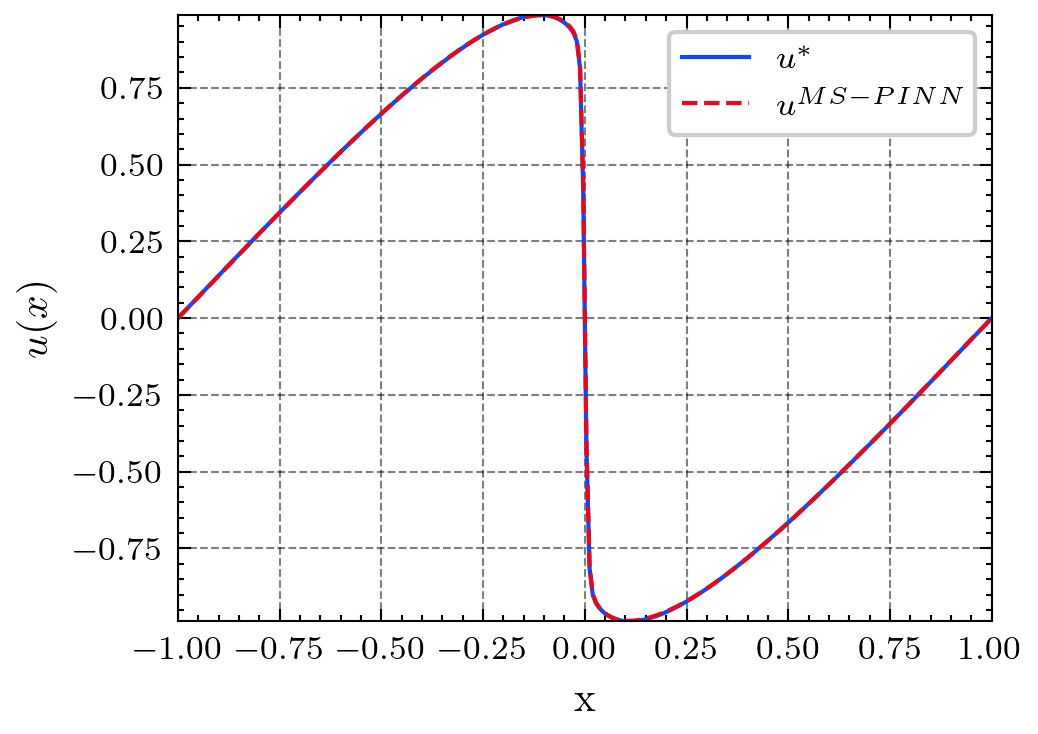}}
\subfloat[t=0.6]{\includegraphics[width = 0.24\textwidth]{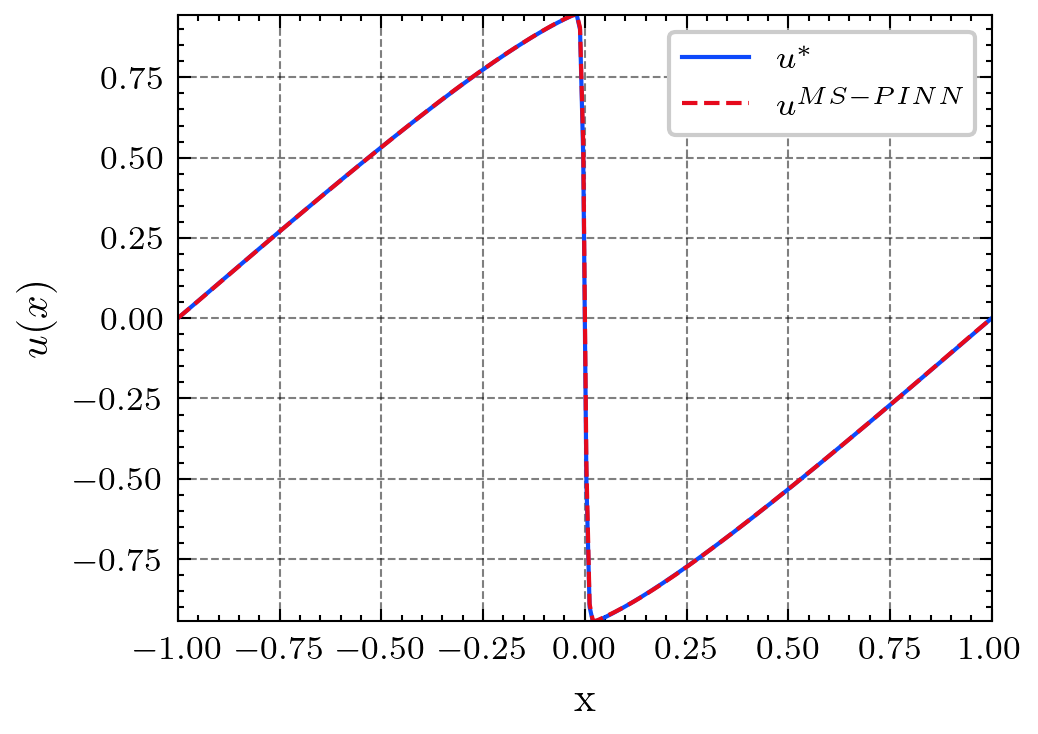}}
\subfloat[t=0.8]{\includegraphics[width = 0.24\textwidth]{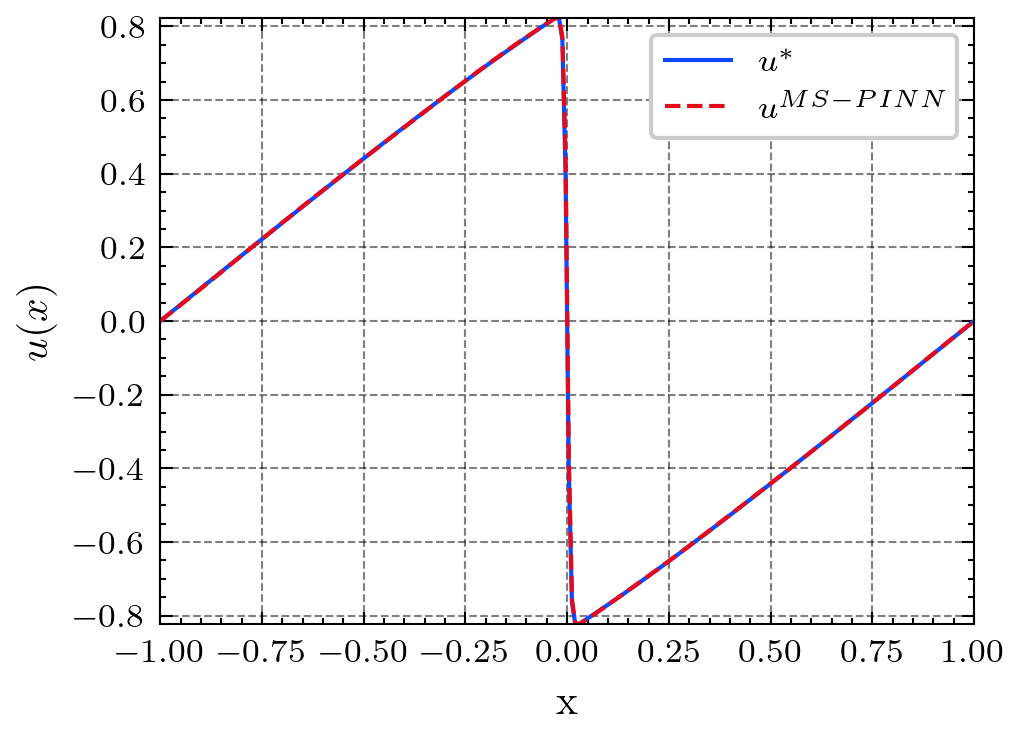}}
\caption{The results of MS-PINN at different times for the one-dimensional Burgers equation. (a) t=0.2; (b) t=0.4; (c) t=0.6; (d) t=0.8.}
\label{fig:MMPDEPINN_Burgers1D}
\end{figure}

The solution $u^{\text{PINN}}$ predicted by PINN  is shown in Fig \ref{fig:analyticsolution_Burgers1D}(a). MMPDE-Net is implemented and  the distribution of new training points is shown in Fig \ref{fig:Newmesh_Burgers1D}(a), where we only show 27 points for clear viewing. Fig \ref{fig:Newmesh_Burgers1D}(b) shows the variation of  the loss function for PINN and MS-PINN with different training epochs, where it is noticed that MS-PINN can attain lower loss values. The solution $u^{\text{MSPINN}}$ predicted by MS-PINN  is shown in Fig \ref{fig:analyticsolution_Burgers1D}(b) and the analytical solution is shown in Fig \ref{fig:analyticsolution_Burgers1D}(c). 
The comparison between the analytic solutions and the approximation solutions of MS-PINN at different time is shown in Fig \ref{fig:MMPDEPINN_Burgers1D}. The relative errors at different time obtained by PINN and MS-PINN are presented in Fig \ref{fig:Burgers_1D_Errorline}, which again shows that our method can achieve more accurate results.

\begin{figure}[htbp]
\centering
\subfloat[$e_\infty(u)$]{\includegraphics[width = 0.40\textwidth]{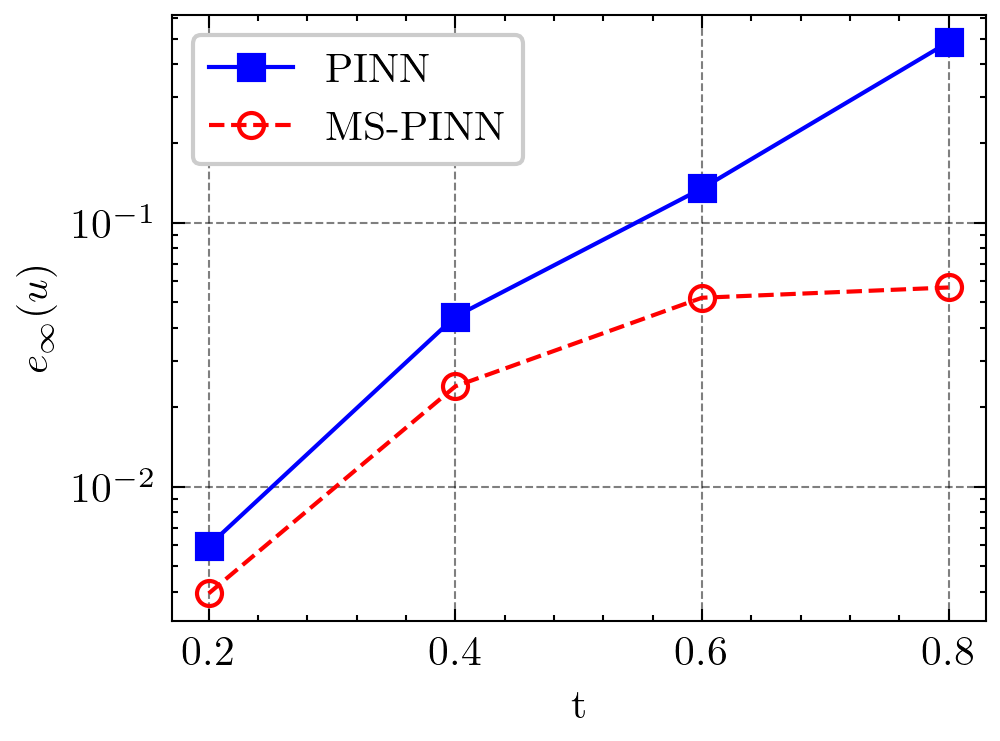}} \quad \quad \quad
\subfloat[$e_2(u)$]{\includegraphics[width = 0.40\textwidth]
{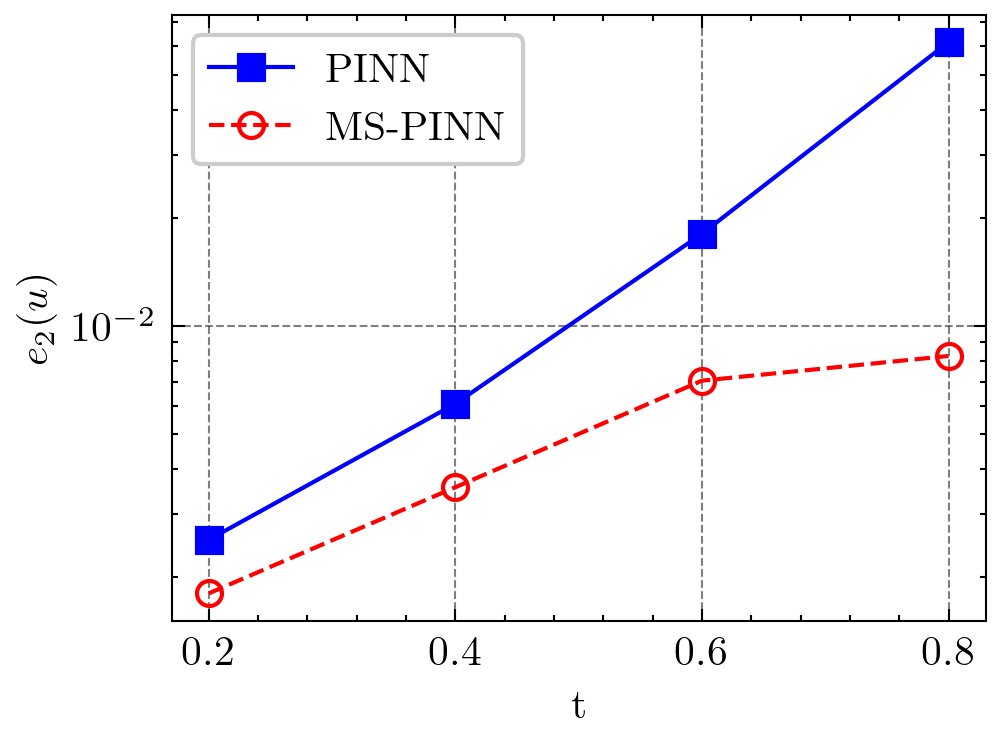}}
\caption{Comparison of the relative errors obtained by PINN and MS-PINN at different time.}
\label{fig:Burgers_1D_Errorline}
\end{figure}
%%%%%%%%%%%%%%%%%%%%%%%%%%%%%%%%%%%%%%%%%%%%%%%%%%%%%%%%%%%%%%%%%%%%%%%%%%%%%%%%%%%
\subsubsection{Inverse problem}
\label{sec:Burgers_1D_Inverse}
Consider the following one-dimensional Burgers equation \cite{PINN}
\begin{equation}
	\label{eq:1d_Burgers_Inverse}
	\hspace{-0.3cm}
	\begin{array}{r@{}l}
		\left\{
		\begin{aligned}
			& u_t+\lambda_1 uu_{x}-\lambda_2 u_{xx} = 0, \quad  x \in (-1,1), \quad t \in (0,1), \\
			& u(x,0) = -\sin(\pi x), \quad  x \in [-1,1], \\
            & u(-1,t) = u(1,t) =0, \quad  t \in [0,1], \\
		\end{aligned}
		\right.
	\end{array}
\end{equation}
where $\lambda_1$ and $\lambda_2$ are the unknown parameters.
\begin{figure}[htbp]
\centering
\subfloat[$u^{\text{PINN}}$]{\includegraphics[width = 0.33\textwidth]{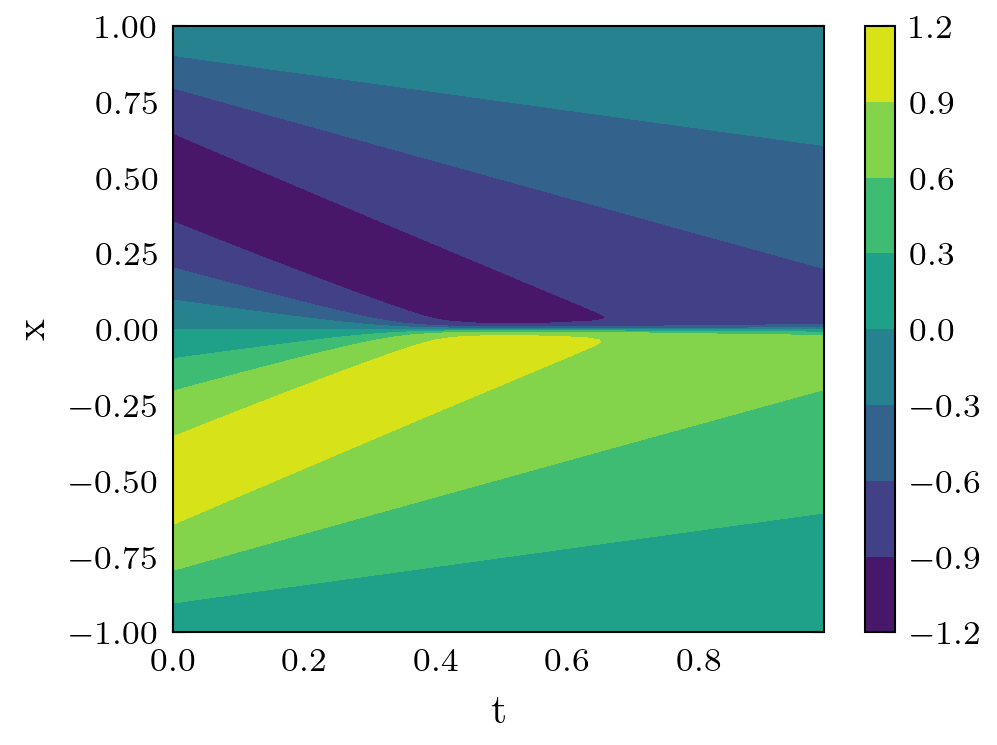}}
\subfloat[$u^{\text{MSPINN}}$]{\includegraphics[width = 0.33\textwidth]
{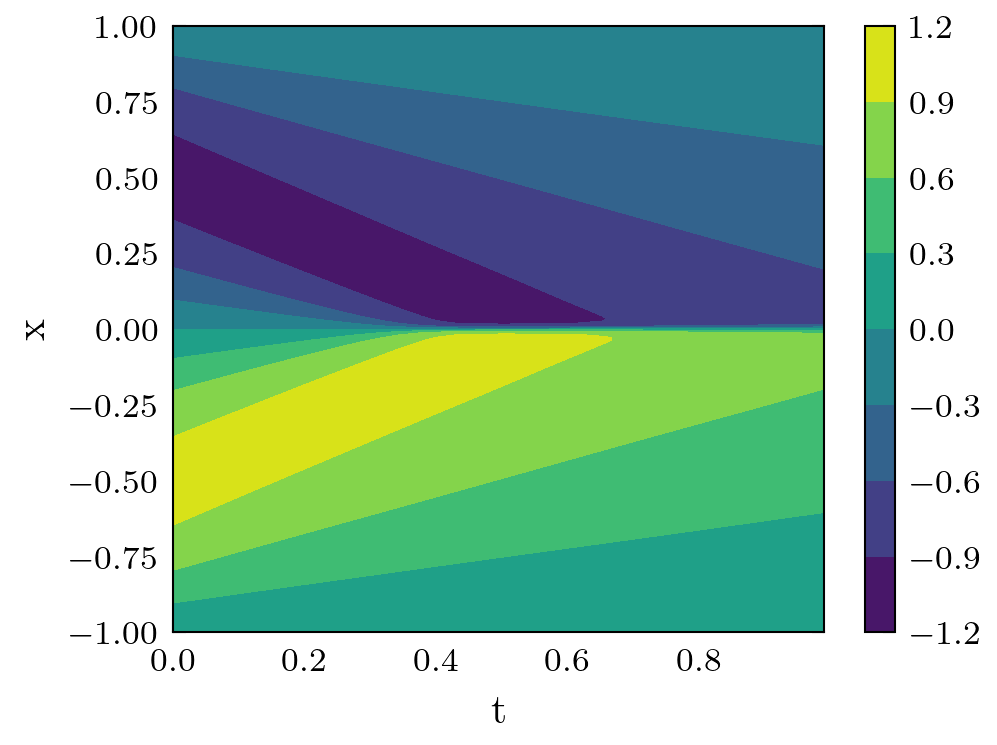}}
\subfloat[$u^{*}$]{\includegraphics[width = 0.33\textwidth]{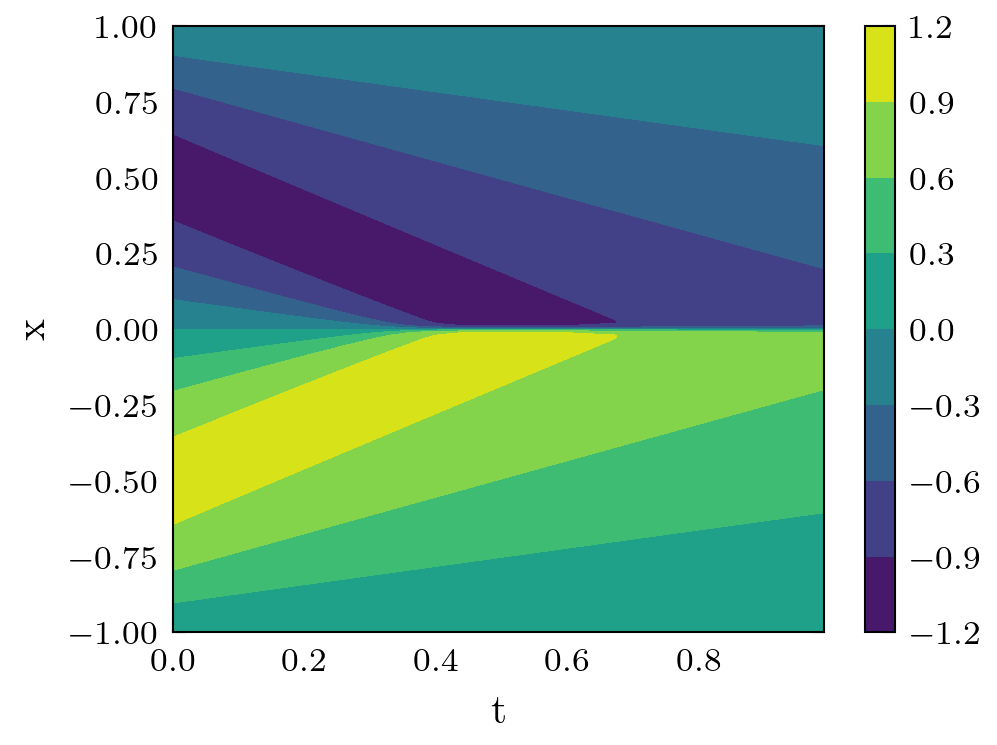}}
\caption{The solutions of the inverse problem for the one-dimensional Burgers equation. (a) the solution $u^{\text{PINN}}$predicted by PINN; (b) the solution $u^{\text{MSPINN}}$ predicted by MSPINN ; (c) the analytic solution $u^{*}$.}
\label{fig:analyticsolution_Burgers1D_Inverse}
\end{figure}
The monitor function that we use in MMPDE-Net is 

\begin{equation}
    \label{eq:1d_Bugers_Inverse_monitorfunction}
    \hspace{-0.3cm}
    \begin{array}{r@{}l}
        \begin{aligned}
            w = \sqrt{1+(0.5u_x)^2}.
        \end{aligned}
    \end{array}
\end{equation}

In $ [-1,1] \times [0,1]$, we sample $202\times 100$  points as the training set and $256 \times 100$ points as the test set. In addition, the analytical solution corresponding to $\lambda_1^* = 1$ and  $\lambda_2^* = \frac{0.01}{\pi}$ is extracted for data-driven values at the uniformly sampled 200 points. The main setting parameters are shown in Table  \ref{tab:parameter-MMPDENet}.
%In PINN, there is 8 hidden layers with 20 neurons per layer. 
For the MS-PINN, we train PINN 70000 epochs and  MMPDE-Net 20000 epochs, where 20000 epochs in the pre-training of PINN  and 50000 epochs in the formal PINN training.  It is worth noting that the LBFGS method is used for fine-tuning in the last 10000 training epochs.
We compare the numerical results of MS-PINN  with the results obtained by PINN trained 70000 epochs to show the effectiveness of our method.

\begin{figure}[htbp]
\centering
\subfloat[new distribution of sampling points]{\includegraphics[width = 0.40\textwidth]{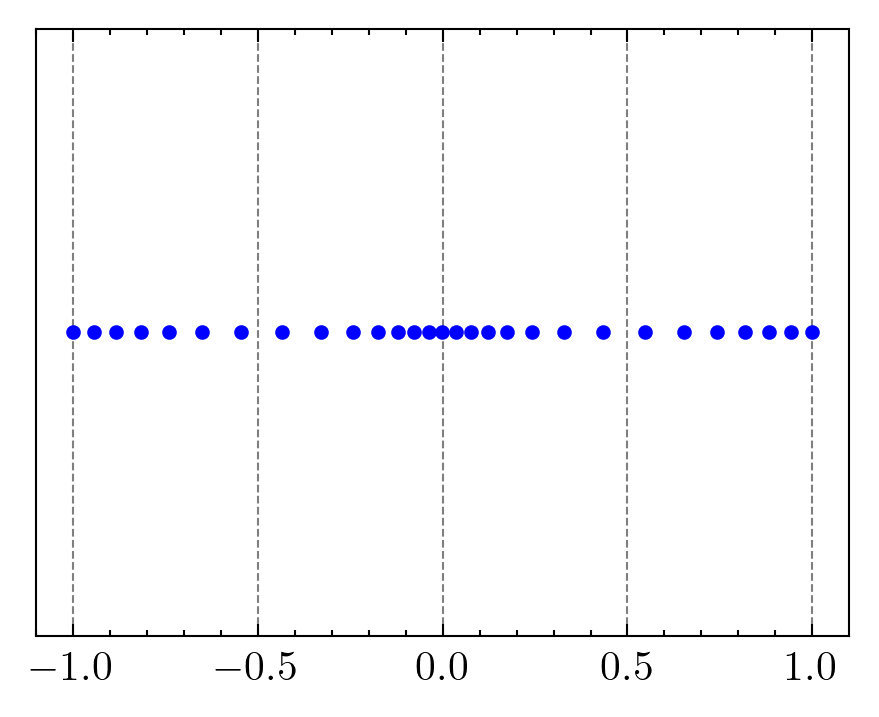}}
\subfloat[performance of the loss function]{\includegraphics[width = 0.40\textwidth]
{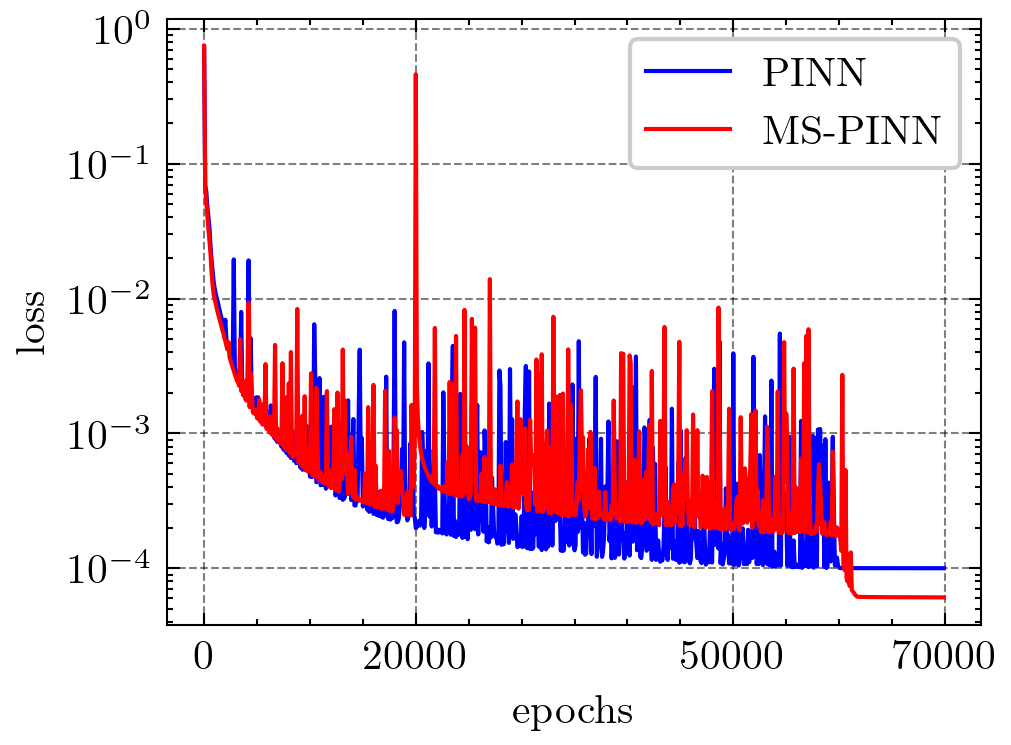}}
\caption{(a) the distribution of new training points obtained by MMPDE-Net for inverse Burgers equation \eqref{eq:1d_Burgers_Inverse}; (b) the performance of loss function with different training epochs.}
\label{fig:Newmesh_Burgers1DR}
\end{figure}

\begin{figure}[htbp]
\centering
\subfloat[t=0.2]{\includegraphics[width = 0.24\textwidth]{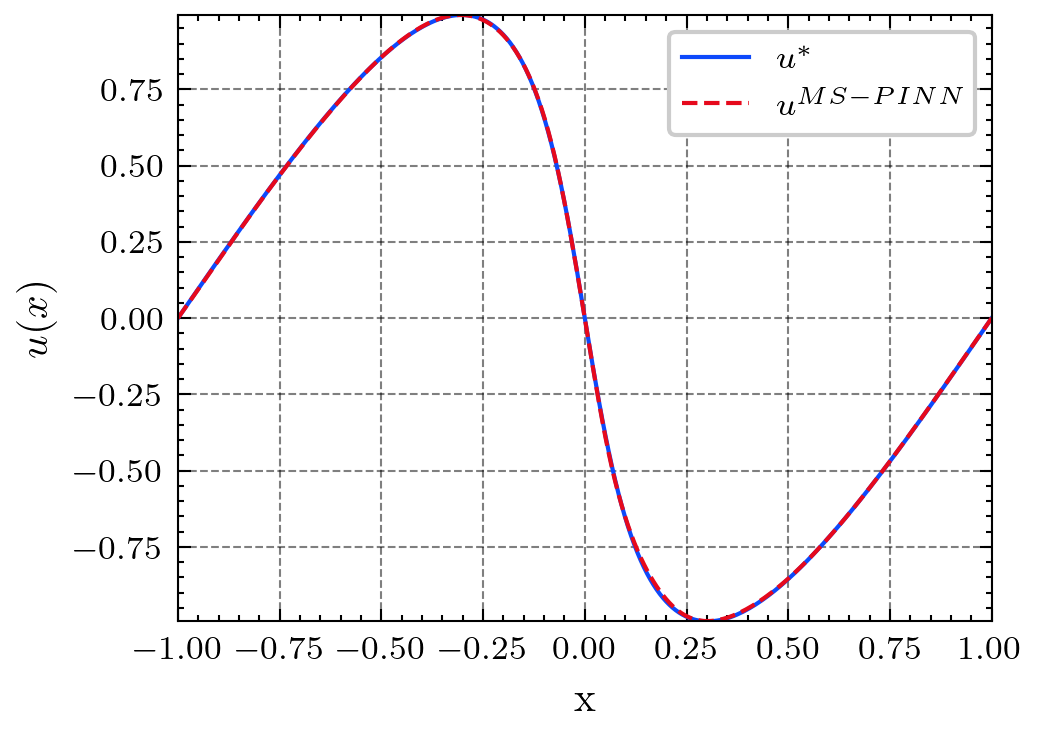}}
\subfloat[t=0.4]{\includegraphics[width = 0.24\textwidth]
{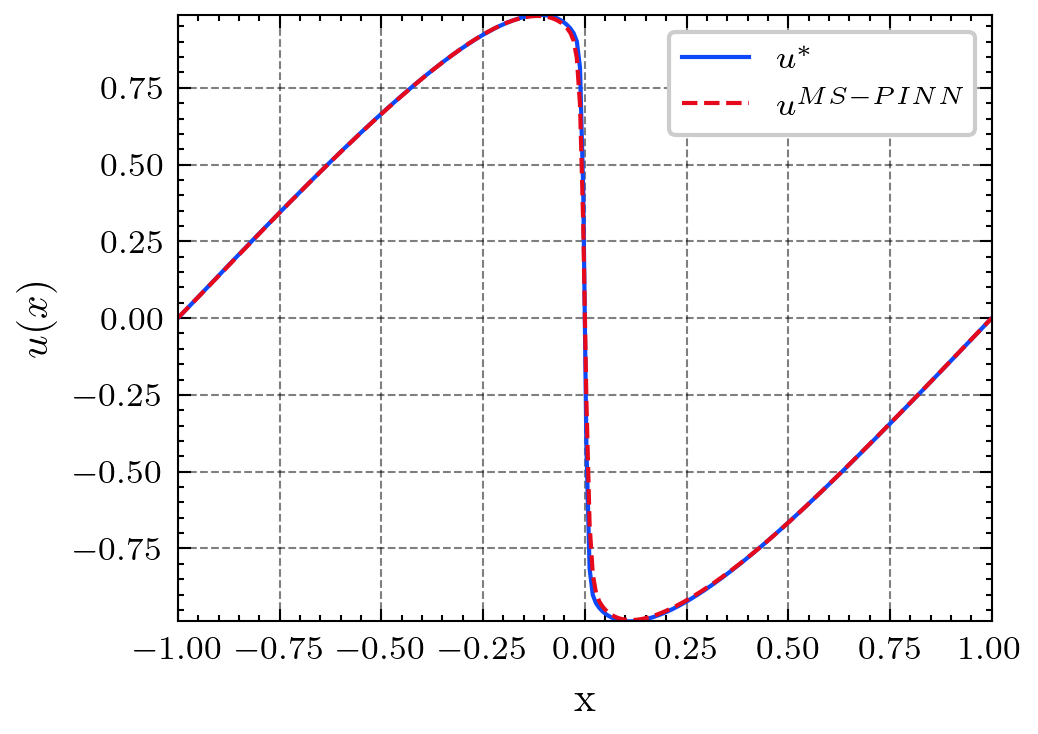}}
\subfloat[t=0.6]{\includegraphics[width = 0.24\textwidth]{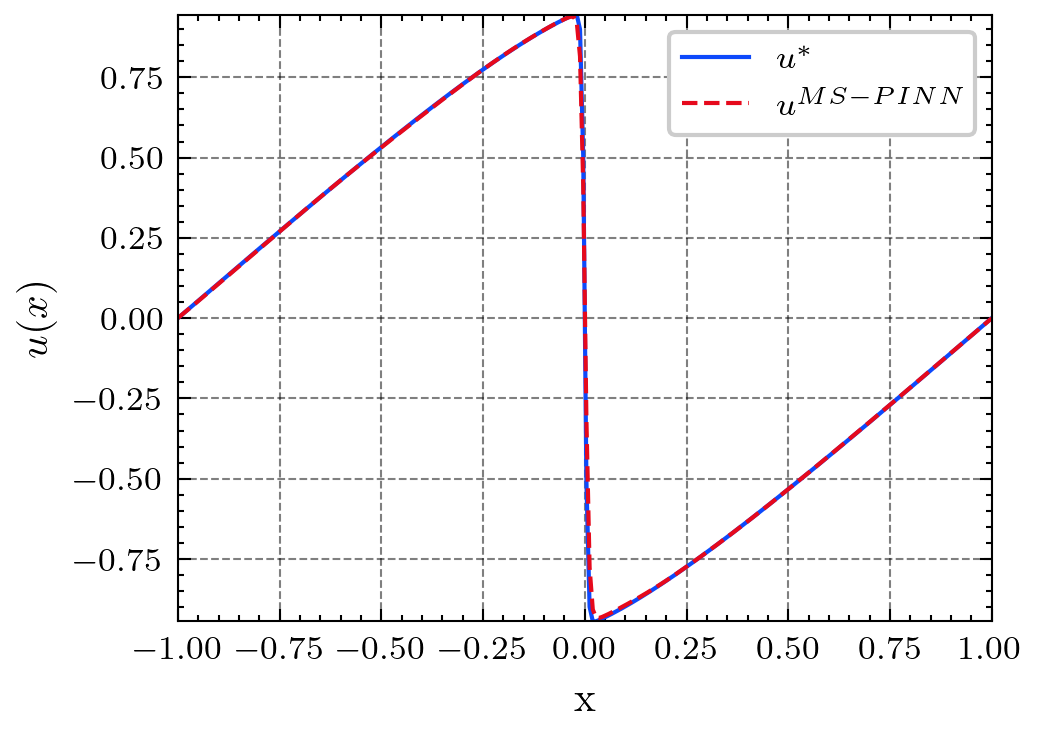}}
\subfloat[t=0.8]{\includegraphics[width = 0.24\textwidth]{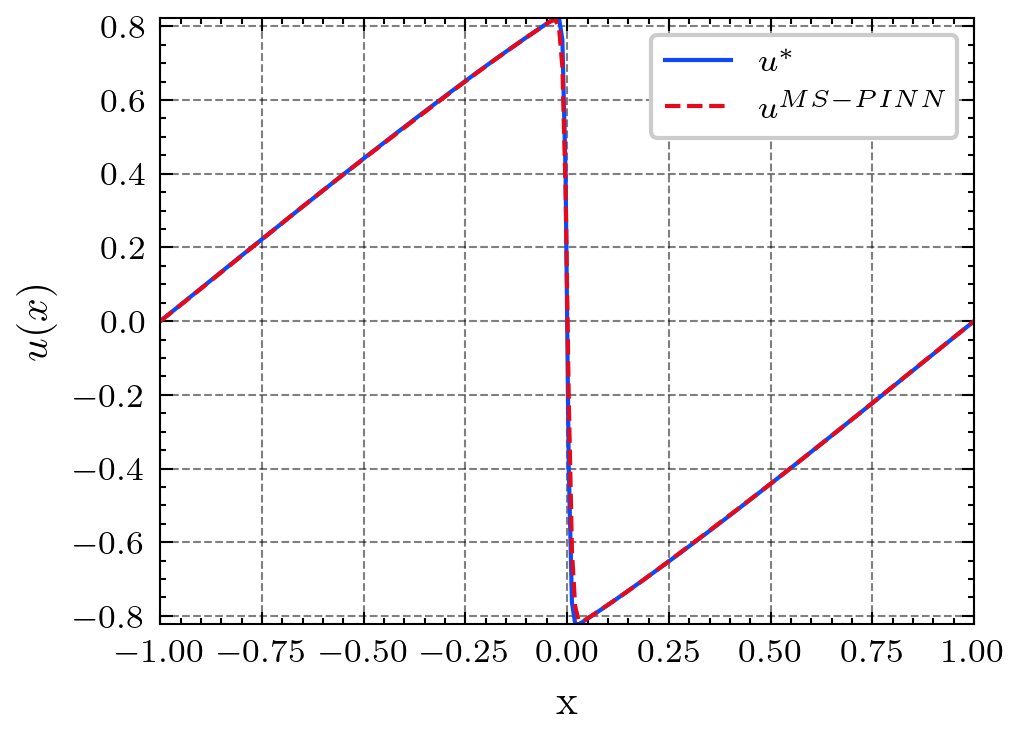}}
\caption{The results of MS-PINN at different times for the one-dimensional inverse Burgers equation. (a) t=0.2; (b) t=0.4; (c) t=0.6; (d) t=0.8.}
\label{fig:MMPDEPINN_Burgers1D_Inverse}
\end{figure}

\begin{table}[h]
\scriptsize
\centering
\caption{
Comparison of the absolute errors of PINN and MS-PINN on $\lambda_1$ and $\lambda_2$, where $\lambda^{\text{NN}}$ is denoted as the prediction of the neural networks. 
}
\setlength{\tabcolsep}{8.mm}{
\begin{tabular}{|c|c|c|}%{p{1cm}p{2.2cm}p{1.2cm}p{1.2cm}p{1.2cm}p{1.2cm}p{1.2cm}p{1.6cm}}
\hline\noalign{\smallskip}
   absolute error&  PINN & MS-PINN  \\
\hline
$\vert \lambda_1^{\text{NN}}  -\lambda_1^* \vert$ & $ 2.932\times 10^{-3} $   &$4.379 \times 10^{-4}$ \\
\hline
$ \vert \lambda_2^{\text{NN}}- \lambda_2^* \vert $  & $2.629 \times 10^{-3}$  &   $1.427 \times 10^{-3}$ \\
\hline
\end{tabular}
}
\label{tab:Burgers1D_Inverse_LambdaResults} 
\end{table}
The solution $u^{\text{PINN}}$ predicted by PINN is shown in Fig \ref{fig:analyticsolution_Burgers1D_Inverse}(a). MMPDE-Net is implemented and  the distribution of new training points is shown in Fig \ref{fig:Newmesh_Burgers1DR}(a), where we use only 27 points for clear viewing. Fig \ref{fig:Newmesh_Burgers1DR}(b) shows the variation of the loss function for PINN and MS-PINN with different training epochs, where MS-PINN can attain lower loss values when using LBFGS to optimize. The solution $u^{\text{MSPINN}}$ predicted by MS-PINN  is shown in Fig \ref{fig:analyticsolution_Burgers1D_Inverse}(b) and the analytical solution is shown in Fig \ref{fig:analyticsolution_Burgers1D_Inverse}(c) .
The comparison between the analytic solutions and the approximation solutions of MS-PINN at different time is shown in Fig \ref{fig:MMPDEPINN_Burgers1D_Inverse}. The comparison of the absolute errors  of $\lambda_1$ and $\lambda_2$ obtained by PINN and MS-PINN is shown in Table \ref{tab:Burgers1D_Inverse_LambdaResults}. 
The relative errors at different time obtained by PINN and MS-PINN are presented in Fig \ref{fig:Burgers_1D_Inverse_Errorline}, which again demonstrates that our method can obtain more accurate results.

\begin{figure}[htbp]
\centering
\subfloat[$e_\infty(u)$]{\includegraphics[width = 0.40\textwidth]{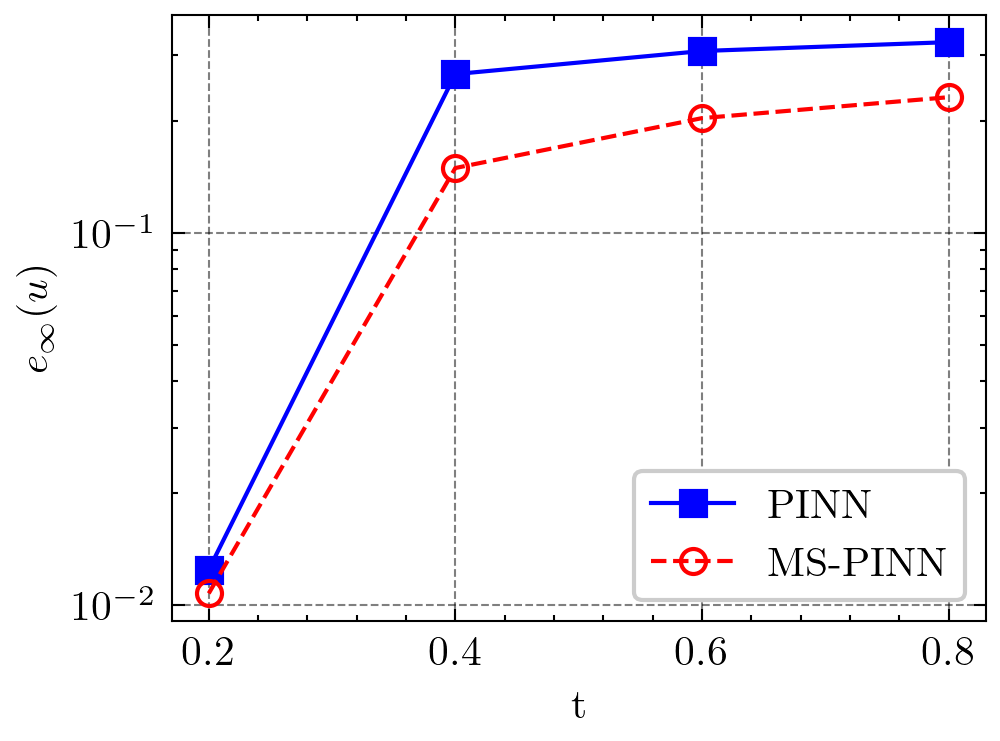}} \quad \quad \quad
\subfloat[$e_2(u)$]{\includegraphics[width = 0.40\textwidth]
{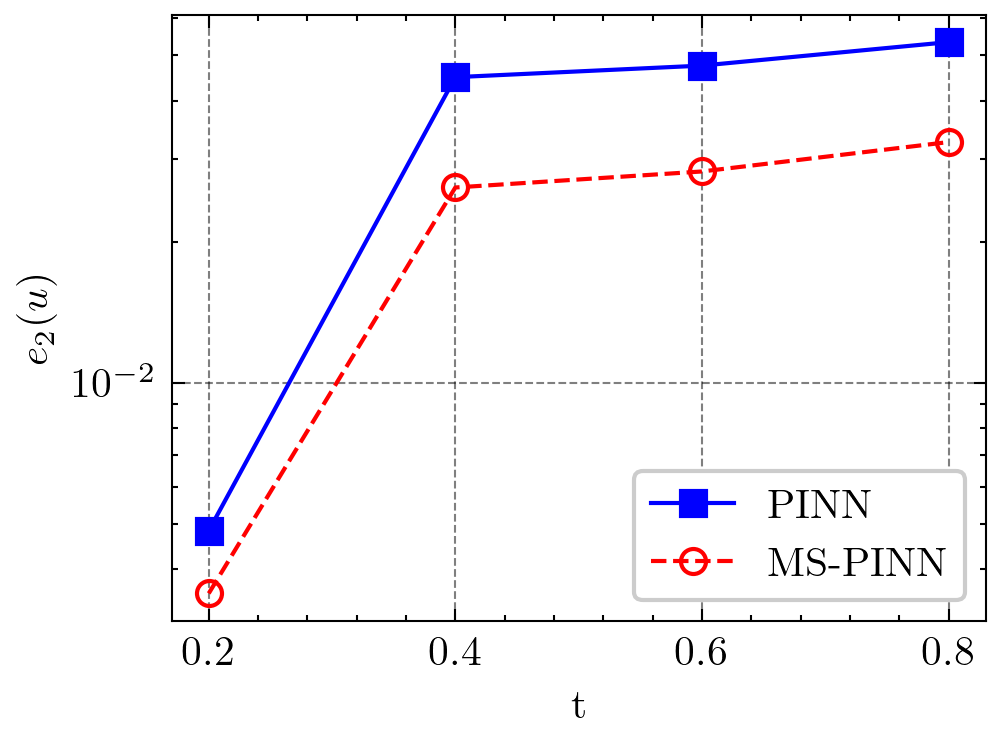}}
\caption{Comparison of the relative errors obtained by PINN and MS-PINN at different time.}
\label{fig:Burgers_1D_Inverse_Errorline}
\end{figure}

% @@@@@@@@@@@@@@@@@@@@@@@@@@% @@@@@@@@@@@@@@@@@@@@@@@@@@
% @@@@@@@@@@@@@@@@@@@@@@@@@@% @@@@@@@@@@@@@@@@@@@@@@@@@@
% @@@@@@@@@@@@@@@@@@@@@@@@@@% @@@@@@@@@@@@@@@@@@@@@@@@@@
\subsection{Two-dimensional Burgers equation}
\label{sec:Burgers_2D}
Now consider the following two-dimensional Burgers equation \cite{Burgers2d_example}
\begin{equation}
	\label{eq:2d_Burgers}
	\hspace{-0.3cm}
	\begin{array}{r@{}l}
		\left\{
		\begin{aligned}
			& u_t +uu_x + uu_y = 0 \quad (x,y) \in [0,4]^2 , \quad t \in [0,\frac{1.5}{\pi}], \\
			& u(x,0,t) = u(x,4,t) \quad(x,t) \in [0,4]\times[0,\frac{1.5}{\pi}],\\
            & u(0,y,t) = u(4,y,t) \quad(y,t) \in [0,4]\times[0,\frac{1.5}{\pi}],\\
            & u(x,y,0) = \sin\left(\frac{\pi}{2}(x+y)\right), \quad (x,y) \in [0,4]^2.\\
		\end{aligned}
		\right.
	\end{array}
\end{equation}
%The computational domain is $\Omega = [0,4]^2 \times [0,\frac{1.5}{\pi}]$. 
The reference solutions at different time are given by method of characteristics and Newton's method, which are shown in Fig \ref{fig:Solution_Burgers2D}(d). 
At $t = \frac{1}{\pi}$, there will be two shock waves on $x + y = 2$ and $x + y = 6$. The monitor function we use in MMPDE-Net is 
\begin{equation}
    \label{eq:2d_Burgers_monitorfunction}
    \hspace{-0.3cm}
    \begin{array}{r@{}l}
        \begin{aligned}
             w = \sqrt{1+ 10u^2_x + 10 u^2_y}.
        \end{aligned}
    \end{array}
\end{equation}

\begin{figure}[htbp]
\centering
\subfloat[$u^{\text{MSPINN}}$ at $ t=\frac{0.375}{\pi}$]{\includegraphics[width = 0.24\textwidth]
{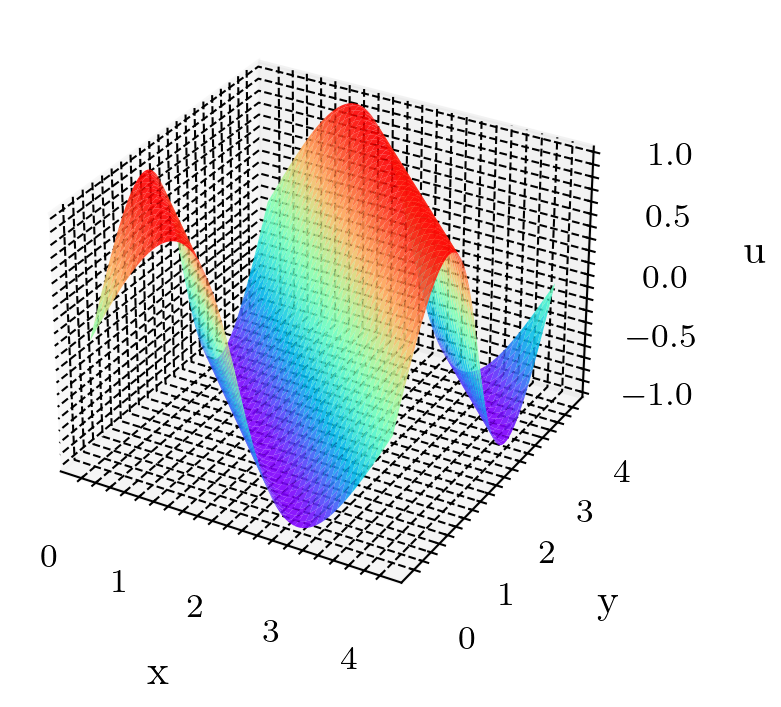}}
\subfloat[$u^*$ at $ t=\frac{0.375}{\pi}$]{\includegraphics[width = 0.24\textwidth]{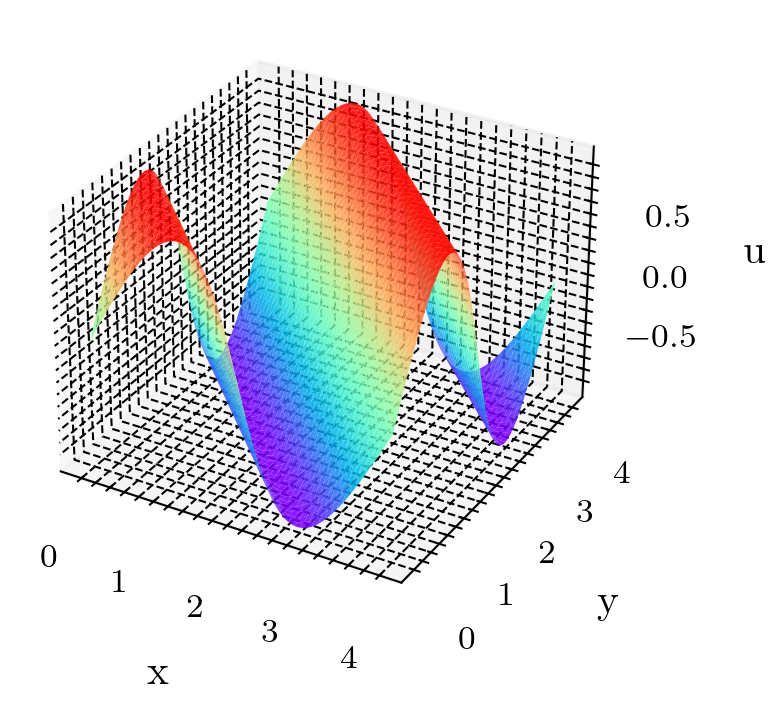}}
\subfloat[$u^{\text{MSPINN}}$ at $ t=\frac{1.5}{\pi}$]{\includegraphics[width = 0.24\textwidth]{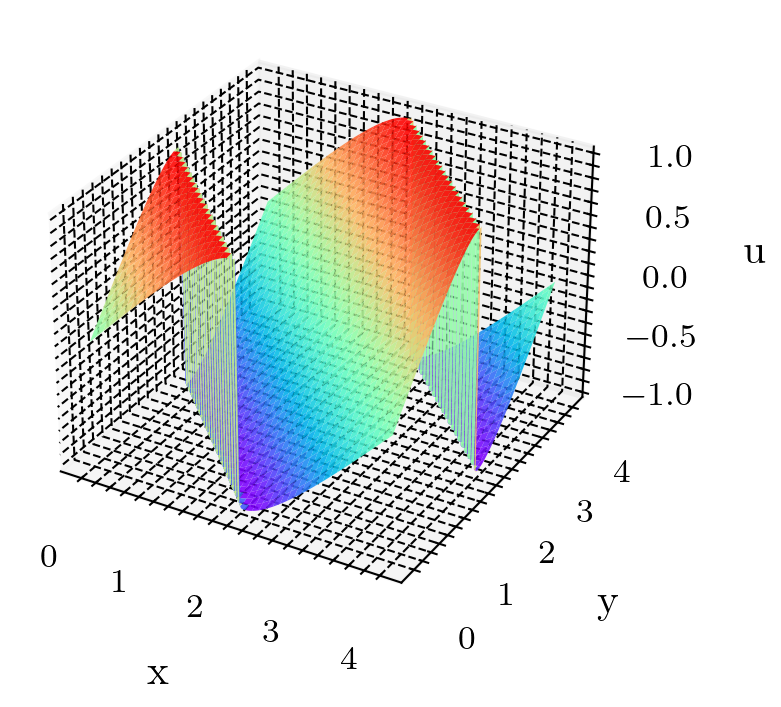}}
\subfloat[$u^*$ at $ t=\frac{1.5}{\pi}$]{\includegraphics[width = 0.24\textwidth]{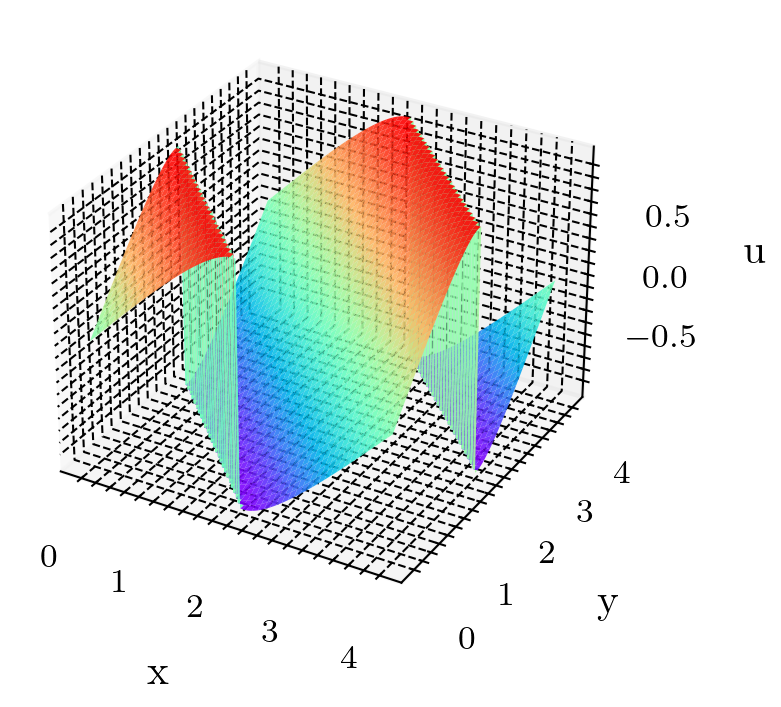}}
\caption{The solution predicted by MS-PINN and the reference solution for the two-dimensional Burgers equation at $t=\frac{0.375}{\pi}$ and $t=\frac{1.5}{\pi}$. }
\label{fig:Solution_Burgers2D}
\end{figure}

 We sample $100\times100 \times 50$ points in $(0,4)^2 \times [0,\frac{1.5}{\pi}]$ and 10000 on the boundary as the training set and $256\times256 \times 100$ points as the test set. The main setting parameters are shown in Table  \ref{tab:parameter-MMPDENet}.
%In addition, we also sampled  boundary training points.
%Here are some parameter settings for PINN, which has 8 hidden layers with 20 neurons per layer.
%, using tanh as the activation function. As for optimization, the Adam optimization method is used and the initial learning rate is 0.0001.
For the MS-PINN, we train PINN 130000 epochs and  MMPDE-Net 40000 epochs, where 20000 epochs in the pre-training of PINN  and 110000 epochs in the formal PINN training. We compare the numerical results of MS-PINN with the results obtained by PINN which trained 130000 epochs to show the effectiveness of our method.
%For comparison, we compare our results with the results of PINN using 130000 epochs to show the effectiveness of our method. 
It is worth noting that during the training of the PINN, we force the initial value condition and use the LBFGS method for fine-tuning in the last 30000 training epochs.

\begin{figure}[htbp]
\centering
\subfloat[$t=\frac{0.375}{\pi}$]{\includegraphics[width = 0.24\textwidth]{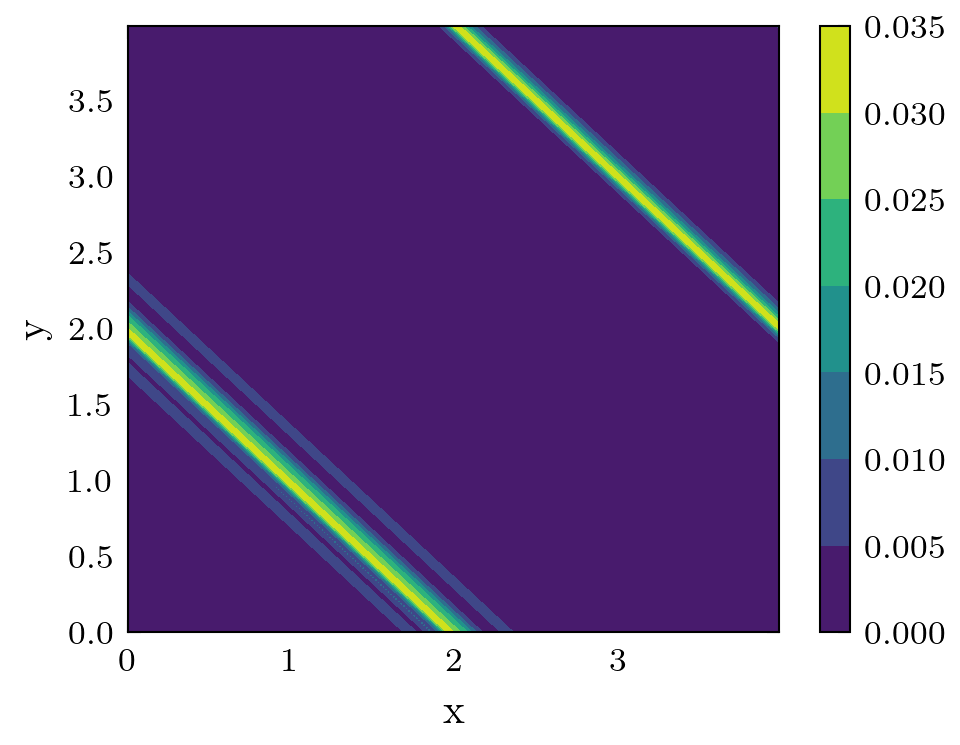}}
\subfloat[$t=\frac{0.75}{\pi}$]{\includegraphics[width = 0.24\textwidth]
{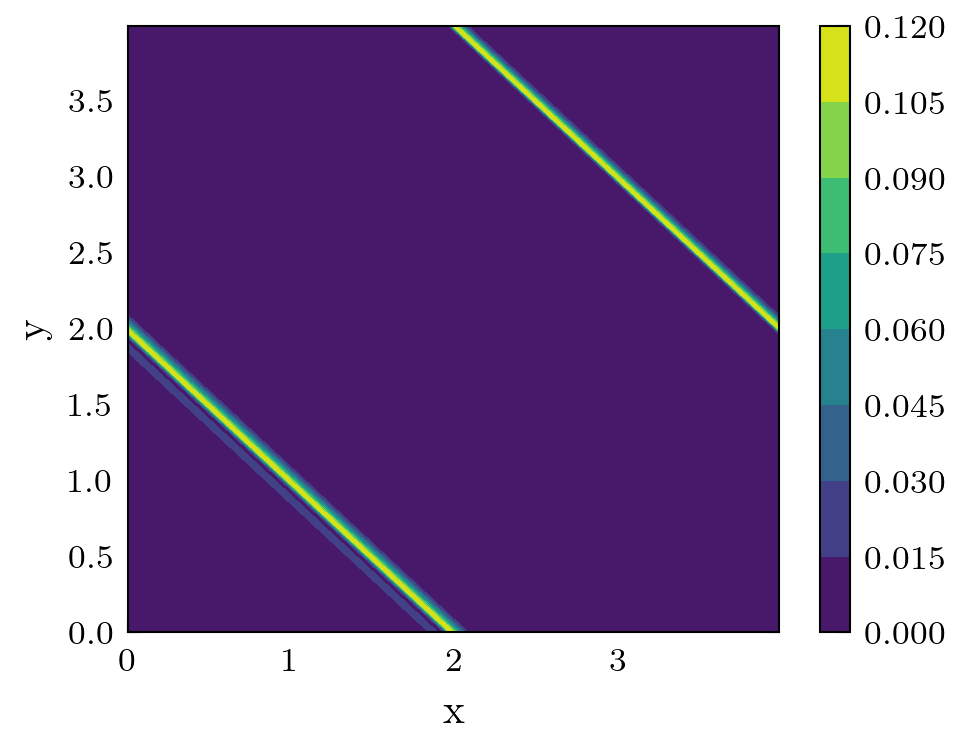}}
\subfloat[$t=\frac{1.125}{\pi}$]{\includegraphics[width = 0.24\textwidth]{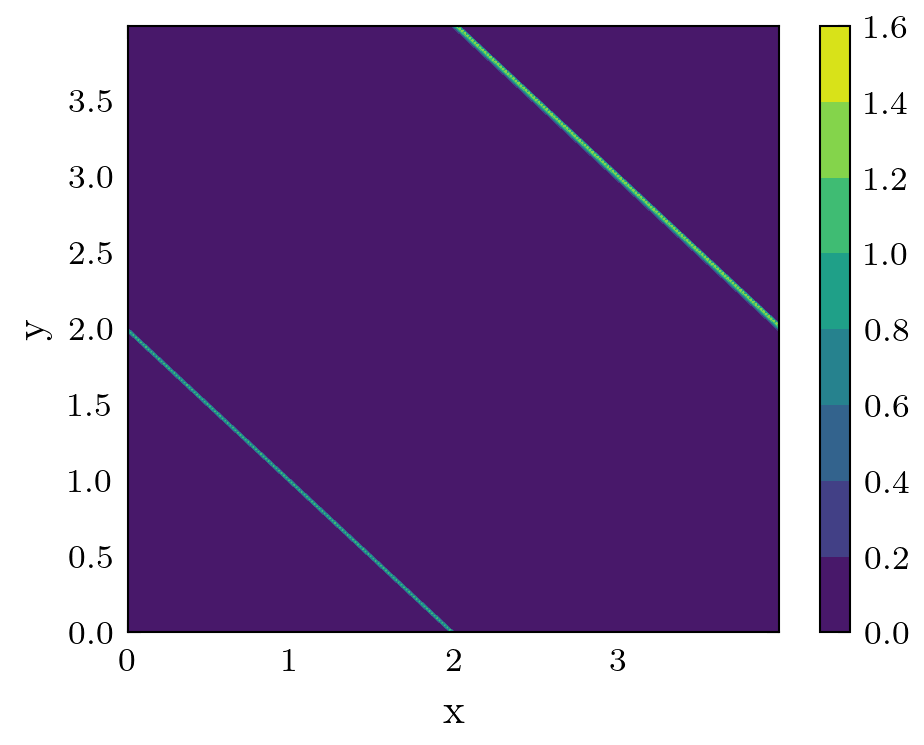}}
\subfloat[$t=\frac{1.5}{\pi}$]{\includegraphics[width = 0.24\textwidth]{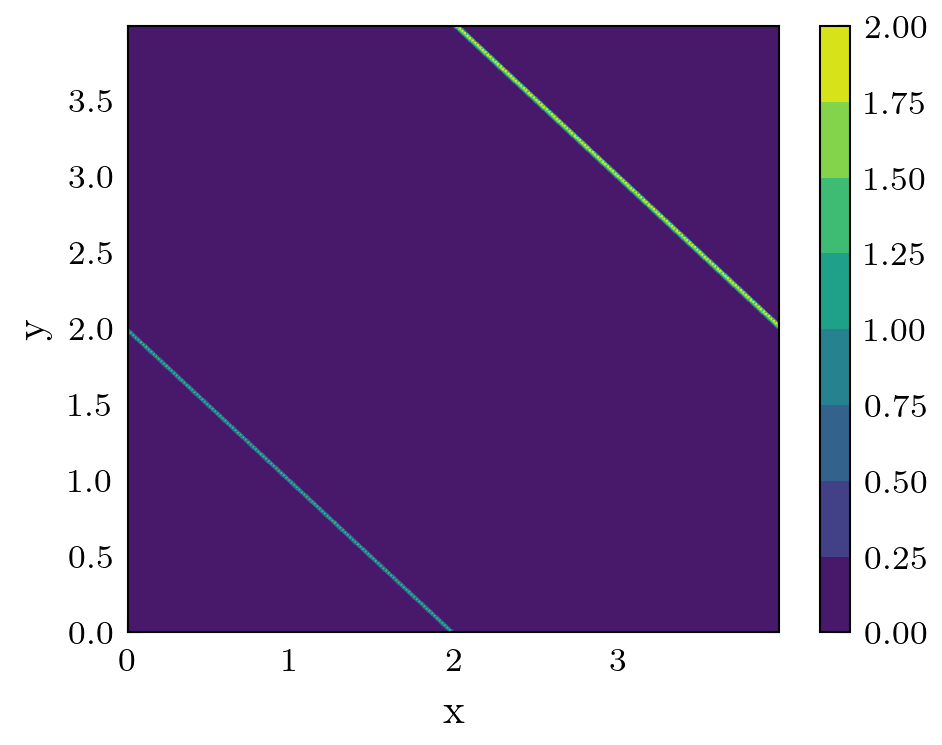}}
\caption{The absolute errors approximated by PINN at different time for the two-dimensional Burgers equation. (a) $t=\frac{0.375}{\pi}$ (b) $t=\frac{0.75}{\pi}$ (c) $t=\frac{1.125}{\pi}$ (d) $t=\frac{1.5}{\pi}$ .}
\label{fig:PINN_Burgers2D}
\end{figure}

\begin{figure}[htbp]
\centering
\subfloat[new distribution of sampling points]{\includegraphics[width = 0.35\textwidth]{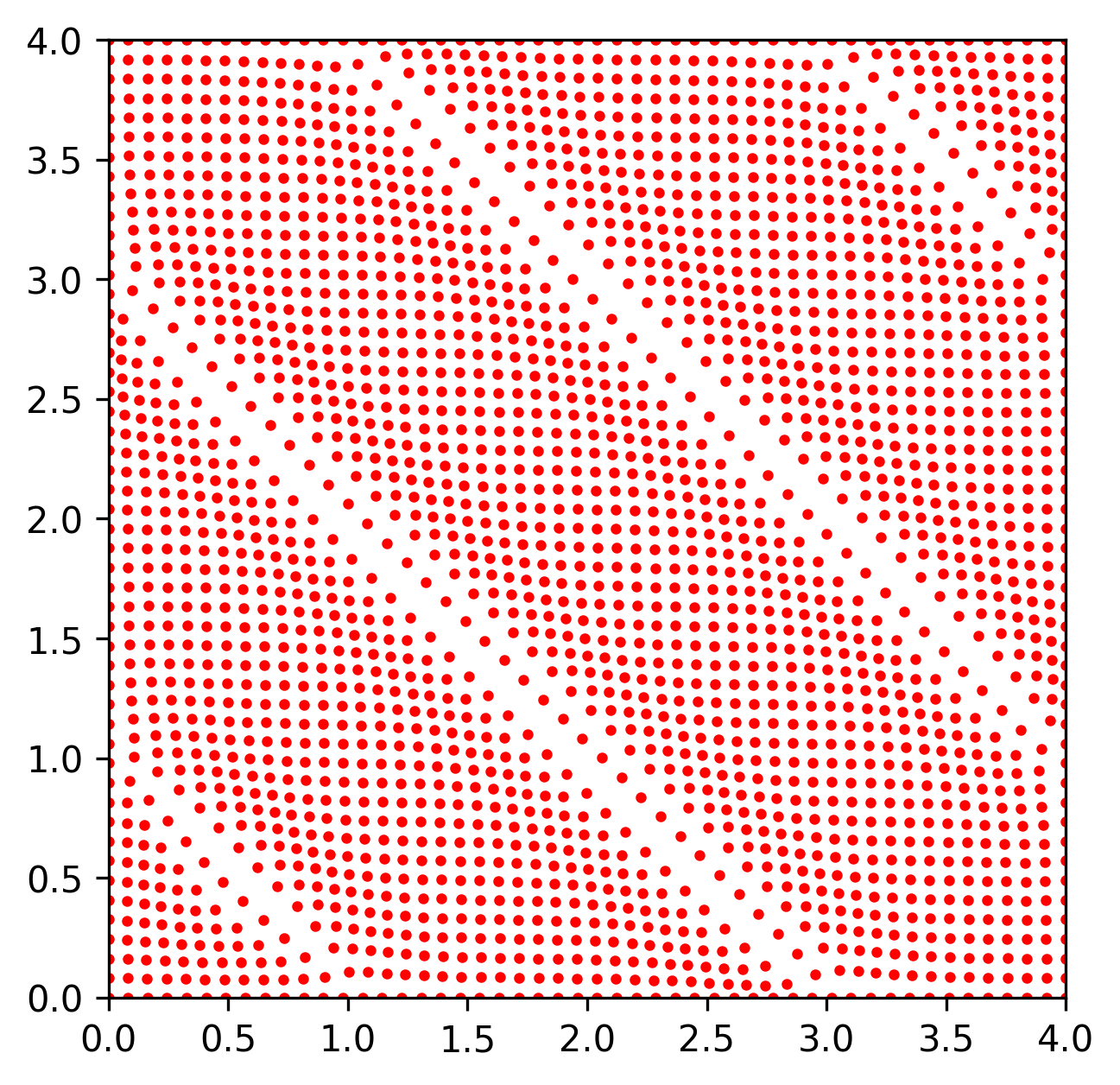}}
\subfloat[performance of the loss function]{\includegraphics[width = 0.43\textwidth]{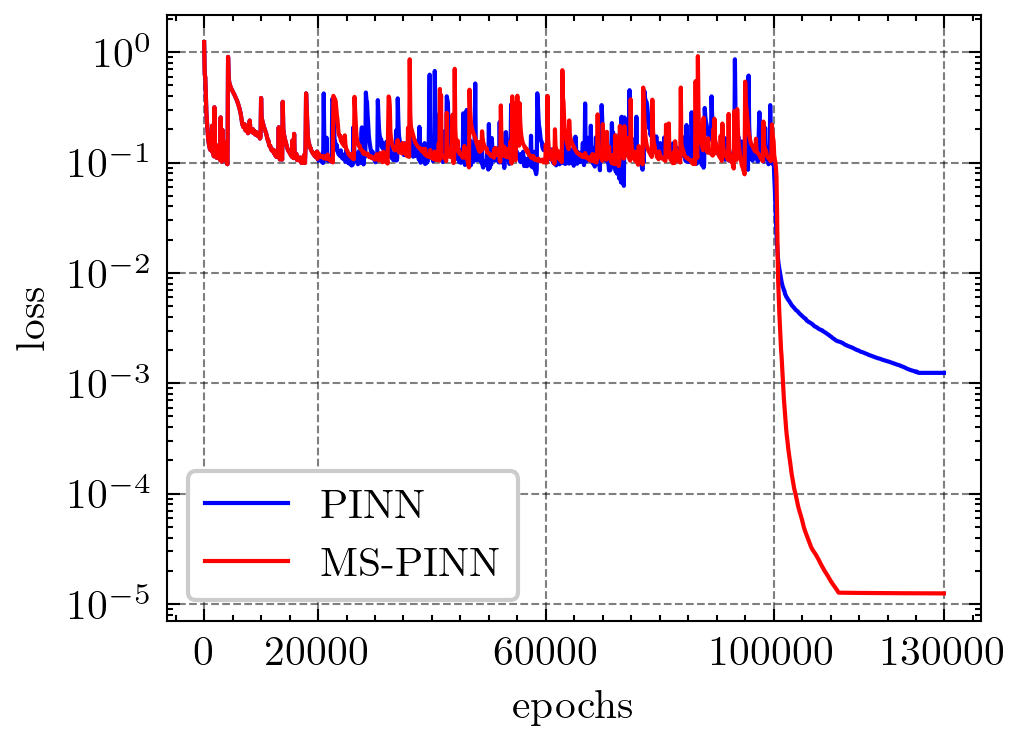}}
\caption{(a) the distribution of new training points obtained by MMPDE-Net; (b) the loss function with different training epochs.}
\label{fig:Newmesh_Burgers2D}
\end{figure}

\begin{figure}[htbp]
\centering
\subfloat[$t=\frac{0.375}{\pi}$]{\includegraphics[width = 0.24\textwidth]{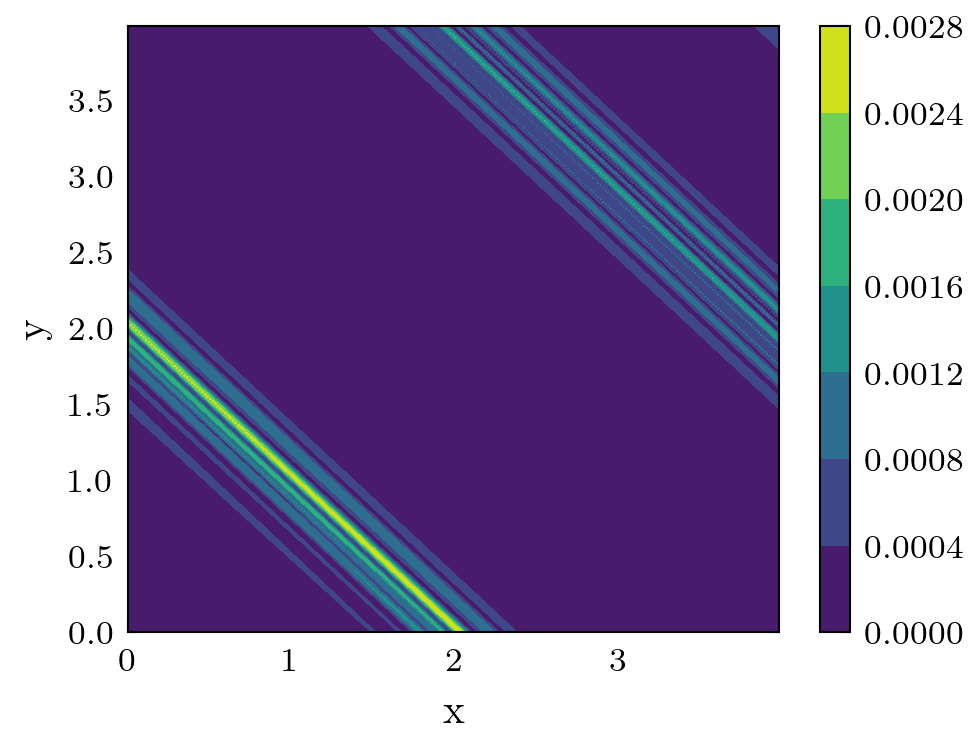}}
\subfloat[$t=\frac{0.75}{\pi}$]{\includegraphics[width = 0.24\textwidth]
{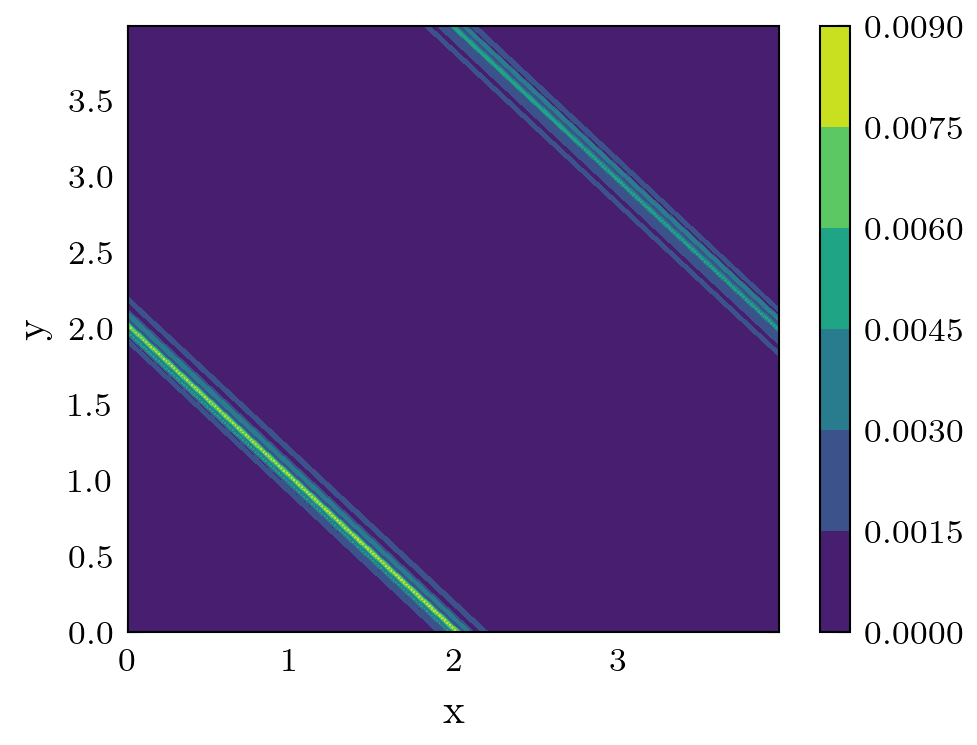}}
\subfloat[$t=\frac{1.125}{\pi}$]{\includegraphics[width = 0.24\textwidth]{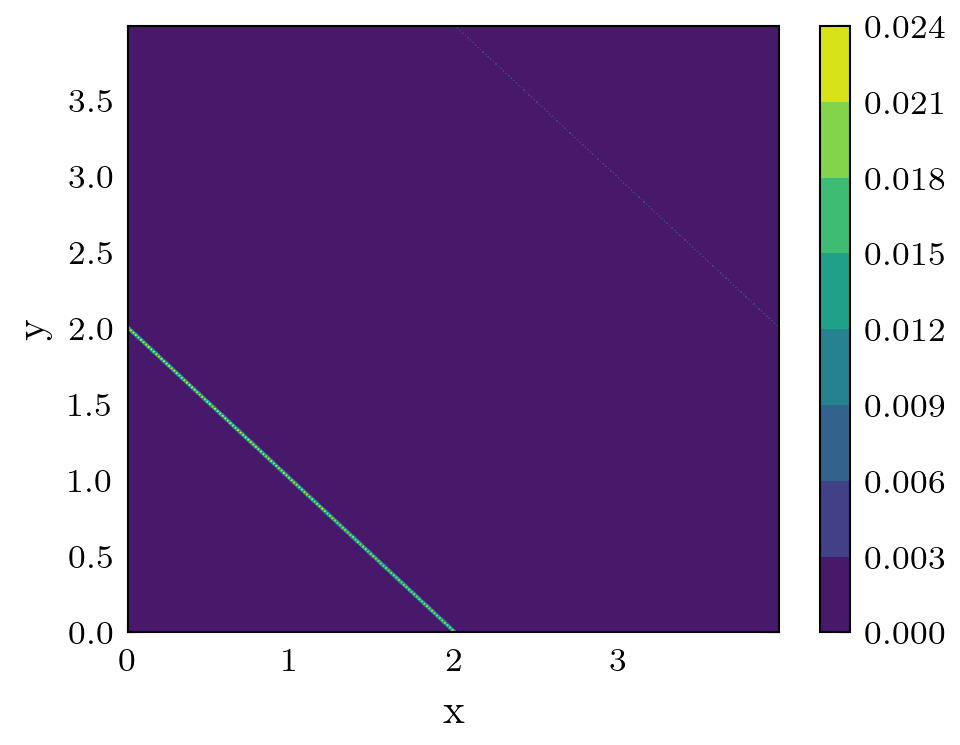}}
\subfloat[$t=\frac{1.5}{\pi}$]{\includegraphics[width = 0.24\textwidth]{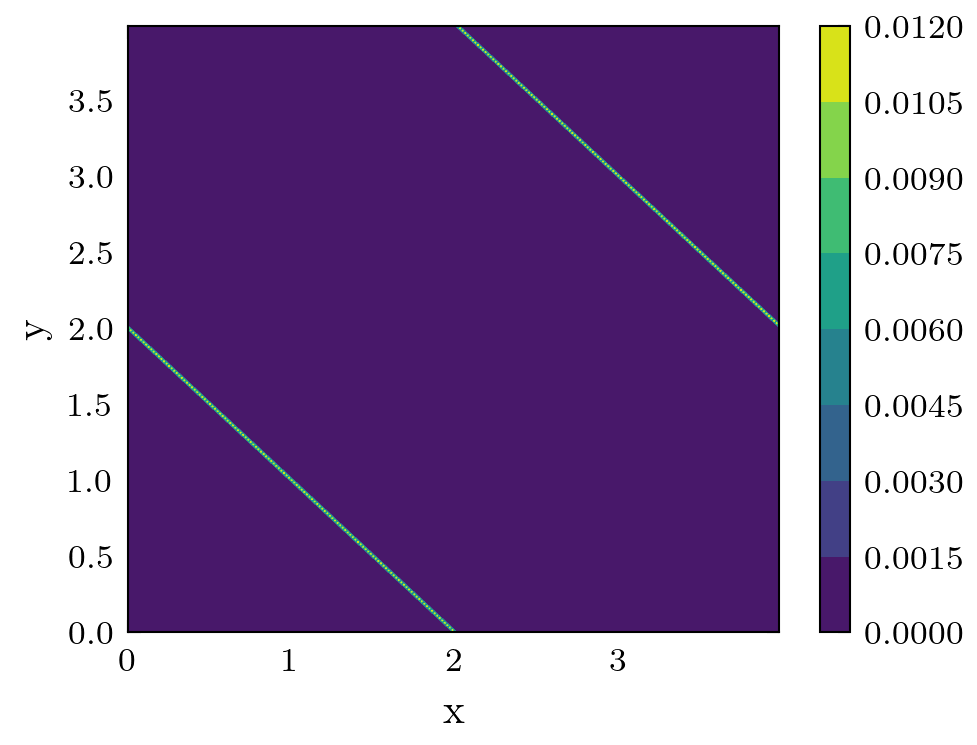}}
\caption{The absolute errors approximated by MS-PINN at different time for the two-dimensional Burgers equation. (a) $t=\frac{0.375}{\pi}$ (b) $t=\frac{0.75}{\pi}$ (c) $t=\frac{1.125}{\pi}$ (d) $t=\frac{1.5}{\pi}$.}
\label{fig:MMPDEPINN_Burgers2D}
\end{figure}

Fig \ref{fig:PINN_Burgers2D} shows the absolute error approximated by PINN at different time. Fig \ref{fig:Newmesh_Burgers2D}(a) shows the output sampling points, which is  performed at $t=0$ and we only show $50 \times 50$ points for clear viewing. Fig \ref{fig:Newmesh_Burgers2D}(b) shows the variation of the loss function for PINN and MS-PINN with different training epochs, where MS-PINN can obtain much lower loss values. Fig \ref{fig:MMPDEPINN_Burgers2D} shows  the absolute error approximated by MS-PINN at different time. For a more detailed view, the relative errors of the approximation solutions of PINN and MS-PINN at different time are shown in Fig \ref{fig:Burgers_2D_Errorline}. Our method can attain better results.

\begin{figure}[htbp]
\centering
\subfloat[$e_\infty(u)$]{\includegraphics[width = 0.35\textwidth]{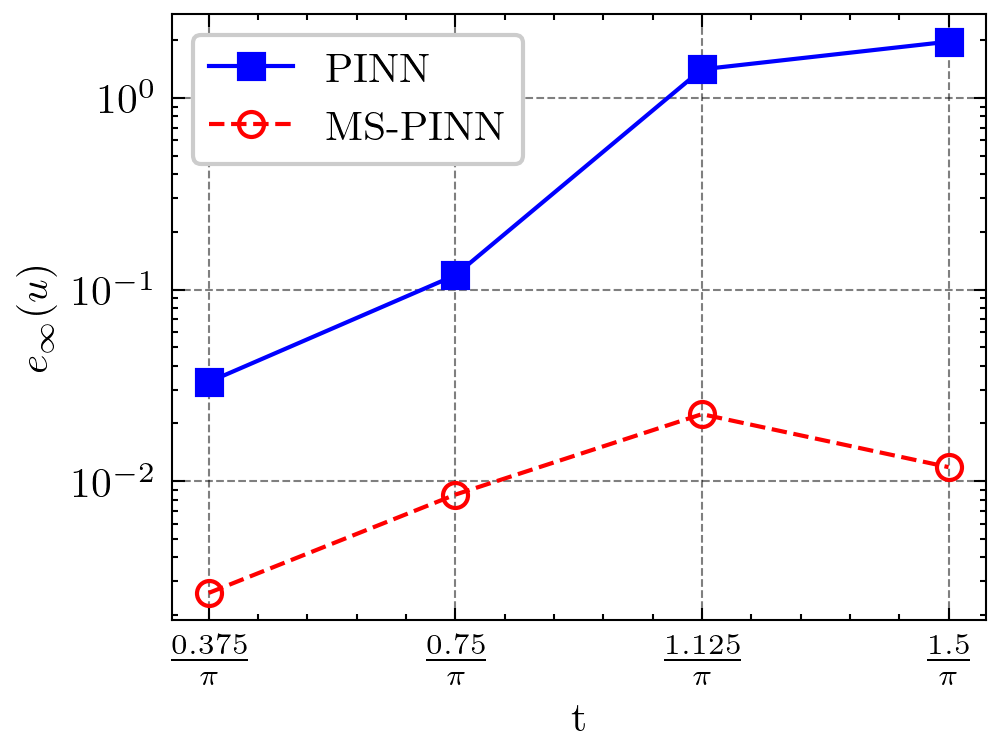}} \quad \quad \quad
\subfloat[$e_2(u)$]{\includegraphics[width = 0.35\textwidth]
{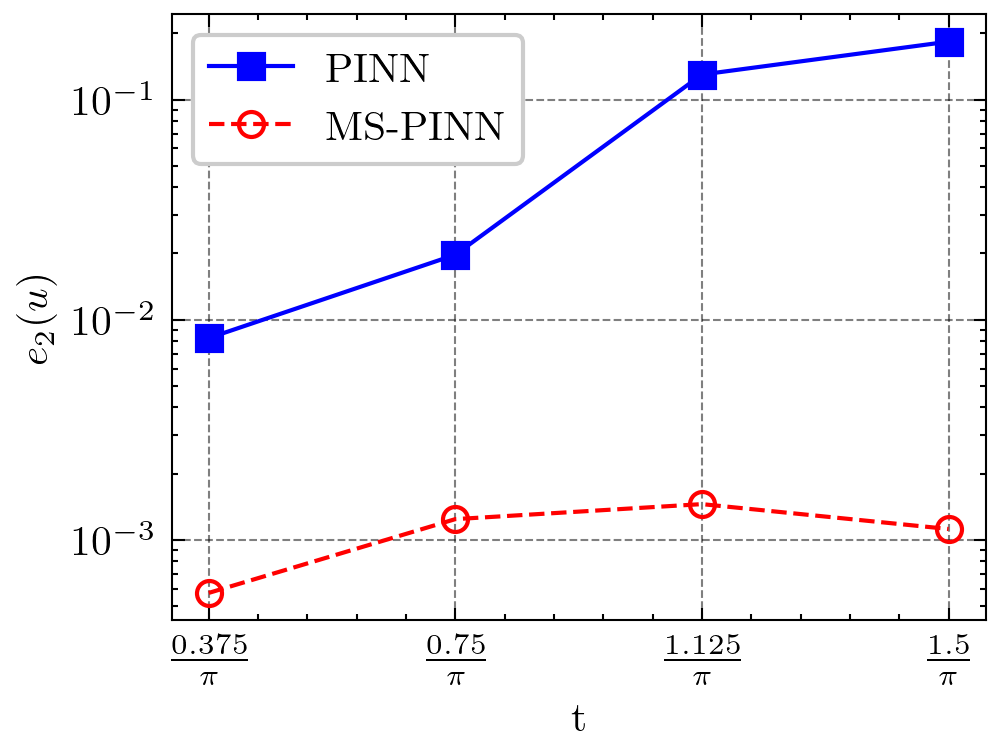}}
\caption{Comparison of the relative errors obtained by PINN and MS-PINN at different time.}
\label{fig:Burgers_2D_Errorline}
\end{figure}

% @@@@@@@@@@@@@@@@@@@@@@@@@@
\section{Conclusions}
\label{sec:conclusion}

In this work, unlike the usual sampling methods based on the loss function perspective, we propose the neural networks MMPDE-Net, which based on the moving mesh PDE method. It is an end-to-end, deep learning solver-independent framework which is designed to implement adaptive sampling. We also develop an iterative algorithm based on MMPDE-Net that is used to control the distribution of sampling points more precisely. The Moving Sampling PINN (MS-PINN) is proposed by combing MMPDE-Net and PINN. Some error estimates are given and sufficient numerical experiments are shown to demonstrate the effectiveness of our methods.

 In the future, we will work on how to develop MMPDE-Net in the case of high-dimensional PDEs,  how to combine MMPDE-Net with other deep learning solvers to improve the efficiency and exquisite convergence analysis of our MMPDE-Net.

% @@@@@@@@@@@@@@@@@@@@@@@@@@
%\section*{Contribution statement}
% \noindent Yu Yang: Methodology, Coding,  Writing \& Editing. Helin Gong: Conceptualization, Methodology, Nuclear engineering data curation, Writing \& Editing, Review, Funding acquisition. Qiaolin He: Conceptualization, Methodology, Review, Funding acquisition. Qihong Yang: Coding. Yangtao Deng: Supervision, Review. Shiquan Zhang: Conceptualization, Review, Funding acquisition.   

%Yu Yang, Helin Gong*, Shiquan Zhang*, Qihong Yang, Zhang Chen, Qiaolin He, Qing Li

% @@@@@@@@@@@@@@@@@@@@@@@@@@
\section*{Acknowledgment}

This research is partially sponsored by the National key R \& D Program of China (No.2022YFE03040002) and the National Natural Science Foundation of China (No.11971020, No.12371434). 

\section*{Data Availability}
 The datasets generated during and/or analysed during the current study are available from the corresponding author on reasonable request.

\textcolor{black}{The codes related MMPDE-Net are publicly available on GitHub at}

\textcolor{black}{https://github.com/YangYuSCU/MMPDE-Net.}

\newpage
% \fi

\section*{Appendix}
\label{sec:appendix}

\renewcommand{\proofname}{\textbf{Proof of Theorem \ref{thm:1}}}

\begin{proof}
Using Definition \ref{de:F} and Assumption \ref{asu:F}, it is obtained that

\begin{equation} \label{eq:thm1_r_mr_u}
    \begin{aligned}
    \Vert  r(\bx;\hat{\theta}) \Vert_{2,U,M_{r}}^2 
    &= \mathcal{F}\left(k_1,\frac{b_1}{2} M_r,M_r\right) + \mathcal{F}\left(k_2,\frac{b_2}{2} M_r,M_r\right) \\
    & = \frac{1}{b_1^q}\left(\frac{k_1}{k_2}\right)^p \mathcal{F}\left(k_2,\frac{1}{2} M_r,M_r\right) + \frac{1}{b_2^q}\mathcal{F}\left(k_2,\frac{1}{2} M_r,M_r\right)\\
    & = \left[\frac{1}{b_1^q}\left(\frac{k_1}{k_2}\right)^p+\frac{1}{b_2^q}\right] \mathcal{F}\left(k_2,\frac{1}{2} M_r,M_r\right).
    \end{aligned}
\end{equation}

\begin{equation} \label{eq:thm1_r_mr_rho}
    \begin{aligned}
    \Vert  r(\bx;\hat{\theta}) \Vert_{2,\rho_{MM},M_{r}}^2 
    &= \mathcal{F}\left(k_1,\frac{a_1}{2}M_r,M_r\right) + \mathcal{F}\left(k_2,\frac{a_2}{2}M_r,M_r\right) \\
    & =\frac{1}{a_1^q} \left(\frac{k_1}{k_2}\right)^p \mathcal{F} \left(k_2,\frac{1}{2} M_r,M_r\right) + \frac{1}{a_2^q} \mathcal{F}\left(k_2,\frac{1}{2} M_r,M_r\right)\\
    & = \left[\frac{1}{a_1^q} \left(\frac{k_1}{k_2}\right)^p+\frac{1}{a_2^q}\right] \mathcal{F}\left(k_2,\frac{1}{2} M_r,M_r\right).
    \end{aligned}
\end{equation}
Since $\mathcal{F}\left(k_2,\frac{1}{2} M_r,M_r\right) > 0$, it is only necessary to prove that 
\begin{equation} \label{eq:thm1_a1a2_b1b2}
    \begin{aligned}
    \frac{1}{a_1^q} \left(\frac{k_1}{k_2}\right)^p+\frac{1}{a_2^q} \leq \frac{1}{b_1^q}\left(\frac{k_1}{k_2}\right)^p+\frac{1}{b_2^q}.
    \end{aligned}
\end{equation}
%\frac{a_1^q(\frac{a_2^q}{b_2^q}-1)}{a_2^q(1-\frac{a_1^q}{b_1^q})} \geq (\frac{k_1}{k_2})^p

Since $\vert \Omega_1 \vert  > \vert \Omega_2 \vert$, we have $ 0< b_2 < 1 < b_1 < 2$. As we described in Section \ref{sec:MMPDE-Net}, the points output by MMPDE-Net will be concentrated in $\Omega_2$, due to $k_2 > k_1$. Therefore, we claim $0< a_1 < b_1< 2, \  0< b_2 < a_2 < 2$. Then, Eq \eqref{eq:thm1_a1a2_b1b2} can be written as 
\begin{equation} \label{eq:thm1_a1a2_b1b2_2}
    \begin{aligned}
    &\frac{1}{a_1^q} \left(\frac{k_1}{k_2}\right)^p+\frac{1}{a_2^q} \leq \frac{1}{b_1^q}\left(\frac{k_1}{k_2}\right)^p+\frac{1}{b_2^q}
    \iff & \frac{a_1^q \left(\frac{a_2^q}{b_2^q}-1\right)}{a_2^q \left(1-\frac{a_1^q}{b_1^q}\right)} \geq \left(\frac{k_1}{k_2}\right)^p.
    \end{aligned}
\end{equation}
We take $a_1$ as the independent variable and introduce a continuous function $h(x)$ on $(0,b_1)$
\begin{equation} \label{eq:thm1_con_fun_b1b2}
 h(x) = \frac{x^q\left(\frac{(2-x)^q}{b_2^q}-1\right)}{\left(2-x\right)^q\left(1-\frac{x^q}{b_1^q}\right)}. 
\end{equation}
We have 
\begin{equation} \label{eq:thm1_con_fun_lim_b1b2}
    \left\{
    \begin{aligned}
    h'(x) &  \geq 0, \\
    \lim\limits_{x\rightarrow 0} h(x) &= 0,\\
    \lim\limits_{x\rightarrow b_1} h(x) &= \left(\frac{b_1}{b_2}\right)^q.\\    
    \end{aligned}
    \right.
\end{equation}
Thus, $\exists \ x^* \in (0,b_1)$,  s.t. $h(x^*) \geq \left(\frac{k_1}{k_2}\right)^q$, since $\left(\frac{b_1}{b_2}\right)^q > 1 >\left(\frac{k_1}{k_2}\right)^q$.  Using Assumption \ref{asu:rho} and the description in Section \ref{sec:MMPDE-Net-Iterations}, choosing an appropriate monitor function and iterative process, the sampling points can be arbitrarily concentrated in $\Omega_2$, that is, the sampling points can be arbitrarily sparse in $\Omega_1$. Therefore, %let $a_1=x^*$, which means that 
%our algorithm can find 
$a_1$ and $a_2$ satisfying Eq \eqref{eq:thm1_a1a2_b1b2} can be found by our algorithm.

\end{proof}
\bigskip

% \begin{fontsize}{20pt}{20pt}
% \textbf{Proof of Theorem \ref{thm:2}}
% \end{fontsize}

% Suppose that Assumption \ref{asu:Normrelations} and Assumption \ref{asu:r-boundary} hold, let $\bu^*$ be the exact solution of Eq \eqref{eq:gen_PDE}. Given $\delta > 0$, $\{x_i\}_{i=1}^{M_r}$ is a collection of i.i.d. random samples from the PDF $\rho_{MM}$
% and $\forall \bu(\bx;\theta) \in \mathcal{F}_r$, the following error estimate holds

% \begin{equation*}
%     \Vert \bu^*(\bx)-\bu(\bx;\theta) \Vert_{2,\rho_{MM}} \leq  \frac{\sqrt{2}}{C_1} \left(\Vert  r(\bx;\theta) \Vert_{2,\rho_{MM},M_r} ^2 + 
%      \Vert b(\bx;\theta)\Vert_{2,\rho_{MM},\partial\Omega}^2 + 2\mathcal{R}_{M_r}(\mathcal{F}_r) + \delta \right)^{\frac{1}{2}},
% \end{equation*}
% with probability at least $1-\exp\left(-\frac{M_r\delta^2}{2b^2}\right)$.
\renewcommand{\proofname}{\textbf{Proof of Theorem \ref{thm:2}}}
\begin{proof}
By  Assumption \ref{asu:Normrelations}, we have
\begin{equation}
     \begin{aligned}
     C_1 \Vert \bu^*(\bx)-\bu(\bx;\theta) \Vert_{2,\rho_{MM}}
     &\leq \Vert \mathcal{A}(\bu^*(\bx))-\mathcal{A}(\bu(\bx;\theta)) \Vert_{2,\rho_{MM},\Omega}  + \Vert \mathcal{B}(\bu^*(\bx))-\mathcal{B}(\bu(\bx;\theta)) \Vert_{2,\rho_{MM},\partial\Omega} \\
     &= \Vert \mathcal{A}(\bu(\bx;\theta)) - f \Vert_{2,\rho_{MM},\Omega} + \Vert \mathcal{B}(\bu(\bx;\theta)) - g \Vert_{2,\rho_{MM},\partial\Omega} \\
     &= \Vert r(\bx;\theta) \Vert_{2,\rho_{MM},\Omega} + \Vert b(\bx;\theta)\Vert_{2,\rho_{MM},\partial\Omega}\\
     &\leq  \sqrt{2}  \left(\Vert r(\bx;\theta) \Vert_{2,\rho_{MM},\Omega}^2 + 
     \Vert b(\bx;\theta)\Vert_{2,\rho_{MM},\partial\Omega}^2\right)^{\frac{1}{2}}.
     \end{aligned}     
\end{equation}
Using Lemma \ref{lem:Uniform laws of large numbers-Rademacher}, we have

\begin{equation}
     \Vert r(\bx;\theta) \Vert_{2,\rho_{MM},\Omega}^2 \leq \Vert  r(\bx;\theta) \Vert_{2,\rho_{MM},M_r} ^2+2\mathcal{R}_{M_r}(\mathcal{F}_r) + \delta,
\end{equation}
with probability at least $1-\exp\left(-\frac{M_r\delta^2}{2b^2}\right)$. Therefore,
\begin{equation}
    \begin{aligned}
     \Vert \bu^*(\bx)-\bu(\bx;\theta) \Vert_{2,\rho_{MM}} &\leq  \frac{\sqrt{2}}{C_1}(\Vert r(\bx;\theta) \Vert_{2,\rho_{MM},\Omega}^2 + 
     \Vert b(\bx;\theta)\Vert_{2,\rho_{MM},\partial\Omega}^2)^{\frac{1}{2}}\\
     &\leq \frac{\sqrt{2}}{C_1} \left(\Vert  r(\bx;\theta) \Vert_{2,\rho_{MM},M_r} ^2 + 
     \Vert b(\bx;\theta)\Vert_{2,\rho_{MM},\partial\Omega}^2 + 2\mathcal{R}_{M_r}(\mathcal{F}_r) + \delta \right)^{\frac{1}{2}},
     \end{aligned} 
\end{equation}
 with probability at least $1-\exp\left(-\frac{M_r\delta^2}{2b^2}\right)$.
\end{proof}
\bigskip

% \begin{fontsize}{20pt}{20pt}
% \textbf{Proof of Corollary \ref{cor:1}}
% \end{fontsize}

%     Suppose that Theorem \ref{thm:1} and Theorem \ref{thm:2} hold, then 
%     \begin{equation*}
%     \begin{aligned}
%      \Vert \bu^*(\bx)-\bu(\bx;\hat{\theta}) \Vert_{2,\rho_{MM}}
%      &\leq \frac{\sqrt{2}}{C_1} \left(\Vert  r(\bx;\hat{\theta}) \Vert_{2,\rho_{MM},M_r} ^2 + 
%      \Vert b(\bx;\hat{\theta})\Vert_{2,\rho_{MM},\partial\Omega}^2 + 2\mathcal{R}_{M_r}(\mathcal{F}_r) + \delta \right)^{\frac{1}{2}} \\
%      &\leq \frac{\sqrt{2}}{C_1} \left(\Vert  r(\bx;\hat{\theta}) \Vert_{2,U,M_r} ^2 + 
%      \Vert b(\bx;\hat{\theta})\Vert_{2,U,\partial\Omega}^2 + 2\mathcal{R}_{M_r}(\mathcal{F}_r) + \delta \right)^{\frac{1}{2}},
%      \end{aligned} 
% \end{equation*}
%     with probability at least $1-\exp\left(-\frac{M_r\delta^2}{2b^2}\right)$.
\renewcommand{\proofname}{\textbf{Proof of Corollary \ref{cor:1}}}
\begin{proof}
By  Theorem \ref{thm:1}, we have
\begin{equation}
  \Vert  r(\bx;\hat{\theta}) \Vert_{2,\rho_{MM},M_r} ^2 \leq \Vert  r(\bx;\hat{\theta}) \Vert_{2,U,M_r} ^2.
\end{equation}
%Since MMPDE-Net only changes the distribution of points within $\Omega$, the distribution of points on $\partial\Omega$ does not change.
Since only the distribution of points within $\Omega$ varies after training by MMPDE-Net and the distribution of points on $\partial\Omega$ is unchanged, we have

\begin{equation}
  \Vert  b(\bx;\hat{\theta}) \Vert_{2,\rho_{MM},\partial\Omega} ^2 = \Vert  b(\bx;\hat{\theta}) \Vert_{2,U,\partial\Omega} ^2.
\end{equation}
Therefore,
    \begin{equation*}
    \begin{aligned}
     \Vert \bu^*(\bx)-\bu(\bx;\hat{\theta}) \Vert_{2,\rho_{MM}}
     &\leq \frac{\sqrt{2}}{C_1} \left(\Vert  r(\bx;\hat{\theta}) \Vert_{2,\rho_{MM},M_r} ^2 + 
     \Vert b(\bx;\hat{\theta})\Vert_{2,\rho_{MM},\partial\Omega}^2 + 2\mathcal{R}_{M_r}(\mathcal{F}_r) + \delta \right)^{\frac{1}{2}} \\
     &\leq \frac{\sqrt{2}}{C_1} \left(\Vert  r(\bx;\hat{\theta}) \Vert_{2,U,M_r} ^2 + 
     \Vert b(\bx;\hat{\theta})\Vert_{2,U,\partial\Omega}^2 + 2\mathcal{R}_{M_r}(\mathcal{F}_r) + \delta \right)^{\frac{1}{2}},
     \end{aligned} 
    \end{equation*}
 with probability at least $1-\exp\left(-\frac{M_r\delta^2}{2b^2}\right)$.
\end{proof}
\vfill

\bibliographystyle{cas-model2-names}
\bibliography{MS-PINN}
% \begin{thebibliography}{99}  
% \bibitem{renardy2006introduction}Renardy, Michael, and Robert C. Rogers. An introduction to partial differential equations. Vol. 13. Springer Science \& Business Media, 2004.
% \end{thebibliography}

\end{document}